\documentclass[12pt]{amsart}
\usepackage{tikz}
\usepackage{longtable}
\usepackage[normalem]{ulem}
\usepackage{amsmath,amsfonts,amsthm,amssymb,amscd, verbatim, graphicx,textcomp, hyperref}
\usepackage{upgreek}
\usepackage{wasysym}
\usepackage[margin=2cm]{geometry}
\usepackage{pinlabel}
\usepackage{pdflscape}
\usepackage{tabularx}
\usepackage{latexsym}
\usepackage{multirow}
\usepackage{mathtools}
\usepackage{float}
\usepackage{wasysym}
 \usepackage{color}

\newtheorem{Thm}{Theorem}[section]

\newtheorem{Cor}[Thm]{Corollary}
\newtheorem{Conj}{Conjecture}
\newtheorem{Prop}[Thm]{Proposition}

\newtheorem{Lem}[Thm]{Lemma}

\newtheorem*{thma}{Main Result}

\theoremstyle{definition}
\newtheorem{Def}[Thm]{Definition}

\newtheorem{Ex}[Thm]{Example}
\newtheorem{Rem}[Thm]{Remark}
\theoremstyle{remark}

\numberwithin{equation}{section}

\newcommand{\Aut}{\operatorname{Aut}}

\newcommand{\Hom}{\operatorname{Hom}}
\newcommand{\McL}{\operatorname{McL}}
\newcommand{\Mor}{\operatorname{Mor}}
\newcommand{\Inj}{\operatorname{Inj}}

\newcommand{\Id}{\operatorname{Id}}

\newcommand{\Syl}{\operatorname{Syl}}

\newcommand{\Out}{\operatorname{Out}}
\newcommand{\Dih}{\operatorname{Dih}}
\newcommand{\SDih}{\operatorname{SDih}}

\newcommand{\Inn}{\operatorname{Inn}}

\newcommand{\Ly}{\operatorname{Ly}}

\def \a {\alpha}

\newcommand{\D}{\mathcal{D}}

\newcommand{\x}{\textbf{x}}

\newcommand{\Sym}{\operatorname{Sym}}

\newcommand{\Alt}{\operatorname{Alt}}
\newcommand{\M}{\operatorname{M}}
\newcommand{\HN}{\operatorname{HN}}

\newcommand{\PSU}{\operatorname{PSU}}

\newcommand{\Gee}{\operatorname{G}}

\newcommand{\Sp}{\operatorname{Sp}}

\newcommand{\PSL}{\operatorname{PSL}}
\newcommand{\PGL}{\operatorname{PGL}}
\newcommand{\GL}{\operatorname{GL}}

\newcommand{\SL}{\operatorname{SL}}

\newcommand{\Suz}{\operatorname{Suz}}
\newcommand{\SU}{\operatorname{SU}}
\newcommand{\Th}{\operatorname{Th}}
\newcommand{\PSp}{\operatorname{PSp}}
\newcommand{\Co}{\operatorname{Co}}

\newcommand{\R}{\mathcal{R}}  % The real numbers.
\renewcommand{\epsilon}{\varepsilon}
\renewcommand{\bar}{\overline}

\renewcommand{\S}{\mathcal{S} }

\newcommand{\Q}{\mathcal{Q} }

\newcommand{\sgb}{\operatorname{\mathbf {SmallGroup}}}

\newcommand{\G}{\mathcal{G}}
\newcommand{\U}{\mathcal{U}}

\newcommand{\F}{\mathcal{F}}

\def \ov {\overline}
\newcommand{\T}{\mathcal{T}}

\renewcommand{\P}{\mathcal{P}}
\newcommand{\A}{\mathcal{A}}
\newcommand{\E}{\mathcal{E}}

\renewcommand{\a}{\mathbf{a}}

\def \Frob {\mathrm{Frob}}

\renewcommand \G {\mathcal G}
\begin{document}

\title{Algorithms for fusion systems with applications to $p$-groups of small order}

\author{Chris Parker}
\address{
School of Mathematics\\
University of Birmingham\\
Edgbaston\\
Birmingham B15 2TT\\ United Kingdom
} \email{c.w.parker@bham.ac.uk}

\author{Jason Semeraro}
\address{Heilbronn Institute for Mathematical Research, School of Mathematics, University of Leicester,  United Kingdom}
\email{jpgs1@leicester.ac.uk}

\begin{abstract}
For a prime $p$, we describe a protocol for handling a specific type of fusion system on a $p$-group  by computer. These fusion systems contain all saturated fusion systems. This framework allows us to computationally determine whether or not two subgroups are conjugate in the fusion system for example. We describe a generation procedure for automizers of every subgroup of the $p$-group. This allows a computational check of saturation. These procedures have been implemented using {\sc Magma} \cite{Magma}. We  describe a computer  program which searches for saturated fusion systems $\F$ on $p$-groups with  $O_p(\F)=1$ and $O^p(\F)=\F$. Employing these computational methods we  determine all such fusion systems on groups of order $p^n$ where $(p,n) \in \{(3,4),(3,5),(3,6),(3,7),(5,4),(5,5),(5,6),(7,4),(7,5)\}$. This gives the first  complete  picture of which groups can support saturated fusion systems with $O_p(\F)=1$ and $O^p(\F)=\F$ on small $p$-groups of odd order.

\end{abstract}


\maketitle
\section{Introduction}

Our primary goal is to address questions raised in \cite[Open Problem 4, page 217]{AschbacherKessarOliver2011} and so to develop a deeper understanding of saturated fusion systems on $p$-groups of odd order. With this in mind, we give a computational framework in which to perform calculations with a certain special  type of fusion system which contains the class of all saturated fusion systems. In particular, we describe an implementation of algorithms which check saturation of such fusion systems. Our main contribution  is an explicit description of  an algorithm to determine all saturated fusion systems $\F$ on a given $p$-group for any prime $p$.  Exploiting  the library of small groups available in the  {\sc Magma} system \cite{Magma},  we use our implementation to    create a complete list of  saturated fusion systems $\F$ with $O_p(\F)=1$ and $\F=O^p(\F)$ on small $p$-groups of odd order.

\begin{thma}\label{Thm1} For all  $(p,n) \in \{(3,4),(3,5),(3,6),(3,7),(5,4),(5,5),(5,6),(7,4),(7,5)\}$, all saturated fusion systems $\F$ with $O_p(\F)=1$ and $O^p(\F)=\F$ on  $p$-groups of order $p^n$ are determined and listed in Appendix \ref{s:appendix}.
\end{thma}

Further details of the saturated fusion systems which appear are given in Theorems~\ref{Thmp^4}, \ref{Thmp^5}, \ref{Thmp^6} and \ref{Thmp^7}, and a forensic look at each of them is provided in Appendix~\ref{s:appendix}. An analogous result for $2$-groups of order at most $2^9$ was obtained recently by Andersen, Oliver and Ventura in \cite{AOV2017}. We recover their classification for groups of order at most $2^8$ via an automated approach (see Theorem \ref{2^8}).  In Section~\ref{s:conclude},  we present a number of general results and conjectures motivated by  examining the fusion systems that our programs output. This begins to address the questions raised by Oliver as described above.

 As recalled in Section \ref{s:background}, a saturated fusion system $\F$ on a finite $p$-group $S$ is a category whose objects are the subgroups of $S$, and whose morphisms are group homomorphisms which are indistinguishable from those coming from conjugacy relations in a finite group in which $S$ is a Sylow $p$-subgroup. We wish to encode fusion systems, and our method for doing so begins with the Alperin-Goldschmidt Theorem (Theorem  \ref{t:alp}) which asserts that a saturated fusion system $\F$ is completely determined by $\Aut_\F(S)$ and the $\F$-automorphisms of certain so-called $\F$-essential subgroups. For this reason, we propose that  a certain family of fusion systems should be computationally understood in terms of a \textit{fusion datum} on $S$ which is simply a collection $\Q$ of subgroups of $S$ (including $S$ itself), together with a map $\A$ which assigns a group of automorphisms to each member of $\Q$. To a fusion datum $\D=(\Q,\A)$ on $S$ we associate the unique minimal fusion system $\F=\F(\D)$ on $S$ in which $\A(Q) \le \Aut_\F(Q)$ for all $Q \in \Q$.  In order to calculate with $\F$ algorithmically, we associate a combinatorial object $\Gamma(\D)$ to $\D$ called its \emph{fusion graph}. Our first main result, Theorem \ref{t:main}, demonstrates that we can use the fusion graph to read off  whether subgroups of $S$ are $\F$-conjugate  and to calculate  $\Aut_\F(P)$ for each $P \le S$. The proof of Theorem \ref{t:main} is constructive so that the calculation of $\Aut_\F(P)$ from $\D$ can be implemented. In particular, we can calculate $\Aut_\F(Q)$, and the condition $\A(Q)=\Aut_\F(Q)$ forms part of the definition of an \textit{automizer sequence} which is a fusion datum in which elements of $\Q$ resemble essential subgroups in a saturated fusion system. More precisely, by the Alperin-Goldschmidt Theorem, every saturated fusion system determines an automizer sequence in which $\Q$ is a set of $\F$-conjugacy class representatives  of $\F$-essential subgroups. All that remains, therefore, is to determine which automizer sequences give rise to saturated fusion systems. We achieve this by checking the surjectivity property holds for fully $\F$-normalized $\F$-conjugacy class representatives of $\F$-centric subgroups.

In Section \ref{s:alg} we describe our algorithm to determine saturated fusion systems on a small group $S$. We begin by identifying various group theoretic properties of $S$ and its subgroups which preclude $S$ from supporting a saturated fusion system $\F$ which has $\F$-essential subgroups. Some of the theorems we cite rely on the classification of finite simple groups. For example, we choose to use the structure of Sylow $p$-subgroups of groups with a strongly $p$-embedded subgroup. Though such results are not strictly necessary for the algorithms, they do speed up our procedures. We apply these results to severely restrict the list of small groups of order $p^n$ which require more intensive computation. For any group $S$ which cannot be ruled out with the initial tests, we produce a list $\P_S$ of $S$-conjugacy class representatives of subgroups of $S$ which are potentially $\F$-essential in some saturated fusion system on $S$. The next step determines the possibilities for $\Out_\F(S)$. Roughly, for each  $\Out(S)$-conjugacy class representative $B_0 \le \Out(S)$ and each subset $\Q$ of $\P_S$ we describe an algorithm to list all automizer sequences $\D=(\Q,\A)$ on $S$ in which $\Out_{\F(\D)}(S)=B_0$. We determine saturation by means of the procedure outlined above.

In Section \ref{s:results} we summarise the results of our computations. In his thesis \cite{moncho2018fusion} Moragues Moncho determines all simple saturated fusion systems on groups of order $p^4$ and our results are consistent with his calculations. When $p \ge 5$ the list of simple saturated fusion systems on groups of order $p^5$ can be classified using results of Grazian \cite{grazian2018fusion} (see Corollary \ref{p^5-abelian subgroup}). We complete this classification with an examination of the case $p=3$.

Many of the examples which appear in the paper are saturated fusion systems on groups of maximal class.  We follow Blackburn's
notation and give a description of these presentations in Appendix \ref{s:appendix}.
This allows us to describe various subgroups explicitly and so as well as using the \texttt{SmallGroup} notation in {\sc Magma} we also identify these groups using Blackburn's notation which carries more structural information. We hope that this will provoke conjectures concerning the groups of maximal class which support saturated fusion systems $\F$ in which $O_p(\F)=1$ and $O^p(\mathcal{F})=\F$. For example, in Conjecture \ref{c:p6} we make a speculation about which Blackburn groups of order $p^6$ support such fusion systems.

In Theorem \ref{Thmp^7} we show that, of the 9310 groups of order $3^7$, remarkably only $8$ support a saturated fusion system $\F$ in which $O^3(\F)=\F$ and $O_3(\F)=1$. Of those, $3$ are direct products and $2$ are Sylow $3$-subgroups of sporadic simple groups.

We complete the introduction with a list of basic functions which have been implemented in {\sc Magma} using the approach outlined above; some of these functions are discussed in more detail in Section \ref{s:auto}.  The package of {\sc Magma} functions can be found  at \cite{PS-Progs-2020}.

Let $p$ be a prime, $S$ be a finite $p$-group, $\F$ and $\G$ be fusion systems on $S$ (understood in terms of fusion data), $P$ and $R$ be subgroups of $S$ and $G$ be a finite group.

\begin{tabular}{rl}
\texttt{IsIsomorphic(F,G)}:& returns \texttt{true} iff $\F$ and $\G$ are isomorphic\\
\texttt{GroupFusionSystem(G,p)}:  & returns the $p$-fusion system associated to the group $G$\\
\texttt{IsConjugate(F,P,R)}:& returns \texttt{true} iff $P$ and $R$ are $\F$-conjugate \\
\texttt{IsFullyNormalized(F,P)}:&returns \texttt{true} iff $P$ is fully $\F$-normalized\\
\texttt{IsFullyAutomized(F,P)}:&returns \texttt{true} iff $P$ is fully $\F$-automized\\
\texttt{IsFullyCentralized(F,P)}:&returns \texttt{true} iff $P$ is fully $\F$-centralized\\
\texttt{IsCentric(F,P)}:&returns \texttt{true} iff $P$ is $\F$-centric\\
\texttt{SurjectivityProperty(F,P)}:&returns \texttt{true} iff $P$ has the surjectivity property in $\F$\\
\texttt{IsSaturated(F)}:&returns \texttt{true} iff $\F$ is saturated \\
\texttt{Core(F)}:& returns $O_p(\F)$, the largest $\F$-normal subgroup of $S$ \\
\texttt{FocalSubgroup(F)}:& returns the focal subgroup $\mathfrak{foc}(\F)$\\
\texttt{IsWeaklyClosed(F,P)}:&returns \texttt{true} iff $P$ is weakly $\F$-closed\\
\texttt{IsStronglyClosed(F,P)}:&returns \texttt{true} iff $P$ is strongly $\F$-closed\\
\texttt{AllFusionSystems(S)}:& lists all saturated $\F$ on $S$ with $O^p(\F)=\F$ and $O_p(\F)=1$\\
\end{tabular}

There are additional parameters which can be added to \texttt{AllFusionSystems(S)} which allow it to calculate saturated fusion systems on $S$ with $O_p(\F) \ne 1$ or $O^p(\F) \ne \F$.

\subsection*{Acknowledgements} We thank David Craven for giving us access to computational resources provided by the  Royal Society. We also thank the referees for a careful reading of our original manuscript, for suggesting  improvements to both the text and the proofs and for pointing out various inaccuracies.

\section{Background}\label{s:background}

Let $p$ be a prime, and $S$ be a finite $p$-group. If $P,Q \le S$ and $G$ is a group with $S \le G$ define $$N_G(P,Q)=\{g \in G \mid P^g \le Q\} \mbox{ and } \Hom_G(P,Q)=\{c_g \mid g \in N_G(P,Q)\},$$ where $c_g$ is the conjugation map induced by $g$ given by $c_g: x \mapsto g^{-1}xg.$ We have $\Hom_G(P,Q) \subseteq \Inj(P,Q)$, where $\Inj(P,Q)$ is  the set of all injective group homomorphisms from $P$ to $Q$. A \textit{fusion system} or \textit{$p$-fusion system} $\F$ on $S$ is a category
whose objects are the set of subgroups of $S$ and,  for all objects $P, Q \le S$, the morphisms from $P$ to $Q$, $\Mor_\F(P,Q)$, have the following two properties
\begin{itemize}
\item[(1)] $\Hom_S(P,Q) \subseteq \Mor_\F(P,Q) \subseteq \Inj(P,Q);$
\item[(2)]  every morphism $\varphi  \in \Mor_\F(P,Q)$ factorizes  as an $\F$-isomorphism $P \rightarrow P\varphi$ followed by an inclusion  $P\varphi \hookrightarrow Q$.
\end{itemize}
 We write $\Hom_\F(P,Q)$ in place of $\Mor_\F(P,Q)$.
 The \textit{universal fusion system} on $S$ is denoted by $\U$ and  defined by setting $\Hom_\U(P,Q)=\Inj(P,Q)$. If $G$ is a group with $S \le G$, we write $\F_S(G) \subseteq \U$ for the fusion system on $S$ defined by $$\Hom_{\F_S(G)}(P,Q) = \Hom_G(P,Q).$$
Moreover, we write $\U^G$ for the set of all subfusion systems of $\U$ containing $\F_S(G)$. By (1) above $\U^S$ is the set of all fusion systems on $S$, and it is a fact due to Park \cite{Park2016} that the map $\F_S(-)$ which sends finite groups containing $S$ to $\U^S$ is surjective. Our framework for calculating fusion systems on a group $T$, first maps $T$ to a standard copy $S$ of $T$ where all the fusion systems on $S$ are constructed and so all fusion systems on $T$ can be transported into  $\U^S$.

If $\F \in \U^S$ and $\theta \in \Aut(S)$, then we define  $\F^\theta$ to be the fusion system in $\U^S$ defined by $$\Hom_{\F^\theta}(P,Q) = \{\theta^{-1}\gamma\theta\mid \gamma\in  \Hom_\F(P\theta^{-1},Q\theta^{-1}) \} .$$ If $\G \in \U^S$, we say that $\F$ is \textit{isomorphic} to $\G$ and write $\F \cong \G$ if there exists $\theta \in \Aut(S)$ such that $\G=\F^\theta$. Thus $\Aut(S)$ partitions $\U^S$ into \textit{isomorphism classes} of fusion systems.

We next consider generation of fusion systems. Following \cite[Definition I.3.4(b)]{AschbacherKessarOliver2011}, if $\mathcal{X}$ is a set of monomorphisms between subgroups of $S$ and/or fusion systems over subgroups of $S$, the fusion system generated by $\mathcal{X}$ is denoted $\langle \mathcal{X} \rangle$ and defined to be the smallest fusion system on $S$ whose morphism set contains $\mathcal{X}$.
 Equivalently, $$\langle \mathcal{X} \rangle= \bigcap_{\G \in \U^S, \mathcal X \subset \Hom(\G)}\G.$$  The morphisms in $\langle \mathcal{X} \rangle$ are composites of restrictions of homomorphisms in $\mathcal{X} \cup \Inn(S)$. Note that $\langle X\rangle^\alpha = \langle X^\alpha\rangle$ for all $\alpha \in\Aut(S)$.

With the above terminology in mind, we now introduce the class of fusion systems on which we focus all our attention: those which are saturated. If $\F$ is a fusion system on $S$ and $P \subseteq S$, we write $$P^\F=\{P\alpha \mid \alpha \in \Hom_\F(P,S)\}$$ for the set of all images of $P$ under morphisms in $\F$ and call this the \textit{$\F$-conjugacy class} of $P$. For $g \in S$, we write $g^\F$ rather than $\{g\}^\F$.

\begin{Def}
Let $\F$ be a fusion system on a finite $p$-group $S$ and $P,Q \le S$. Then,
\begin{itemize}
\item[(1)] $P$ is \textit{fully $\F$-normalized} provided $|N_S(P)| \ge |N_S(Q)|$ for all $Q \in P^\F$;
\item[(2)] $P$ is \textit{fully $\F$-centralized} provided $|C_S(P)| \ge |C_S(Q)|$  for all $Q \in \P^\F$;
\item[(3)]   $P$ is \textit{fully $\F$-automized} provided $\Aut_S(P) \in \Syl_p(\Aut_\F(P))$;
\item[(4)] $P$ is \textit{$S$-centric} if $C_S(P)=Z(P)$, and \textit{$\F$-centric} if $Q$ is centric for all $Q \in P^\F$;
\item[(5)] $P$ is \textit{$\F$-essential} if $P < S$, $P$ is $\F$-centric, fully $\F$-normalized and $\Out_\F(P)$ contains a strongly $p$-embedded subgroup; write $\E_\F$  to denote the set of $\F$-essential subgroups of $\F$;
\item[(6)] $P$ is \textit{weakly $\F$-closed} if $P\alpha=P$ for all $\alpha \in \Hom_\F(P,S)$;
\item[(7)] $P$ is \textit{strongly $\F$-closed} if for each $g \in P$, $g^\F \subseteq P$;
\item[(8)] if $\alpha \in \Hom_\F(P,Q)$ is an isomorphism, $$N_\alpha=\{g \in N_S(P) \mid \alpha^{-1}c_g \alpha \in \Aut_S(Q)\}$$ is the $\alpha$-\emph{extension control subgroup} of $S$;
\item[(9)] $Q$ is $\F$-\emph{receptive} provided for all isomorphisms $\alpha \in \Hom_\F(P,Q)$, there exists $\widetilde{\alpha} \in \Hom_\F(N_\alpha,S)$ such that $\widetilde{\alpha}|_P =\alpha;$
\item[(10)] $P$ is $\F$-\emph{saturated} provided there exists $Q \in P^\F$ such that $Q$ is simultaneously
\begin{itemize}
\item[(a)] fully $\F$-automized; and
\item[(b)] $\F$-receptive;
\end{itemize}
\item[(11)] $\F$ is \emph{saturated} if every subgroup of $S$ is $\F$-saturated.
\end{itemize}
\end{Def}

If $G$ is a finite group containing $S$ as a Sylow $p$-subgroup then $\F_S(G)$ is a saturated fusion system. A saturated fusion system $\F$ on $S$ is \textit{realizable} if $\F=\F_S(G)$ for such a group $G$, otherwise $\F$ is \textit{exotic}.

We now turn to the problem of demonstrating that a fusion system is saturated. The following property may be regarded a modification of $\F$-receptivity:

\begin{Def}
Let $\F$ be a fusion system on $S$. $P \le S$ has the \textit{$\F$-surjectivity property} if for each $C_S(P)P \le R \le N_S(P)$, $$N_{\Aut_\F(R)}(P) \rightarrow N_{\Aut_\F(P)}(\Aut_R(P))$$ is surjective.
\end{Def}

Here is the main computational tool we use when proving saturation:

\begin{Thm}\label{t:surjprop}
Let $\F$ be a fusion system on $S$. Then $\F$ is saturated if and only if the following conditions hold:
\begin{itemize}
\item[(1)] $S$ is fully $\F$-automized;
\item[(2)] in every $\F$-conjugacy class of $\F$-centric subgroups there is a fully $\F$-normalized subgroup with the $\F$-surjectivity property;
\item[(3)] $\F=\langle \Aut_\F(P) \mid \mbox{ $P$ is $\F$-centric }\rangle. $
\end{itemize}
\end{Thm}

\begin{proof}
See \cite[Lemma 6.6, Theorem 6.16]{CravenTheory}.
\end{proof}

Most of our calculations will work with $\Aut_\F(S)$-conjugacy classes of subgroups of $S$. Note that if $P$ is fully $\F$-normalized and has the   surjectivity property then the same is true of $P\beta$   for all $\beta \in \Aut_\F(S)$. This means that it is enough to check Theorem~\ref{t:surjprop} (2) for a single $\Aut_\F(S)$-conjugacy class representative within each $\F$-conjugacy class.

In order to restrict the class of subgroups we apply the following result which may be regarded as a strengthening of the `only if' direction of Theorem \ref{t:surjprop}:

\begin{Thm}[Alperin-Goldschmidt]\label{t:alp}
If $\F$ is saturated then $\F= \langle \Aut_\F(S), \Aut_\F(E) \mid E \in \E_\F \rangle$.
\end{Thm}

\begin{proof}
See \cite[Theorem I.3.5]{AschbacherKessarOliver2011}.
\end{proof}

If some subset $\E$ of the set of subgroups of $S$ has the property that   $\{\Aut_\F(E) \mid E\in \E\}$  generates $\F$ then the same is true of any set of $\F$-conjugacy class representatives of elements of $\E$:

\begin{Lem}\label{l:fclass}
Let $\E$ be a set of subgroups of $S$ such that  $\F= \langle \Aut_\F(S), \Aut_\F(E) \mid E \in \E \rangle$.  Let  $\E^\circ \subseteq \E$ be a set of $\F$-conjugacy class representatives. Then $\F= \langle \Aut_\F(S), \Aut_\F(E) \mid E \in \E^\circ \rangle$.
\end{Lem}

\begin{proof}
See \cite[Proposition 7.25]{CravenTheory}.
\end{proof}

We close this section with some additional definitions pertaining to the normal structure of a saturated fusion system $\F$ on $S$. We denote by $O_p(\F)$ the largest subgroup of $S$ which is left invariant by all morphisms in $\F$. A basic fact concerning this subgroup is that $O_p(\F)$ is contained in every member of $\E_\F$ (see \cite[Theorem 5.39]{CravenTheory}). The smallest normal subsystem of $\F$ of index prime to $p$ is denoted $O^{p'}(\F)$ (see \cite[Section I.6]{AschbacherKessarOliver2011}.) Finally recall from \cite[Section I.7]{AschbacherKessarOliver2011} that $$\begin{array}{rcl} \mathfrak{foc}(\F)&:=&\langle g^{-1}(g\alpha) \mid g \in P \le S \text{ and } \alpha \in \Aut_\F(P)  \rangle; \\ \mathfrak{hyp}(\F)&:=&\langle g^{-1}(g\alpha) \mid g \in P \le S \text{ and } \alpha \in O^p(\Aut_\F(P))  \rangle,\end{array}$$ and that a subsystem $\E$ of $\F$ \textit{has $p$-power index} in $\F$ if $T \ge \mathfrak{hyp}(\F)$ and $\Aut_\E(P) \ge O^p(\Aut_\F(P))$ for each $T \le S$. If $\F$ is saturated then there is a unique minimal saturated fusion system on $\mathfrak{hyp}(\F)$ which is normal in $\F$ which we denote by $O^p(\F)$ (\cite[Theorem I.7.4]{AschbacherKessarOliver2011}).
A saturated fusion system $\F$ is \textit{reduced} if and only if $\F=O^p(\F)=O^{p'}(\F)$ and $O_p(\F)=1$.
 We need the following two facts:

\begin{Lem}\label{l:focf}
Let $\F$ be a saturated fusion system on $S$. The following hold:
\begin{enumerate}
\item $\mathfrak{foc}(\F)=\langle g^{-1}(g\alpha) \mid g \in Q \le S,$ $Q$ is $\F$-essential or $Q=S$, $\alpha \in \Aut_\F(Q) \rangle$.
\item $O^p(\F)=\F$ if and only if $\mathfrak{foc}(\F)=\mathfrak{hyp}(\F)=S$.
\end{enumerate}
\end{Lem}

\begin{proof}
(1) is an immediate consequence of Theorem \ref{t:alp} while (2) follows from \cite[Corollary I.7.5]{AschbacherKessarOliver2011}.
\end{proof}

\section{The fusion graph and automizer sequences}\label{s:auto}

In what follows, we   assume $S$ is a fixed finite $p$-group. We think of this as a fixed post: if we have a fusion system on a group isomorphic to $S$, we move it by an isomorphism to be a fusion system in $\mathcal U^S$ and we consider $S$ as our fixed representative of all finite groups isomorphic to $S$.

\begin{Def}\label{f:fd}
A \textit{fusion datum} $\D$ on $S$ is a pair $(\Q,\A)$, where:
\begin{itemize}
\item[(1)]  $\Q$ is a  sequence of $S$-centric subgroups  of $S$ whose first element is $S$;
\item[(2)] $\mathcal{A}$ is a map which associates to each $Q \in \Q$ a subgroup  $\A(Q) \le \Aut(Q)$ which contains $\Aut_S(Q)$.
    \item[(3)] no two members of $\Q$ are in the same $\A(S)$-orbit.
\end{itemize}
We write $\F(\D):=\langle \A(Q) \mid Q \in \Q \rangle$ for the smallest fusion system on $S$ containing each $\A(Q)$.
\end{Def}

By Theorem \ref{t:alp}, every saturated fusion system is given by a fusion datum on $S$. However it is not the case that every fusion system on $S$ arises in this way. For example, if $S=\langle s,t\rangle$ is elementary abelian of order $4$, we can define a fusion system $\G$ which has all the inclusion maps evident in $\F_S(S)$ and is then determined by  $\Aut_\G(X)=1$ for all $X \le S$, $\Hom_\G(\langle s\rangle, \langle st\rangle)=\emptyset = \Hom_\G(\langle t\rangle, \langle st\rangle)$ and $\theta: s \mapsto t$ the unique element of $\Hom_\G(\langle s\rangle, \langle t\rangle)$.

  Finally, note that, if $\F=\F(\D)$ arises from some fusion datum $\D$, we have $\A(Q) \le \Aut_\F(Q)$ for each $Q \in \Q$ but there is no guarantee that equality holds. We will return to this point shortly.

\begin{Def}
The  fusion data $\D_1=(\Q_1,\A_1)$ and $\D_2=(\Q_2,\A_2)$ on $S$ are \emph{isomorphic} if and only if there exists $\theta\in \Aut(S)$ such that $$\Q_2 = \{Q\theta \mid Q \in \Q _1\}$$  and, for $Q \in \Q_1$, $$ \A_2(Q\theta)=\A_1(Q)^\theta \le \Aut(Q\theta).$$
\end{Def}

 The following lemma is immediate:

\begin{Lem}\label{l:autseq}
Suppose that $\D_1 $ and $\D_2$ are isomorphic fusion data on $S$.  Then  $\F(\D_1)$ and $\F(\D_2)$ are isomorphic.
\end{Lem}

We emphasise that two different fusion data on a group $S$ can have equal fusion systems.

Given a fusion datum $\D=(\Q,\A)$, we let $S/\D$ denote a fixed set of $\A(S)$-orbit representatives of all subgroups of $S$ chosen so that each $Q \in \Q$ is contained in $S/\D$.  For $X \le S$, let $[X]$ be the $\A(S)$-orbit representative of $X$ in $S/\D$.

\begin{Def}\label{fg-def}
Suppose that $\D=(\Q,\A)$ is a   fusion datum on $S$.   The \textit{fusion graph of $\D$ } is a    graph $\Gamma(\D)$ whose vertices are  elements of $S/\D$ and edges are defined to be the $2$-element sets $\{R, T\} \subseteq S/\D$ such that  there exist $\alpha , \gamma \in \A(S)$, $Q \in \Q$ and $\beta \in \A(Q)$ such that $R\alpha, T\gamma \le Q$ and $R\alpha \beta=T\gamma$.
\end{Def}

 Given a fusion datum $\D$, $\Gamma(\D)$ is a finite  graph.  For each edge $\{R,T\}$ in the undirected graph  $\Gamma(\D)$, we temporarily choose an arbitrary orientation $[R,T]$. Since $\{R,T\}$ is an edge,  there exists $Q \in \Q$, $\alpha, \gamma \in \A(S)$ and $\beta \in \A(Q)$ such that $R\alpha \beta \gamma^{-1}= T$ and $T\gamma\beta^{-1}\alpha^{-1}= R$.  We set $\Theta_{R,T}= \alpha\beta\gamma^{-1}$ and $\Theta_{T,R}=\Theta_{R,T}^{-1}$ and label $\{R,T\}$ with the pair $[\Theta_{R,T},\Theta_{T,R}]$.
Clearly $\Gamma(\D)$  is uniquely determined up to isomorphism by $\mathcal D$ whereas typically there will be many different labels for the edges of $\Gamma(\D)$. Given $\D=(\Q,\A)$ and $\F= \F(\D)$, we can encode  Definition~\ref{fg-def} in order to produce a function \texttt{LabelledFusionGraph(F)} which has as input $\F=\F(\D)$ and outputs the graph $\Gamma(\D)$ and a set of labels, which are not necessarily well-defined but are  stored and  fixed once and for all.

%
%Note there may be numerous choices for $Q$, $b,d, $ and $\varphi$ in Definition \ref{fg-def}.  We choose one such  quadruple $(Q,\varphi,b,d)$ and label the edges $(R,T)$ and $(T,R)$ with the morphisms $\Theta_{R,T}$ and $\Theta_{T,R}$ respectively where $\Theta_{R,T}:=c_b^{-1} \varphi|_{R'} c_d \in \Hom_\F(R,T)$. We also insist that $\Theta_{T,R}=\Theta_{R,T}^{-1}$. We call the resulting labeled graph the \emph{labelled fusion graph} and continue to denote it by $\Gamma$.

\begin{Lem}\label{l:fusdat} Assume that $\D$ is a fusion datum and set $\F=\F(\D)$.
\begin{enumerate}
\item If $X, Y $ are vertices in $\Gamma(\D)$ and $\{X=X_0, X_1\},\{X_1,X_2\}, \dots, \{X_{k-1},X_k=Y\}$ is a path connecting $X$ to $Y$ in $\Gamma(\D)$, then $\prod_{i=0}^k\Theta_{X_{i},X_{i+1}}\in \Hom_\F(X,Y)$.
 \item If $X', Y'\le S$ are $\F$-conjugate and $\alpha, \beta \in \A(S)$ are such that $X' \alpha = X$ and $Y'\beta = Y$ for some $X,Y \in S/\D$, then $X$  and $Y$ are in the same connected component of $\Gamma(\D)$.
\end{enumerate}
  \end{Lem}

 For distinct $X,Y \in S/\D$ in the same connected component of $\Gamma(\D)$, we temporarily choose an arbitrary orientation $[X,Y]$ and label $\{X,Y\}$ with the pair $[\Theta_{X,Y},\Theta_{Y,X}]$ where  $\Theta_{X,Y}$ is the element of $\Hom_\F(X,Y)$ given by Lemma \ref{l:fusdat}(1) and $\Theta_{Y,X}:=\Theta_{X,Y}^{-1}$.

 Next we explain how, using $\Gamma(\D)$, one can calculate $\Aut_\F(P)$ for $P \le S$. By replacing $P$ by a suitable $\A(S)$-conjugate if necessary we may assume that $P $ is a vertex of $\Gamma(\D)$. Let $\Gamma(P)$ be the connected component of $\Gamma(\D)$ containing the vertex $P$ and notice that $$P^\F=\{X\beta \mid X \in \Gamma(P), \beta \in \A(S)\}.$$  For $X \in P^\F$, we define $\Theta_{X,X}$ to be the identity map. For $X',Y' \in P^\F$ with $X'\ne Y'$ we select  $\beta,\delta \in \A(S)$   such that
$X'=X\beta$ and $Y'=Y\delta$ with $X,Y \in S/\D$  and specify $$\Theta_{X',Y'}=\beta^{-1} \Theta_{X,Y} \delta \mbox{ and } \Theta_{Y',X'}=\Theta_{X',Y'}^{-1} .$$
It will be inconsequential that these assignments are not well-defined. The objective is to provide an element of $\Hom_\F(X, Y)$ for each $X,Y \in  P^\F$.

By construction, $\Theta_{Y,X} = \Theta_{X,Y}^{-1}$.  For a pair of subgroups $X,Y \in S/\D$, and indeed for subgroups $X,Y \le S$, by the above discussion we can compute and hence also implement   \texttt{IsConjugate(F,X,Y)}, \texttt{IsCentric(F,X)}, \texttt{IsFullyNormalized(F,X)}, \texttt{IsFullyCentralized(F,X)}.  Furthermore, if   \texttt{IsConjugate(F,X,Y)} outputs true, it also outputs an element  $\theta \in \Hom_\F(X,Y)$.

For each $T \in P^\F$, define  $$\Q^{\ge T} =\{Q \in \Q \mid T \le Q\}.$$  We set $$C_T= \langle \Theta_{T,T\mu}(\mu|_{T\mu}) ^{-1}\mid \mu \in \A (Q), Q \in \Q^{\ge T} \rangle  $$
and then put $$B_P= \langle \Theta_{P,T}\psi\Theta_{T,P}\mid \psi \in C_T \text { and } T \in P^\F\rangle.$$ Finally define $$A_P = \langle B_P, \Theta_{P,R}\Theta_{R,T}\Theta_{T,P} \mid R,T \in P^\F\setminus \{P\}\rangle .$$
Since $\F$ is a fusion system,  $(\mu|_{T\mu}) ^{-1} \in \Hom_\F(T\mu, T)$ for each $\mu \in \A(Q)$ and $Q \in \Q^{\ge T}$. Therefore $C_T \le \Aut_\F(T)$ and then $B_P, A_P \le \Aut_\F(P)$.

\begin{Prop}\label{p:AutFP} For $P \le S$,  we have $\Aut_\F(P)= A_P$.
\end{Prop}

\begin{proof} By definition and construction we have $A_P \le \Aut_\F(P)$.
Suppose that $\alpha \in \Aut_\F(P)$.  Since $\F=\F(\D)$, for some $k \ge 1$ there exist $Q_1, Q_2, \ldots, Q_k \in\Q$ and $\widehat \alpha_i \in \A(Q_i)$ such that $P\le Q_1$, and setting $P_0=P$, and for $i \ge 1$, $P_i=P_{i-1}\widehat{\alpha_i}$  and $ \alpha_{i}=\widehat \alpha_i|_{P_{i-1}}$ we have $$\alpha=\alpha_1\alpha_2 \cdots \alpha_k.$$
Since $\alpha \in \Aut_\F(P)$, we have  $P=P_k$.  Notice it may well be that $Q_1=S$ and that $P_2$ is the element of $S/\D$ which represents $P$.  By examining the product $\alpha_1\alpha_2 \cdots \alpha_k$, we intend to  show that $\alpha$ may be written as a product of elements in $A_P$.
By definition for each $0 \le i \le k-1$ we have $\alpha_{i+1}= \mu|_{P_i}$ for some $\mu \in \A(Q_{i+1})$ with $P_i \le Q_{i+1}$. Hence $\Theta_{P_i,P_{i+1}} \alpha_{i+1}^{-1} \in C_{P_i}$ and
we may write $\alpha_{i+1} = c_i \Theta_{P_{i},P_{i+1}}$ for some $c_i \in C_{P_i}$.  Now we have
$$\begin{array}{rcl}
\alpha &=& \alpha_1 \cdots\alpha_k\\
&=&c_0 \Theta_{P,P_{1}}c_1  \Theta_{P_{1},P_{2}}\cdots c_{k-1}\Theta_{P_{k-1},P}\\
&=&c_0 (\Theta_{P,P_{1}}c_1\Theta_{P_{1},P})\Theta_{P_{},P_1}  \Theta_{P_{1},P_{2}}c_2 \Theta_{P_2,P_3}c_3\Theta_{P_3,P_4}c_4 \cdots c_{k-1}\Theta_{P_{k-1},P} \\

&=&c_0 (\Theta_{P,P_{1}}c_1\Theta_{P_{1},P})(\Theta_{P_{},P_1}  \Theta_{P_{1},P_{2}}c_2 \Theta_{P_2,P_1} \Theta_{P_1,P}) \Theta_{P,P_1}\Theta_{P_1,P_2}\Theta_{P_2,P_3}c_3\Theta_{P_3,P_4}c_4 \cdots c_{k-1}\Theta_{P_{k-1},P} \\

& \vdots & \\

&=& c_0 \displaystyle\prod_{i=1}^{k-1} (\Theta_{P,P_1} \dots \Theta_{P_{i-1},P_i} c_i\Theta_{P_i,P_{i-1}} \dots  \Theta_{P_1,P}) \cdot \Theta_{P,P_1}  \Theta_{P_{1},P_{2}}\dots \Theta_{P_{k-1},P}\\
\end{array}$$
Now fix $1 \le i \le k-1$ and set $\beta_i=\Theta_{P,P_1} \dots \Theta_{P_{i-1},P_i} c_i\Theta_{P_i,P_{i-1}} \dots  \Theta_{P_1,P}$. We claim that $\beta_i \in A_P$. By definition of $A_P$, for each $1 \le j \le i-1$ we have \begin{equation}\label{e:hj}\Theta_{P,P_j} \Theta_{P_j,P_{j+1}}\Theta_{P_{j+1},P}=h_j \end{equation} for some $h_j \in A_P$. Hence:

$$\begin{array}{rcl} \beta_i &= & \Theta_{P,P_1}\Theta_{P_1,P_2}\Theta_{P_2,P_3}\Theta_{P_3,P_4} \cdots \Theta_{P_{i-1},P_i} c_i\Theta_{P_i,P_{i-1}} \dots  \Theta_{P_4,P_3}\Theta_{P_3,P_2}\Theta_{P_2,P_1}\Theta_{P_1,P} \\ & = & h_1\Theta_{P,P_2}\Theta_{P_2,P_3}\Theta_{P_3,P_4} \cdots \Theta_{P_{i-1},P_i} c_i\Theta_{P_i,P_{i-1}} \dots  \Theta_{P_4,P_3}\Theta_{P_3,P_2}\Theta_{P_2,P}h_1^{-1} \\
 & = & h_1h_2\Theta_{P,P_3}\Theta_{P_3,P_4} \cdots \Theta_{P_{i-1},P_i} c_i\Theta_{P_i,P_{i-1}} \dots  \Theta_{P_4,P_3}\Theta_{P_3,P}h_2^{-1}h_1^{-1} \\
&  \vdots & \\
&= &h_1h_2 \cdots h_{i-1} \Theta_{P,P_i} c_i \Theta_{P_i,P} h_{i-1}^{-1} h_{i-2}^{-1} \cdots h_1^{-1}  \in A_P\\
 \end{array}$$ since $\Theta_{P,P_{i}}c_i \Theta_{P_i,P} \in B_P \le A_P$.

%As an example consider $$\Theta_{P,P_1}\Theta_{P_{1},P_{2}}\Theta_{P_{2},P_{3}}c_3\Theta_{P_{3},P_{2}} \Theta_{P_2,P_1}\Theta_{P_1,P}.$$ Letting $h_i $ denote appropriate elements of $A_P$ we have

%$$\begin{array}{rcl}\Theta_{P,P_1}\Theta_{P_{1},P_{2}}\Theta_{P_{2},P_{3}}c_3\Theta_{P_{3},P_{2}} \Theta_{P_2,P_1}\Theta_{P_1,P}&=& h_1\Theta_{P,P_2}\Theta_{P_{2},P_{3}}c_3 \Theta_{P_{3},P_{2}} \Theta_{P_2,P}h_2 \\ &=& h_1h_3\Theta_{P,P_{3}}c_3 \Theta_{P_{3},P} h_4h_2\\&  \in& A_P
%\end{array}$$

Finally, setting $\displaystyle\beta=\prod_{i=1}^{k-1} \beta_i$ and letting $h_j$ be defined as in (\ref{e:hj}) in the case $i=k-1$ we have,

$$\begin{array}{rcl}
\alpha &=& c_0 \beta \cdot \Theta_{P,P_1}  \Theta_{P_{1},P_{2}}\Theta_{P_2,P_3}\Theta_{P_3,P_4}\cdots \Theta_{P_{k-1},P}\\

&=& c_0 \beta \cdot h_1 \Theta_{P,P_2} \Theta_{P_2,P_3}\Theta_{P_3,P_4}\cdots \Theta_{P_{k-1},P}\\

&=& c_0 \beta \cdot h_1h_2 \Theta_{P,P_3}\Theta_{P_3,P_4}\cdots \Theta_{P_{k-1},P} \\

& \vdots & \\
&=&c_0  \beta \cdot h_1h_2 \cdots h_{k-2} \in A_P,\\
\end{array}$$
as required.
\end{proof}

\begin{Rem}
We make some remarks concerning the efficient computation of $A_P$ from $\Gamma$. For $X \in \Gamma(P)$, we fix a transversal $\mathcal T_X$ to $N_{\A(S)}(X)$ in $\A(S)$ which contains $1_{\A(S)}$. Then $X^{\A(S)}= \{X\beta \mid \beta \in \mathcal T_X\}$. For $X,Y\in \Gamma(P)$, $X'\in X^{\A(S)}$ and $Y'\in Y^{\A(S)}$ we first define $\Theta_{X',Y'} = \beta^{-1}\Theta_{X,Y} \delta$ where $\beta \in \mathcal T_X$, $\delta \in \mathcal T_Y$ with $X'=X\beta$ and $Y'=Y\delta.$ With this definition, for $X,Y \in \Gamma(P)$, we have $$ \Theta_{P,X'}\Theta_{X',Y'}\Theta_{Y', P}=\Theta_{P,X}\Theta_{X,Y}\Theta_{Y,P}.$$ Hence $$\Aut_\F(P)= \langle B_P, \Theta_{P,X}\Theta_{X,Y}\Theta_{Y,P}\mid X,Y \in \Gamma(P)\rangle.$$ Now a typical generator of $B_P$ has the form $$\Theta_{P,X'}\Theta_{X',Y'}\mu^{-1}|_{Y'}\Theta _{X',P}$$ where $X', Y' \le Q\in \Q$ and $\mu \in \A(Q)$. Notice that
$$\begin{array}{rcl}\Theta_{P,X'}\Theta_{X',Y'}\mu^{-1}|_{Y'}\Theta _{X',P} &=& \Theta_{P,X}\beta \beta^{-1}\Theta_{X,Y} \delta \mu^{-1}|_{Y'} \beta^{-1}\Theta_{X,P} \\&=& \Theta_{P,X} \Theta_{X,Y}\delta\mu^{-1}|_{Y'}\beta^{-1}\Theta_{X,P}. \end{array}$$ Thus $$B_P= \langle  \Theta_{P,X} \Theta_{X,Y}\delta\mu^{-1}|_{Y'}\beta^{-1}\Theta_{X,P} \rangle$$
where

\begin{itemize}
\item[(1)] $X, Y$ are vertices in $\Gamma(P)$;
\item[(2)] $\beta \in \T_X$ and $\delta \in \T_Y$ are such that $\langle X\beta, Y\delta\rangle \le Q$ for some $Q \in \Q$; and
\item[(3)] $\mu \in \A(Q)$ is such that $X\beta\mu=Y\delta$.
\end{itemize}

Finally, to avoid running over all the elements of $\A(Q)$, we note that we can write  $\mu= \sigma \tau$ where $\sigma \in N_{\A(Q)}(X\beta)$. Thus in practice we add all $\Theta_{P,X} \Theta_{X,Y}\delta\rho^{-1}|_{Y'}\beta^{-1}\Theta_{X,P}$ to $B_P$ where $\rho$ runs through a set of  generators for $N_{\A(Q)}(X\beta)$ and then run through a transversal of $N_{\A(Q)}(X\beta)$ in $\A(Q)$. Furthermore, we note that, if $T \in P^\F $ with $T \le Q \in\Q$, then, as $A_P$ contains all elements of the form $\Theta_{P,T}\Theta_{T,R}\Theta_{R,P}$,   for $\mu \in \A(Q)$ we have $\Theta_{P,T}\Theta_{T,T\mu}\mu^{-1}\Theta_{T,P} \in A_P$ if and only if $$(\Theta_{P,T}\Theta_{T,T\mu}\Theta_{T\mu,P})^{-1}\Theta_{P,T}\Theta_{T,T\mu}\mu^{-1}\Theta_{T,P} (\Theta_{P,T}\Theta_{T,T\mu}\Theta_{T\mu,P}) \in A_P$$ if and only if
$$\Theta_{P,T\mu}\mu^{-1}\Theta_{T,T\mu}\Theta_{T\mu,P}\in A_P$$
if and only if
$$\Theta_{P,T\mu}\Theta_{T\mu,T}\mu\Theta_{T\mu, P} \in A_P.$$
This means that we only need to add automorphisms to $B_P$ which correspond to $\A(Q)$-orbit   representatives of elements of $\Gamma(P)$ contained in $Q$.
\end{Rem}

Motivated by saturated fusion systems, we make the following definition:

\begin{Def}\label{d:autseq}
Let $\D=(\Q,\A)$ be a fusion datum on $S$ with fusion system $\F=\F(\D)$. Then $\D$ is an \textit{automizer sequence} on $S$ if for each $Q \in \Q$:
\begin{itemize}
\item[(1)] $\A(Q)=\Aut_\F(Q)$;
\item[(2)] $Q$ is $\F$-centric;
\item[(3)]  $\Aut_S(Q)$ is a Sylow $p$-subgroup of $\A(Q)$;
\item[(4)] for $Q \neq S$, $\A(Q)/\Inn(Q)$ has a strongly $p$-embedded subgroup.
\end{itemize}
\end{Def}

%%%%%%%%%%%%%%%

Suppose that $\D$ is an automizer sequence. Then $\A(S)$ has $\Inn(S)$ as a normal subgroup and $\A(S)/\Inn(S)$ is a $p'$-group by Definition \ref{d:autseq}(3). By the Schur--Zassenhaus Theorem \cite[Theorem 6.2.1]{Gorenstein1980}, $\A(S)$ has a complement $K$ to
$\Inn(S)$.  The group $B= S \rtimes K$ has $S \in \Syl_p(B)$ and satisfies  $\Aut_B(S)=\A(S)$ and $B/S\cong \A(S)/\Inn(S) \cong K$. We refer to the group $B$ as the \textit{Borel group} of $\D$. Using $\Gamma(\D)$, we can determine whether or not all points in Definition \ref{d:autseq} hold. We also have the following observation which summarises our discussion so far.

\begin{Thm}\label{t:main}
Let $\D$ be a fusion datum on $S$ with fusion system $\F=\F(\D)$. The procedures described above can be used to calculate $\Aut_\F(P)$ for all subgroups $P \le S$, determine $\F$-conjugacy between subgroups and check if $\D$ is an automizer sequence.
\end{Thm}

We now turn to the precise relationship between automizer sequences and saturated fusion systems:

\begin{Def}
Let $\F$ be a saturated fusion system on $S$. An \textit{Alperin sequence} $\D$ associated to $\F$ is a fusion datum $(\Q,\A)$ such that
\begin{itemize}
\item[(1)] $\Q \subseteq \E_\F \cup \{S\}$ contains a set of $\F$-conjugacy class representatives of $\E_\F$;
\item[(2)] $\A(Q)=\Aut_\F(Q)$ for each $Q \in \Q$.
\end{itemize}
\end{Def}

\begin{Prop}\label{p:alpaut}
If $\D$ is an Alperin sequence associated to a saturated fusion system $\F$ then $\D$ is an automizer sequence and $\F=\F(\D)$.
\end{Prop}

\begin{proof}
Since $\F$ is saturated, $S$ is fully $\F$-automized and $p \nmid |\Out_\F(S)|$ so part (3)  holds in Definition \ref{d:autseq} in the case $Q=S$. The remaining parts of the definition follow easily, so $\D$ is an automizer sequence. We obtain $\F=\F(\D)$ from Theorem \ref{t:alp} and Lemma \ref{l:fclass}.
\end{proof}

The converse statement also holds:

\begin{Prop}\label{p:essaut}
Let $\D$ be an automizer sequence on $S$ with fusion system $\F=\F(\D)$. If $\F$ is saturated, then $\D$ is an Alperin sequence associated to $\F$.
\end{Prop}

We prove this in a series of lemmas. Recall that, following \cite[Proposition I.3.3(b)]{AschbacherKessarOliver2011}, if $P < S$ is fully $\F$-normalized, $H_\F(P)$ is defined to be the subgroup of $\Aut_\F(P)$ which is generated by those $\F$-automorphisms of $P$ which extend to $\F$-isomorphisms between strictly larger subgroups of $S$. By \cite[Proposition I.3.3 (b)]{AschbacherKessarOliver2011}, if $P$ is fully $\F$-normalized and $H_\F(P) < \Aut_\F(P)$, then $H_\F(P)/\Inn(P)$ is strongly $p$-embedded in $\Out_\F(P)$. Consequently, $H_\F(P) < \Aut_\F(P)$ if and only if $P$ is $\F$-essential.

\begin{Lem}\label{l:davidprop}
Suppose that $\F$ is saturated and $Q < S$ is fully $\F$-normalized. Let $\mathcal{Q}_0$ be the set of all subgroups of $S$ which are not $\F$-conjugate to $Q$. Then $\F=\langle \Aut_\F(P) \mid P \in \mathcal{Q}_0 \rangle$ if and only if  $Q \notin \E_\F.$
\end{Lem}

\begin{proof}
 Set $\F_0=\langle \Aut_\F(P) \mid P \in \mathcal{Q}_0 \rangle$. If $Q \notin \E_\F$, then $\mathcal{Q}_0$ contains $\E_\F$ so $\F=\F_0$ by Theorem \ref{t:alp}. Conversely, assume that $Q \in \E_\F$. It suffices to prove that $H_\F(Q)=\Aut_{\F_0}(Q)$, for then $\Aut_{\F_0}(Q)=H_\F(Q) < \Aut_\F(Q)$ by \cite[Proposition I.3.3(b)]{AschbacherKessarOliver2011} and $\F \neq \F_0$, as needed. If $\varphi \in \Aut_{\F_0}(Q)$, there exist subgroups $Q=Q_0,Q_1,\ldots,Q_n=Q$ and $P_1,\ldots, P_n$ of $S$ and automorphisms $\psi_i \in \Aut_\F(P_i)$ such that $$|P_i| > |Q|, \hspace{2mm} Q_{i-1}\psi_i=Q_i \hspace{1mm} \mbox{and} \hspace{1mm} \varphi=\psi_1|_{Q_0}  \psi_2|_{Q_1} \cdots \psi_n|_{Q_{n-1}}.$$ For $1 \le i \le n-1$, let $\chi_i=\psi_{i+1}|_{Q_i} \cdots  \psi_n|_{Q_{n-1}} \in \Hom_\F(Q_i,Q).$ Since $\Aut_S(Q_i)^{\chi_i}$ is a $p$-subgroup of $\Aut_\F(Q)$ and $Q$ is fully $\F$-automized, there exists $\eta_i \in \Aut_\F(Q)$ with $\Aut_S(Q_i)^{\chi_i'} \le \Aut_S(Q)$ where $\chi_i'=\chi_i\eta_i \in \Hom_\F(Q_i,Q)$. Set $\chi_0'=\chi_n'=\Id_Q$ and $\theta_i=\left(\chi_{i-1}'\right)^{-1} \psi_i|_{Q_{i-1}}  \chi_i' \in \Aut_\F(Q)$ for $1 \le i \le n.$ Now, $$\begin{array}{rcl} \Aut_S(Q)^{\theta_i} \cap \Aut_S(Q) &\ge& \Aut_S(Q)^{\left(\chi_{i-1}'\right)^{-1}  \psi_i|_{Q_{i-1}}  \chi_i'} \cap \Aut_S(Q_i)^{\chi_i'} \\ & \ge & \Aut_S(Q_{i-1})^{\psi_i|_{Q_{i-1}}  \chi_i'} \cap \Aut_S(Q_i)^{\chi_i'} \\ & \ge &\Aut_{P_i}(Q_{i-1})^{\psi_i|_{Q_{i-1}}  \chi_i'} \cap \Aut_S(Q_i)^{\chi_i'} \\ &=& \Aut_{P_i}(Q_i)^{\chi_i'} \cap \Aut_S(Q_i)^{\chi_i'} \\ &=& \Aut_{P_i}(Q_i)^{\chi_i'} > \Inn(Q_i)^{\chi_i'} = \Inn(Q),  \end{array}$$
so $\theta_i \in H_\F(Q)$ by \cite[proof of Proposition I.3.3(b)]{AschbacherKessarOliver2011} and then $\varphi=\theta_1\theta_2 \cdots \theta_n \in H_\F(Q)$. Plainly $H_\F(Q) \le \Aut_{\F_0}(Q)$, so $H_\F(Q)=\Aut_{\F_0}(Q)$ and the result follows.
\end{proof}

\begin{Lem}\label{l:fonjrep}
Let $\F$ be a saturated fusion system on $S$. Let $\E$ be a set of subgroups of $S$ such that  $\F= \langle \Aut_\F(S), \Aut_\F(E) \mid E \in \E \rangle$.  Then there exists a subset $\E^\circ \subseteq \E$ of $\F$-conjugacy class representatives of $\F$-essential subgroups.
\end{Lem}

\begin{proof}
If  $Q$ is $\F$-essential, then the $\F$-automorphism groups of all subgroups of $S$ which are not $\F$-conjugate to $Q$ does not generate $\F$ by Lemma \ref{l:davidprop}. It follows that $Q$ must be $\F$-conjugate to an element of $\mathcal{\E}$.
\end{proof}

\begin{proof}[Proof of Proposition \ref{p:essaut}] Write  $\D= (\Q,\A)$. As $\D$ is an automizer sequence, $\Aut_\F(Q)= \A(Q)$ for each $Q \in \Q$. Since $\F=\F(\D)$, Lemma \ref{l:fonjrep} implies every $\F$-essential subgroup is $\F$-conjugate to an element of $\mathcal{Q}$.   Conversely, for any $Q \in \mathcal{Q}$ , $Q < S$ is fully $\F$-automized by Definition \ref{d:autseq} (1) and (3) and fully $\F$-centralized (it is $\F$-centric by Definition \ref{d:autseq} (2)). Therefore $Q$ is fully $\F$-normalized and hence $\F$-essential by Definition \ref{d:autseq} (4). Thus $\mathcal Q \subseteq \E_\F \cup\{S\}$. This completes the proof.
\end{proof}

 The problem of determining whether or not the fusion system $\F=\F(\D)$ of a given automizer sequence $\D$ is saturated reduces to checking whether or not $\D$ is an Alperin sequence. By Theorem \ref{t:main}, we can solve this problem by checking that the surjectivity property holds for fully $\F$-normalized $\F$-conjugacy class representatives of $\F$-centric subgroups. Furthermore, by the remark after Theorem~\ref{t:surjprop}, we only have to check that each connected component of the fusion graph $\Gamma(\D)$ has a vertex which satisfies the surjectivity property.
This produces a command for a sequence $\D=(\Q,\A)$ with fusion system $\F=\F(\D)$: \texttt{IsSaturated(F)}.
\subsection{Implementation}

We have created a type in {\sc Magma} called \texttt{FusionSystem} and accompanying this we have a command \texttt{CreateFusionSystem} which takes as its input a sequence $\texttt{A}$ of automorphism groups from the fusion datum $\D=(\Q,\A)$.   The declaration \texttt{F:=CreateFusionSystem(A)} calculates and then assigns  a number of attributes to the object $\texttt{F}$:\mbox{} \newline

\begin{tabular}{rl}
\texttt{F`group}:& the $p$-group $S$\\
 \texttt{F`prime}:& the prime $p$\\
\texttt{F`borel}:& the Borel group of $\F$\\
\texttt{F`essentials}: &the sequence $\Q$ (starting with $S$)\\
 \texttt{F`essentialautos}:& the sequence of automorphism groups $\A$
  \\
 \texttt{F`subgroups}:& the subgroups of $S$ up to $B$-conjugacy\\
 \texttt{F`AutF}:& an associative array indexed over \texttt{F`subgroups} where $\Aut_\F(P)$\\& is stored as it is made.
\end{tabular}

Once the fusion graph has been calculated, the attribute \texttt{F`fusiongraph} is  assigned and, if \texttt{IsSaturated(F)} has been executed,  then we also record the outcome as \texttt{F`saturated.}

\section{Searching for saturated fusion systems}\label{s:alg}
 Let $S$ be a finite $p$-group. We now address the problem of determining  automizer sequences $(\Q,\A)$ on $S$ which could potentially give rise to a saturated fusion system $\F$ with the properties that $O^p(\F)=\F$ and $O_p(\F)=1$. Similarly to the strategy in \cite{AOV2017} we first analyze the   $p$-group on its own before we go to the expense of calculating all the potential Borel groups associated with $S$. The input for the procedure is a $p$-group $S$. We immediately transform $S$ into a \texttt{PCGroup}, a group given by a power commutator presentation,  as the calculations that we will perform are more timely with groups in this category. Thus $S:= \texttt{PCGroup(S)}$.

The first lemma tells us not to consider abelian groups.
\begin{Lem}\label{p:burnside}
Suppose that $\F$ is a saturated fusion system on a non-trivial $p$-group $S$ with $O_p(\F)=1$. Then $S$ is non-abelian.
\end{Lem}

\begin{proof}
This is well-known and follows from Theorem~\ref{t:alp}.
\end{proof}

We can also exclude the case when $\Aut(S)$ is a $p$-group when $p \ge 5$:

\begin{Lem}
Let $p \ge 5$. If $\F$ is a saturated fusion system on a non-trivial $p$-group $S$ such that $\Aut_\F(S)$ is a $p$-group, then $\mathfrak{foc}(\F) < S$.
\end{Lem}

\begin{proof}
See \cite[Corollary 6.2]{DGMP2010}.
\end{proof}

The next lemma removes groups where the centre has small index.

 \begin{Lem}\label{two abelian index 2}
 Suppose that $S$ is non-trivial, $\F$ is a saturated fusion system on $S$ with  $O_p(\F)=1$ and $O^p(\F)=\F$. If $|S:Z(S)|\le p^2$,  then  $S \cong p^{1+2}_+$. In particular, all such fusion systems are known.

 \end{Lem}

 \begin{proof} Since  $|S:Z(S)|\le p^2$  and $S$ is non-abelian by Lemma~\ref{p:burnside}, $S/Z(S)$ is not cyclic and so $S$ has at least three abelian maximal subgroups and $|S:Z(S)|= p^2$.  If $p$ is odd, then $S\cong p^{1+2}_+$  by \cite[Theorem 2.1]{Oliver2014} (in fact the hypothesis of this theorem requires $\F$ is reduced  but the proof only uses $O_p(\F)= 1$). So suppose that $p=2$. In this case  \cite[Lemma 1.9]{Oliver2014} implies   $S'$ has order $2$. Noting that \cite[Proposition 2.5]{andersen2013fusion} only uses $O^2(\F)=\F$ and $O_2(\F)=1$, we apply this result  and using $|S:Z(S)|=4$ completes the proof that $S \cong p_+^{1+2}$. Using \cite[Example I.2.7]{AschbacherKessarOliver2011}
and \cite[Theorem 1.1]{RuizViruel2004}, we obtain a description of all the fusion systems. \end{proof}

Since we are interested in finding potential $\F$-essential subgroups $E$ and these subgroups have the property that $\Out_\F(E)$ has a strongly $p$-embedded subgroup we collect some facts about such groups. Recall that a finite group is \emph{almost simple} if and only if it has a unique minimal normal subgroup and this subgroup is simple and non-abelian.

\begin{Lem}\label{strongly p structure}
Suppose that $p$ is a prime and $H$ is a group with a strongly $p$-embedded subgroup $M$. If $M$ contains an elementary abelian subgroup $A$ of order $p^2$, then $O_{p'}(H) \le M$, $M/O_{p'}(H)$ is strongly $p$-embedded in  $H/O_{p'}(H)$ and $H/O_{p'}(H)$ is an  almost simple group.
\end{Lem}

\begin{proof} Set $R= O_{p'}(H)$ and let  $T \in \Syl_p(M)$.  Then, by \cite[Proposition 11.23]{GLS2} (coprime action), $R=\langle C_R(a) \mid a \in A^\#\rangle$. Since $C_H(a) \le M$, for all $a \in A^\#$, $R \le M$.  Then $M/R$ is strongly $p$-embedded in $H/R$. Henceforth, we   assume that $R=1$. Since $O_p(H)=1=R$, the minimal normal subgroups of $H$ are non-abelian. Let $N$ be such a minimal normal subgroup. Then $N$ has order divisible by $p$ and  $N \not \le M$ as otherwise $M \ge N_H(N\cap T) N= H$ by the Frattini Argument.  Let $N_1$ be a simple factor in $N$ and let $C= C_{H}(N_1)$. If $T \cap C \ne 1$, then
     $N_1\le C_H(T\cap C)\le M$ and $C\le C_{H}(N_1\cap T) \le M$.  Thus $N=N_1C_{N}(N_1)\le N_1C \le M$, a contradiction.  Therefore $T\cap C=1$.
Thus $N= N_1$, and $C$ is normal in $H$. Now $T \cap C=1$ implies $C\le O_{p'}(H)=1$ and so $H$ embeds into $\Aut(N)$.  This proves the claim.
\end{proof}

The next proposition is proved by analysing generation properties of the finite  simple groups. This means that it is known only due to the classification theorem of the finite simple groups.

\begin{Prop} \label{Strongly p-embedded Sylows} Suppose that $p$ is a prime, $X$ is a  group,
$K=F^*(X)$   and $T \in \Syl_p(X)$. Assume that $O_{p'}(X)=1$ and that $M$ is a strongly $p$-embedded subgroup of $X$ containing  $T$. Then $O_p(X)=1$, $K$ is a non-abelian simple group, $M \cap K$ is strongly $p$-embedded in $K$,
and  $p$ and $K$ are as follows:
\begin{enumerate}
\item $p$ is arbitrary, $a \ge 1$ and $K \cong \PSL_2(p^{a+1})$, $\PSU_3(p^a)$ ($p^a \ne 2$),  ${}^2\mathrm B_2(2^{2a+1})$ $(p=2)$ or ${}^2{\rm G}_2(3^{2a+1})$ $(p=3)$ and $X/K$ is a
$p'$-group.
 \item $p > 3$, $K \cong \Alt(2p)$,
$|X/K| \le 2$ and $T$ is elementary abelian of order $p^2$.
\item $p=3$, $K \cong \PSL_2(8)$, $X \cong \Aut(\PSL_2(8))\cong {}^2\mathrm G_2(3)\cong \PSL_2(8){:}3$, $T\cong 3^{1+2}_-$ and $T \cap K$ is cyclic of order $9$.
\item $p=3$, $K \cong \PSL_3(4)$, $X/K$ is a $2$-group and $T$ is elementary abelian of order $3^2$.
\item $p=3$, $X=K\cong \mathrm{M}_{11}$ and $T$ is elementary abelian of order $3^2$.
\item $p=5$,  $K \cong {}^2\mathrm B_2(32)$, $X \cong \Aut({}^2\mathrm B_2(32)) \cong {}^2\mathrm B_2(32){:}5$, $T\cong 5^{1+2}_-$ and $T \cap K$ is cyclic  of order $25$.
\item $p=5$, $K \cong
{}^2\mathrm F_4(2)^\prime$,  $|X/K|\le 2$ and $T$ is elementary abelian of order $5^2$.
 \item $p=5$, $K \cong \mathrm {McL}$,  $|X/K|\le 2$ and $T\cong 5^{1+2}_+$.
\item $p=5$, $K  \cong \mathrm{Fi}_{22}$, $|X/K|\le 2$ and  $T$ is elementary abelian of order $5^2$. \item
$p=11$,  $X=K \cong \mathrm J_4$ and   $T\cong 11^{1+2}_+$.
\item $p$ is odd and $T=T \cap K$ is cyclic.
\end{enumerate}\end{Prop}

\begin{proof} Using \cite [Chapter 4,  Lemma 10.3 ]{GLS4} yields that $K$ is a non-abelian simple group, $M \cap K $ is strongly $p$-embedded in $K$ and either $T \le K$ or (3) or (6) holds. If $T$ is cyclic, then this case is listed as (11).  Using \cite [Chapter 4,  Proposition 10.2 ]{GLS4} yields the possibilities for $K$ when $T$ is not cyclic.
 The description of $|X/K|$ follows from the   structure of the automorphism group of $K$ as described in \cite[Theorem 2.5.12, Theorem 5.2.1, Table 5.3]{GLS3}. Finally, the Sylow $p$-subgroup structure listed of the groups in (2)-(10) is either   easy to write down or follows from \cite[Table 5.3]{GLS3} for the sporadic groups.
\end{proof}

To ease application of Proposition~\ref{Strongly p-embedded Sylows} we record the following elementary consequence:

\begin{Cor}\label{Strongly p-embedded Sylows2}
Suppose that $H$ is a group with a strongly $p$-embedded subgroup and let $T \in \Syl_p(H)$. Then one of the following holds:
\begin{enumerate}
\item $T$ is cyclic or quaternion;
\item $T$ is elementary abelian of order at least $p^2$;
\item $p^a \neq 2$ and $T$ is isomorphic to a Sylow $p$-subgroup of $\PSU_3(p^a)$ and has order $p^{3a}$;
\item $p=2$, and $T$ is isomorphic to a Sylow $2$-subgroup of ${}^2\mathrm B_2(2^a)$ and has order $2^{2a}$ for some odd integer $a \ge 3$;
\item $p=3$ and $T$ is isomorphic to a Sylow $3$-subgroup of ${}^2\mathrm G_{2}(3^a)$ and has order $3^{3a}$ for some odd integer $a \ge 1$; or
\item $p=5$  and $T \cong 5_{-}^{1+2}$ is extraspecial.
\end{enumerate}
Furthermore, if $H$ is soluble, then  $T$ is quaternion or cyclic.
\end{Cor}

\begin{proof} We may assume that (1) is false.  In this case, \cite[Theorem 5.4.10 (ii)]{Gorenstein1980} states that $T$ contains an elementary abelian subgroup of order $p^2$.  By Lemma~\ref{strongly p structure}, $H/O_{p'}(H)$ is an almost simple group with a strongly $p$-embedded subgroup. The result now follows from Proposition~\ref{Strongly p-embedded Sylows} applied to $H/O_{p'}(H)$.
\end{proof}

When $p=2$, the following lemma is  \cite[Lemma 1.7(a)]{OliverVentura2009}.

\begin{Lem}\label{V bound} Suppose that $H$ has a strongly $p$-embedded subgroup, $T\in \Syl_p(H)$ and $V$ is a faithful $\mathrm {GF}(p)H$-module.  Then $|V| \ge |T|^2$.
\end{Lem}

\begin{proof} By \cite[Lemma 1.7(a)]{OliverVentura2009} we may assume that $p$ is odd.
We suppose that the conclusion of the lemma is false, that is $|V| < |T|^2$. If $|T |=p$, then $|V| \le p$  and we have a contradiction as $p$ divides $|H|$. Hence $|T|>p$.

Suppose that $T$ is cyclic of order $p^a> p$. Then the Jordan form of a generator of $T$ has a block of at least size $p^{a-1}+1$ and so $$(a-1)p+1\le  p^{a-1}+1 < {2a}.$$  But then $p=2$, a contradiction. Since $p$ is odd, $T$ is not quaternion.
  Hence $T$ is neither cyclic nor quaternion so  \cite[Theorem 5.4.10 (ii)]{Gorenstein1980} and Lemma~\ref{strongly p structure} imply that
  $\ov H=H/O_{p'}(H)$ is an almost simple group with a strongly $p$-embedded subgroup and the isomorphism type of this group is given by Proposition~\ref{Strongly p-embedded Sylows}.  Suppose that $F^*(\ov H) \cong \Alt(2p)$. Then $|V| \le p^3$ and $|F^*(\ov H)|$ divides $|\GL_3(p)|$.   However, $|\Alt(6)|$ does not divide $|\GL_3(3)|$ (as the latter group has order coprime to $5$) and for $p \ge 5$, $|\Alt(2p)|> |\GL_3(p)|$: this is immediate to check if $p=5$; for $p \ge 7$, we have $p^9> |\GL_3(p)|$ and $(2n)!/2 > n^9$ for  $n \ge 6$ by induction.
  If $p=3$ and $F^*(\ov H) \cong \PSL_3(4)$, or $M_{11}$, then $|V| < |T|^2= 3^4$  and we have a contradiction as $5$ does not divide $|\GL_3(3)|$. If $p=5$, and $F^*(\ov H) \cong {}^2\mathrm F_4(2)'$, $\mathrm{Fi}_{22}$ or $\McL$, then $|V| < |T|^2= 5^4$ in the first two cases and $|V|<5^6$ in the last case.  In all cases, a Sylow $3$-subgroup is too big to be contained in $\GL(V)$.
  If $p=11$ and $\ov H \cong \mathrm J_4$ then $|V| \le 11^5$, however a Sylow $2$-subgroup of $\ov H$ is far too big for $H$ to embed into $\GL_5(11)$.

  In the cases when $p=3$ and $F^*(\ov H) \cong \PSL_2(8) $ or $p=5$ and $F^*(\ov H) \cong {}^2\mathrm B_2(32)$ we have $|T| \le p^3$  and so require $H$ to embed into $\GL_5(p)$. However $7$ does not divide $|\GL_5(3)|$ and $41$ does not divide $|\GL_5(5)|$.

Now Proposition~\ref{Strongly p-embedded Sylows} implies that $F^*(\ov H)$ is a rank $1$ Lie type group in characteristic $p$.

Suppose that $F^*(\ov H) \cong \PSL_2(p^a)$. Then   $|\ov H|$ is divisible by $p^{a}+1$.
Since $p$ is odd,  Zsigmondy's Theorem \cite{zsigmondy1892theorie}
implies that either there exists a prime $r$   which divides $p^{2a}-1$ but  does not divide $p^b-1$ for all $1\le b < 2a$  or $p^a$ is a Mersenne prime. Since $T$ is not cyclic, the first possibility holds. But $r$ does not divide $|\GL_{2a-1}(p)|$, and so we conclude $|V| \ge p^{2a}$ in this case.

Suppose that $F^*(\ov H) \cong \PSU_3(p^a)$ or  $p=3$ and $F^*(\ov H) \cong {}^2\mathrm G_2(p^a)$. Then $\ov H$ has order divisible by  $p^{3a}+1$. Since $p$ is odd,  Zsigmondy's Theorem implies there exists a prime $r$ which divides $p^{6a}-1$ and which does not divide $p^b-1$ for all $1\le b < 6a$. In particular $r$ does not divide $p^{3a}-1$, so $r$ divides $p^{3a}+1$ and hence also $|\ov H|$. Since $r$ does not divide $|\GL_{6a-1}(p)|$, we have $|V| \ge p^{6a}$.
 \end{proof}

Using the above observations  about groups with a strongly $p$-embedded subgroup, we have the following proposition which describes   the initial tests that we make. The first of these isolates subgroups of $S$ which are potentially $\F$-essential in some saturated fusion system on $S$.

\begin{Prop}\label{p:etest}
Suppose that $\F$ is a saturated fusion system on $S$  and that $E \in \mathcal E_\F$.  Then
\begin{enumerate}
\item $E$ is  $S$-centric;
\item  $\Out_S(E) \cap O_p(\Out(E)) = 1$;
\item $\Out_S(E)$ is isomorphic to a Sylow $p$-subgroup of a group with a strongly $p$-embedded subgroup;
\item $|E/\Phi(E)| \ge |\Out_S(E)|^2$;

\item If $\Out_S(E) $ has an elementary abelian subgroup of order $p^2$, $\Aut_\F(E)$ is not soluble;
\item $\Out_S(E)$ is abelian of order at most $p^2$ whenever $|S| \le p^8$.
\end{enumerate}
\end{Prop}

\begin{proof} That $E$ is $S$-centric is part of the definition of being $\F$-essential.
To see that (2) holds we just observe
$$\Out_S(E) \cap O_p(\Out(E))\le \Out_\F(E)\cap O_p(\Out(E)) \le O_p(\Out_\F(E)) =1.$$
Part (3) follows as $\Out_\F(E)$ has a strongly $p$-embedded subgroup.
For part (4), we set $V= E/\Phi(E)$ and note that this is a faithful $\mathrm{GF}(p)\Out_\F(E)$-module. Thus (4) comes from  Lemma~\ref{V bound}. Part (5) follows from Lemma~\ref{strongly p structure}.
Finally (6) follows from (4) since, if $|\Out_S(E)| \ge p^3$, then $|E| \ge p^6$ and then  $|S| \ge |N_S(E)| \ge p^9$.
\end{proof}

A subgroup $E \le S$ which satisfies conditions (1)-(6) in Proposition \ref{p:etest} is called \emph{potentially essential}. Denote by $\P_S$ the set of all potentially essential subgroups. Note that if $\P_S = \emptyset$, then any saturated fusion system on $S$ satisfies $O_p(\F)= S$. Also note these conditions depend only upon $S$ so that the set $\P_S$ can be determined relatively quickly. Nonetheless, we need to calculate the automorphism groups of every $S$-centric subgroup. To make this execute more quickly, we note that, if $E \in \mathcal P_S$ and $\alpha \in \Aut(S)$, then $P\alpha \in\mathcal P_S$. Thus we calculate these properties just for representatives of $\Aut(S)$-orbits of subgroups.

Since we are interested in saturated fusion systems $\F$ with $O_p(\F)=1$, we will repeatedly use the following lemma to test whether $O_p(\F)\ne 1$.

\begin{Lem}\label{cortest} Suppose that $\F$ is a saturated fusion system on $S$ and $\mathcal S $ is a subset of $\E_\F$ which contains  a set of $\F$-conjugacy class representatives of members of $\E_\F$. Assume that $K \le \bigcap_{E\in \mathcal S}E$ is invariant under $\Aut_\F(E)$ for each $E \in \mathcal S$ and under $\Aut_\F(S)$. Then $K \le O_p(\F)$. In particular, this conclusion holds if $K$ is characteristic in every $E\in \S$ and in $S$.
\end{Lem}

\begin{proof} By Lemma~\ref{l:fclass},   $\F= \langle \Aut_\F(S), \Aut_\F(E) \mid E \in \mathcal S \rangle$.  Suppose that $A$ and $B$ are subgroups of $S$ and $\phi \in \Hom_\F(A,B)$.  Then $\phi = \widetilde {\phi_0}\widetilde{\phi_1} \dots \widetilde{\phi_k}$ where, for $0\le i\le {k}$, $\widetilde{\phi_i}$ is the restriction of a certain element $\phi_i \in \Aut_\F(E_i)$ for $E_i \in \mathcal S \cup \{S\}$. Then, as $K$ is invariant under each $\phi_i$, we have $\phi^*=   \widetilde{\phi_0}^*\widetilde{\phi_1}^* \dots \widetilde{\phi_k}^* \in \Hom_\F(AK,BK)$ and $K\phi^*=K$, where $\widetilde{\phi_i}^*$ is the restriction of $\phi_i$ to the appropriate overgroup of $K$.  This proves the claim.
\end{proof}

We first apply  Lemma~\ref{cortest} with $\mathcal S= \mathcal P_S$. In this way we discard groups $S$ where $O_p(\F)$ is surely non-trivial.

We now perform our first expensive calculation. For an $\Aut(S)$-representative $E\in \mathcal \P_S$, we first determine the subgroups, up to $N_{\Aut(E)}(\Aut_S(E))$-conjugacy, of $\Aut(E)$ which contain $\Aut_S(E)$ as a Sylow $p$-subgroup, and for each such subgroup $K$ we ascertain whether or not $K/\Inn(E)$ has a strongly $p$-embedded subgroup. If no such subgroups exist we remove $E$ from  $\P_S$; otherwise we store a set $\mathfrak A_E$  which contains $N_{\Aut(E)}(\Aut_S(E))$-conjugacy classes of subgroups we have found with  the property just discussed. The elements of $\P_S$ are now called \emph{proto-essential}. Thus, in addition to satisfying the conditions in Proposition \ref{p:etest}, a proto-essential subgroup possesses a group of outer automorphisms $K$ with a strongly $p$-embedded subgroup and $\Out_S(E) \in \Syl_p(K)$.

We now need to calculate all the potential Borel groups. Thus we identify all $\Out(S)$-conjugacy classes of $p'$-subgroups $B_0 \le \Out(S)$. For each class-representative, we select a complement to $\Inn(S)$ and form a group $B$ with the properties that $S \in \Syl_p(B)$ and $\Out_B(S)=B_0$ (see the remarks following Definition \ref{d:autseq}.) By Lemma \ref{l:autseq} this is fine as we are only interested in listing fusion systems up to isomorphism. The implementation of this assumes that $\Aut_\F(S)$ is soluble. If this is not the case, then the execution terminates reporting that the Borel group is not soluble.  For the groups considered in this article, this never happens.  However, it would occur if we were to attempt to enumerate the fusion systems on an extraspecial group of order $11^3$ and exponent $11$.

Our algorithm   runs through each Borel group in turn. Note that $\Aut_B(E)$ must be a subgroup of $\Aut_\F(E)$,  and so we check that for each $E \in \mathcal \P_S$ and $K \in \mathfrak A_E$ there is $L \in K^{N_{\Aut(E)}(\Aut_S(E))}$ which contains  $\Aut_B(E)$. If no such $L$ exists then we remove $K$ from $\mathfrak A_E$, and if $\mathfrak A_E$ eventually becomes empty then we remove $E$ from $\mathcal \P_S$. We refer to this as the \textit{extension test}.

We then expand $\P_S$ to contain all $B$-conjugacy class representatives of proto-essential subgroups (and not just $\Aut(S)$-conjugacy class representatives.)   We call this new set $\mathcal \P_B$. We determine pairs $(E,A)$ in which $E \in \P_B$ and $A \le \Aut(E)$ is a candidate for $\Aut_\F(E)$ in a saturated fusion system $\F$ on $S$ with Borel group $B$ in which $E$ is $\F$-essential. To achieve this, for each $E \in \mathcal P_B$, and $K \in \mathfrak A_E$ (transferred from the automorphism group of the $B$-conjugacy class representative in $\mathcal \P_S$) we determine all $\Aut(E)$-conjugates of $K$ containing $\Aut_B(E)$ and call this set $\mathfrak A_E^*$. If $E$ is a ``large"  proto-essential subgroup, we also require that $\Aut_B(E)= N_K(\Aut_S(E))$. The next proposition gives an account of the types of checks that we make. We think of these as basic compatibility checks.

\begin{Prop}\label{p:atest}
Suppose that $\mathcal D=(\Q,\A)$ is an automizer sequence on $S$ and assume that $\F=\F(\D)$ is saturated.  Let $n$ be the largest order of an element of $\Q \backslash \{S\}.$ The following hold:
\begin{enumerate}

\item If $Q \in \Q \backslash \{S\}$ is such that $|N_S(Q)| > n$; or $|N_S(Q)|=n$ and $N_S(Q)$ is not $\A(S)$-conjugate to an element of $\Q$ then $$N_{\A(Q)}(\Aut_S(Q)) = \{\varphi|_Q \mid \varphi \in N_{\A(S)}(Q)\}.$$

\item If $P,Q \in \Q$ are such that $P < Q$ then $|N_S(P)| \ge |N_S(P\alpha)|$ for each $\alpha \in \A(Q)$.

\item If $\R \subseteq \Q$ is such that $S \in \R$ and there exists $T \le S$ normalized by $\A(R)$ for each $R \in \R$ then $\Aut_S(T)$ is a Sylow $p$-subgroup of $ \langle \varphi|_T \mid \varphi \in \A(R), R \in \R \rangle\le \Aut(T).$

\end{enumerate}
\end{Prop}

\begin{proof}
Let $Q \in \Q \backslash \{S\}$ be such that $|N_S(Q)| > n$; or $|N_S(Q)|=n$ and $N_S(Q)$ is not $\A(S)$-conjugate to an element of $\Q$. Since $\F$ is saturated, $Q$ is receptive. Therefore each element of $N_{\A(Q)}(\Aut_S(Q))$ extends to a morphism in $  \Aut_\F(N_S(Q))$.

The hypothesis on $Q$ together with Theorem~\ref{t:alp} imply that every $\F$-automorphism of $N_S(Q)$  is the restriction of an $\F$-automorphism of $S$ and so
  $$\Aut_\F(N_S(Q))=\Aut_{\A(S)}(N_S(Q)).$$ Hence,  as $\D$ is an automizer sequence, $$N_{\A(Q)}(\Aut_S(Q))\le  \{\varphi|_Q \mid \varphi \in N_{\A(S)}(Q)\} \le N_{\Aut_\F(Q)}(\Aut_S(Q))=N_{\A(Q)}(\Aut_S(Q)).$$
  This proves (1).

 Part (2) is immediate from the fact that $P$ is fully $\F$-normalized.  If $\R$ and $T$ are as in (3) then $T$ is fully $\F$-normalized and $\Aut_S(T) \le  \langle \varphi|_T \mid \varphi \in \A(R), R \in \R \rangle$ so the conclusion follows immediately from the fact that $T$ is fully $\F$-automized.
\end{proof}

We  run through all $\mathcal Q \subseteq \mathcal P_B$, and for each $Q \in \Q$, we assign $\A(Q)=K \in \mathfrak A_Q^*$. This defines a fusion datum $\D=(\Q,\A)$ and a fusion system $\F= \F(\D)$, and further checks are made to ensure that $\D$ is consistent with being an automizer sequence. The previous calculations are performed before calculating the fusion graph which is time-consuming. Finally, we check \texttt{IsSaturated(F)}. To reduce the amount of work, we remove $\D^\alpha$, for $1 \neq \alpha \in \Aut(S)$ from the systems which we need to test for saturation.

One further test, with variants, is used to speed up the calculations in some stubborn cases. Recall that for $r$ a prime, a finite group $G$ is \emph{$r$-closed} if $|\Syl_r(G)|=1$.

\begin{Lem}
Suppose that $\F$ is a saturated fusion system on $S$, $E_1, E_2 \in \E_\F$ with $E_1=N_S(E_2)$ and $|E_1:E_2|=p$. Set $X= \bigcap_{\alpha \in \Aut_\F(E_1)}E_2\alpha$ and assume that $|E_1:X| \le p^2$. Then $|E_1:X| = p^2$,  $\Aut_\F(E_1)/C_{\Aut_\F(E_1)}(E_1/X)$ is $p$-closed and $C_{\Aut_\F(E_1)}(E_1/X)/\Inn(E_1)$ is a    $p'$-group which is not centralized by $\Out_S(E_1)$. Furthermore, $|S| \ge p^{p+3}$.
\end{Lem}

\begin{proof} Set $C=C_{\Aut_\F(E_1)}(E_1/X)$ and assume that $|E_1:X| \le p^2$. As $|E_1/X|\le p^2$, $\Inn(E_1) \le C$.  If $|E_1:X|= p$, then $X= E_2$ and so $E_2$ is normalized by $N_S(E_1)$. Since $E_1= N_S(E_2)$, we deduce that $E_1= S$, a contradiction. Hence $|E_1:X|=p^2$, $E_1/X$ is elementary abelian and $\Aut_\F(E_1)/C$ is isomorphic to a subgroup of $\GL_2(p)$. If $\Aut_\F(E_1)/C$ has at least two Sylow $p$-subgroups, then it contains a subgroup isomorphic to $\SL_2(p)$ and $\Aut_\F(E_1)$ acts transitively on the maximal subgroups of $E_1$ containing $X$.  In particular, there is an $\F$-conjugate of $E_2$ which is normalized by $N_S(E_1)$ and this contradicts $E_2$ being fully $\F$-normalized. Hence  $\Aut_\F(E_1)/C$ is $p$-closed. Set $\bar{C}=C/\Inn(E_1)$ and let $R \le N_S(E_1)$ be such that $R\ge E_1$ and  $\Aut_R(E_1)/\Inn(E_1) \in \Syl_p(\bar{C})$.  Then $\Aut_R(E_1)$ centralizes $E_1/X$ and so leaves  $E_2$ invariant. Hence $R \le N_S(E_2)= E_1$ and  $\Aut_R(E_1)= \Inn(E_1)$ as is required to show that   $\bar{C}$ is a $p'$-group.

Since $\Out_\F(E_1)$ has a strongly $p$-embedded subgroup and $\Aut_\F(E_1)/C$ is $p$-closed, we have $O_p(\Out_\F(E_1))=1$ and $\bar{C}\Out_S(E_1) \unlhd \Out_\F(E_1)$ so $O_p(\bar{C}\Out_S(E_1))=1$. In particular, $\bar{C}$ is not centralized by $\Out_S(E_1)$. Hence $[\bar{C},\Out_S(E_1)]$ is a non-trivial $p'$-group. Since $E_2$ is $\F$-centric, $|\Phi(E_1)| \ge |[E_1,E_1]| \ge |[E_2,E_1]|\ge p$. Now, as $O_p(\Out_\F(E_1))=1$, \cite[Theorem 5.1.4]{Gorenstein1980} implies that  $\Out_\F(E_1)$  acts faithfully on $E_1/\Phi(E_1)$.  By \cite[Theorem 5.3.2]{Gorenstein1980} $\bar{C}$ acts faithfully on $X/\Phi(E_1)$. Hence the kernel of the representation $\bar{C} \Out_S(E_1) \rightarrow \GL(X/\Phi(E_1))$ is a $p$-group and thus trivial as $O_p(\bar{C} \Out_S(E_1))=1$. Therefore $\bar{C} \Out_S(E_1)$ acts faithfully on $X/\Phi(E_1)$.

We know $\bar{C} \Out_S(E_1)$ is $p$-soluble and so the Hall-Higman Theorem \cite[Theorem 11.1.1 (ii)]{Gorenstein1980} shows that  $|X/\Phi(E_1)| \ge p^{p-1}$. Observe that $|S{:}X|= |S{:}E_1||E_1{:}X|\ge p^3$ so that $$|S| = |S:X||X:\Phi(E_1)||\Phi(E_1)|\ge p^{3+(p-1)+1}= p^{p+3}.$$ This completes the proof of the declared statement.
\end{proof}

In the implementation of \texttt{AllFusionSystems} we output only saturated fusion systems in which $O_p(\F)=1$ and $O^p(\F)=\F$. These properties permeate the search and we make frequent checks that our fusion data induce fusion systems with these properties. In particular  we check $S=\langle g^{-1}(g\alpha) \mid g \in Q \in \Q, \alpha \in \A(Q)\rangle$ as is required by Lemma \ref{l:focf}. We also check that for each $1\ne T \le \bigcap_{Q \in \Q} Q$, there exists $Q \in \Q$ and $\varphi \in \A(Q)$ such that $T\varphi \neq T$, as otherwise $O_p(\F) \neq 1$ by Theorem \ref{t:alp}.

An example which illustrates the execution of the above algorithm is provided in Appendix \ref{s:appendixc}.

\section{Saturated fusion systems on $p$-groups of small order}\label{s:results}

We now list the results of our calculations. We say that a $p$-group $S$ has \textit{type} $G$ if $S$ is isomorphic to a Sylow $p$-subgroup of $G$. The table headings below provide the following information about $S$:

\begin{itemize}
\item \textit{group $\#$}: the  {\sc Magma} $\texttt{SmallGroup}$ identification number of $S$;
\item \textit{sec. rank}: the sectional rank of $S$;
\item \textit{ab. ind. $p$}: whether or not $S$ has an abelian maximal subgroup;
\item \textit{$\#$ s.f.s}: the number of saturated fusion systems on $S$  with $O_p(\F)=1$ and $\F=O^p(\F)$;
\item \textit{reference}: subsection containing precise information about these fusion systems;
\item \textit{type}: the type of $S$;
\item \textit{citation(s)}: article(s) in which saturated fusion systems on this group have been previously considered.
\end{itemize}

Recall that the \emph{sectional rank} of a $p$-group $S$ is the maximal number of generators needed for a subgroup of $S$. When $S$ is maximal class and either a $3$-group or metabelian of order at most $p^6$ we will often indicate its description  $B(p,r; \alpha,\beta, \gamma, \delta)$ as a Blackburn group (see Appendix  \ref{s:appendix}.)

\subsection{Groups of order $p^4$}
 In \cite[Theorem 7.1]{moncho2018fusion} Moragues Moncho gives a list of all simple saturated fusion systems on $p$-groups of order $p^4$. From
\cite[Tables 7.1 and 7.2]{moncho2018fusion} we see that for each $p \ge 7$ there are exactly three such fusion systems, all on a Sylow $p$-subgroup of $\PSp_4(p)$: one is the $p$-fusion system of $\PSp_4(p)$; the others are exotic.  For $p=7$, these are the fusion systems listed as $\F(7^4,7,1)$, $\F(7^4,7,3)$ and $\F(7^4, 7,4)$ in Table~\ref{7^4.7t}.  When $p \le 7$ we extend his result modestly by determining the cases which are not simple but have $O_7(\F)=1$ and $O^7(\F)=\F$.

\begin{Thm}\label{Thmp^4}
Suppose that $\F$ is a saturated fusion system on a $p$-group of order $p^4$ such that $O_p(\F)=1$ and $\F= O^p(\F)$ with $p \in \{3,5,7\}$.  Then $S$ has an abelian subgroup of index $p$ and $\F$ is listed in the following table:
\begin{table}[H]
\renewcommand{\arraystretch}{1.4}
\begin{tabular}{|c|c|c|c|c|c|c|}
\hline
$p$ & group \# & sec. rank &  \# s.f.s & reference                     & type  & citation(s)                                   \\ \hline
$3$ & $7$        & $3$   & $6$               & \ref{3^4.7} &  $\PSp_4(3)$, $3 \wr 3$, $B(3,4;0,0,1,0)$ & \cite{ClellandParker2010},\cite{Oliver2014}, \cite{craven2017reduced}  \\
$3$ & $8$        & $2$   & $2$               &  \ref{3^4.8} & $B(3,4;0,1,0,0), B(3,4;0,0,2,0)$  & \cite{DiazRuizViruel2007},  \cite{Oliver2014}, \cite{craven2017reduced}                    \\
$3$ & $9$        & $2$   & $7$               &  \ref{3^4.9} &  $\PSL_3(19)$, $B(3,4;0,0,0,0)$   &   \cite{DiazRuizViruel2007},  \cite{Oliver2014}, \cite{craven2017reduced}                   \\
$5$ & $7$        & $3$   & $30$              & \ref{5^4.7} &  $\PSp_4(5), B(5,4; 0,1,0,0)$ & \cite{ClellandParker2010},\cite{Oliver2014}, \cite{craven2017reduced} \\
$7$ & $7$        & $3$   & $8$               & \ref{7^4.7} &  $\PSp_4(7), B(7,4; 0,1,0,0)$ & \cite{ClellandParker2010},\cite{Oliver2014}, \cite{craven2017reduced} \\ \hline
\end{tabular}
\end{table}
\end{Thm}

\subsection{Groups of order $p^5$}
For $p \in \{3,5,7\}$ we have the following result concerning saturated fusion systems on a group of order $p^5$:

\begin{Thm}\label{Thmp^5} Suppose that $\F$ is a saturated fusion system on a $p$-group of order $p^5$ such that $O_p(\F)=1$ and $\F= O^p(\F)$ with  $p \in \{3,5,7\}$. Then $\F$ is listed in the following table:

\begin{table}[H]
\renewcommand{\arraystretch}{1.4}
\begin{tabular}{|c|c|c|c|c|c|c|}
\hline
$p$ & group \# & sec. rank & \# s.f.s & reference & type & citation(s)                                \\ \hline
$3$ & $26$      & $2$  & $7$      & \ref{3^5.26}     &  $\PGL_3(19)$, $B(3,5;0,0,0,0)$  & \cite{ClellandParker2010},\cite{Oliver2014}, \cite{craven2017reduced}   \\
$3$ & $51$      & $4$             & $9$      & \ref{3^5.51}    &  $\Alt(12)$, $(3 \wr 3) \times 3$ &   \cite{craven2017reduced} \\
$5$ & $30$      & $4$             & $58$      &  \ref{5^5.30}    &  $\PGL_5(11), B(5,5;0,0,0,0)$ &
\cite{ClellandParker2010},
\cite{Oliver2014}, \cite{craven2017reduced}\\
$7$ & $32$      & $4$             & $10$      &  \ref{7^5.32}    &  $B(7,5;0,0,0,0)$ & \cite{ClellandParker2010},\cite{Oliver2014}, \cite{craven2017reduced}\\
$7$ & $37$      & $3$ & $1$      &  \ref{7^5.37}    & $B(7,5;0,1,0,0)$ & \cite{grazian2018fusion2}\\

 \hline
\end{tabular}
\end{table}

\end{Thm}

Together with results of Grazian \cite{grazian2018fusion2} for $p \ge 5$, Theorem \ref{Thmp^5} completes the classification of saturated fusion systems $\F$ on $p$-groups of order $p^5$ in which $O_p(\F)=1$ and $O^{p}(\F)=\F$. For $p=2$, the result for reduced fusion systems is  \cite[Theorem 5.3]{andersen2013fusion}.

\begin{Cor}\label{p^5-abelian subgroup}
Suppose that $p \ge 3$, $S$ is a $p$-group and $\F$ is a saturated fusion system on $S$ with $\F=O^p(\F)$ and $O_p(\F)=1$. If $|S|= p^5$, then either $S$ has an abelian subgroup of index $p$ or $S$ is $\texttt{SmallGroup($7^5,37$)}$ and $\F$ is the Grazian fusion system. In particular, all saturated fusion systems on $p$-groups of order $p^5$ with $\F=O^p(\F)$ and $O_p(\F)=1$ are known.
\end{Cor}

\begin{proof} Assume that $S$ has no abelian subgroups of index $p$.
If $S$ has sectional rank $2$, then it has rank 2 and by \cite{DiazRuizViruel2007} we have $p=3$. In any case, if $p=3$, the result follows from Theorem~\ref{Thmp^5}. Hence we may suppose that $p\ge 5$. Assume that  $S$ has sectional rank $3$.  By \cite[Theorem C]{grazian2018fusion2}, since $|S|=p^5$, $p=7$ and $\F$ is the unique  saturated fusion system on $\texttt{SmallGroup($7^5,37$)}$  which has $O_7(\F)=1$ and $O^7(\F)=\F$.  So suppose that $S$ has sectional rank at least $4$. Since $S$ is not abelian, $S$ has sectional rank 4.  If any maximal subgroup of $S$ has sectional rank $4$, then it is abelian. Thus we may suppose that every proper subgroup of $S$ has rank at most $3$ and so, as $S$ has sectional rank $4$, we have  $|S/\Phi(S)|=p^4$ and $|\Phi(S)|=p$.  Suppose that $E \le S$ is an $\F$-essential subgroup of $S$.  Then $\Phi(S) \le Z(S) \le E$.  In particular, $E$ is normal in $S$ and $\Phi(E) \le \Phi(S)$.  Since $[E,S] \not \le \Phi(E)$ we infer that $E$ is elementary abelian.  Hence $|E| \le p^3$ as $S$ has no abelian subgroups of index $p$.    Since $C_S(E)= E$, we have $\Aut_S(E)$ has order at least $p^2$. However then $|E|< |\Out_S(E)|^2$ and this contradicts Lemma~\ref{V bound}.
\end{proof}
%
%
%and is elementary abelian of order $p^2$. Since $\Aut_\F(E)\cong \Out_\F(E)$ has a strongly $p$-embedded subgroup and embeds into $\SL_3(p)$, we have a contradiction since $\SL_3(p)$ has  no such subgroup with Sylow $p$-subgroups of order $p^2$. Thus $\F$ has no $\F$-essential subgroups and this is a contradiction as $O_p(\F)=1$.

%
%
%\begin{remark}
%We shall see when we describe our algorithm that for any $\F$-essential subgroup $E$ of $S$, we have  $|S| \ge |\Out_S(E)|^3$ and so we could have closed our proof once we had discovered that $E$ was abelian and $|\Aut_S(E)|\ge p^2$.
%\end{remark}

\begin{Rem}
If we relax the assumption that $O^p(\F)=\F$ in Corollary~\ref{p^5-abelian subgroup}, then we need to add the new saturated fusion system on $S=B(3,5; 0,1, 0, 0)$ discovered in \cite{parker2018fusion2}. This has a subsystem isomorphic with the $3$-fusion system of $\PSL_3(19)$ with index $3$.
\end{Rem}

%
%
%
%\begin{remark}
%If we relax the assumption that $O^p(\F)=\F$ in Corollary~\ref{p^5-abelian subgroup}, then, for $p=3$, then   we need to add  $\F_{S}(G)$ for $G \cong \PGL_3(19)$ and the  fusion systems discovered in \cite{parker2018fusion2}. [Perhaps we should do all $3^5$ without the assumption that $O^p(\F)=\F$.]
%\end{remark}
\subsection{Groups of order $p^6$}

Next, we consider saturated fusion systems on groups of order $p^6$ with $p \in \{3,5\}$.
\begin{Thm}\label{Thmp^6}
Suppose that $\F$ is a saturated fusion system on a $p$-group of order $p^6$ with $O_p(\F)=1$ and $\F= O^p(\F)$ with $p \in \{3,5\}$. Then $\F$ is listed in the following table:

\begin{table}[H]
\renewcommand{\arraystretch}{1.4}
\begin{tabular}{|c|c|c|c|c|c|c|c|c|}
\hline
$p$ & group \# & sec. rank & ab. ind. $p$? & \# s.f.s & reference & type  & citation(s) \\ \hline

$3$ & $95$     & $2$  & yes           & $7$      & \ref{3^6.95}     &  $\PSL^\pm(3,q)$, $\nu_3(q\mp 1)= 3$ & \cite{DiazRuizViruel2007}, \cite{Oliver2014}        \\

$3$ & $97$     & $2$  & yes           & $2$      & \ref{3^6.97}     & $B(3,6;0,0,1,0)$             & \cite{DiazRuizViruel2007}, \cite{Oliver2014}                        \\

$3$ & $98$     & $2$  & yes           & $2$      & \ref{3^6.98}     & $B(3,6;0,0,2,0)$  & \cite{DiazRuizViruel2007}, \cite{Oliver2014}                                   \\

$3$ & $99$     & $2$  & no            & $1$      & \ref{3^6.99}     & $B(3,6;0,1,0,0)$      & \cite{parker2018fusion2}                             \\

$3$ & $100$    & $2$  & no            & $3$      & \ref{3^6.100}     & $B(3,6;0,1,1,0)$            & \cite{parker2018fusion2}                                                    \\

$3$ & $149$    & $4$  & no            & $2$      & \ref{3^6.149}     &  $\Gee_2(3)$      & \cite{parker2018fusion}                           \\

$3$ & $307$    & $4$  & no            & $10$     & \ref{3^6.307}     &  $\PSL_4(3)$      & $-$                           \\

$3$ & $321$    & $4$  & no            & $13$     & \ref{3^6.321}     &  $\PSU_4(3)$      & \cite{BFM2019}                           \\

$3$ & $453$    & $4$  & no            & $21$     & \ref{3^6.453}     &  $\PSL_3(3)^2$, $3^{1+2}_+ \times 3^{1+2}_+$ & $-$ \\

$3$ & $469$    & $4$  & no            & $5$      & \ref{3^6.469}     &  $\PSL_3(9)$       & \cite{clelland2007saturated}                          \\

$3$ & $479$    & $5$  & yes           & $4$      & \ref{3^6.479}     &  $\Alt(15)$, $(3 \wr 3) \times 3^2$     & \cite{craven2017reduced}       \\

$5$ & $240$&$4$&no&$12$&\ref{5^6.240}&$\PSL_3(5) \times \PSL_3(5)$&$-$\\

$5$ & $276$&$4$&no&$10$&\ref{5^6.276}&$\PSL_3(25)$&\cite[Thm. 4.5.1]{clelland2007saturated}\\

$5$ & $609$    & $4$  & no           & $8$      & \ref{5^6.609}     &  $\PSL_4(5)$     & $-$     \\

$5$ & $616$    & $4$  & no           & $5$      & \ref{5^6.616}     &  $\PSU_4(5)$     & \cite{moncho2018fusion}      \\

$5$ & $630$ & $4$ & yes & $5$ & \ref{5^6.630}  & $B(5,6;0,0,0,0)$ & \cite{oliver2017reduced}\\

$5$ & $631$&$5$&yes&$37$& \ref{5^6.631} &$B(5,6;0,0,1,0)=5 \wr 5$ &\cite{Oliver2014}, \cite{craven2017reduced}\\

$5$ & $632$&$4$&yes&$5$&\ref{5^6.632} &$B(5,6;0,0,2,0)$&\cite{oliver2017reduced}\\

$5$ & $633$&$4$&yes&$5$&\ref{5^6.633} &$B(5,6;0,0,3,0)$&\cite{oliver2017reduced} \\

$5$ & $634$&$4$&yes&$5$&\ref{5^6.634} &$B(5,6;0,0,4,0)$&\cite{oliver2017reduced}\\

$5$ & $636$&$4$&no&$1$&\ref{5^6.636} &$B(5,6;0,1,0,0)$&\cite{grazian2018fusion}\\

$5$ & $639$&$4$&no&$1$&\ref{5^6.639}&$B(5,6;0,1,1,0)$&\cite{grazian2018fusion}\\

$5$ & $640$&$4$&no&$1$&\ref{5^6.640}&$B(5,6;0,1,2,0)$&\cite{grazian2018fusion}\\

$5$ & $641$&$4$&no&$1$&\ref{5^6.641}&$B(5,6;0,1,3,0)$&\cite{grazian2018fusion}\\

$5$ & $642$&$4$&no&$1$&\ref{5^6.642}&$B(5,6;0,1,4,0)$&\cite{grazian2018fusion}\\

$5$ & $643$&$4$&no &$5$&\ref{5^6.643}& $\mathrm G_2(5)$& \cite[Thm. 5.1]{parker2018fusion} \\
                          \hline
\end{tabular}
\end{table}
 \end{Thm}

In Section \ref{s:conclude} we make a general conjecture about the list of groups of order $p^6$ which support saturated fusion systems $\F$ with $O_p(\F)=1$ and $O^p(\F)=\F$.

\begin{Rem} The saturated  fusion systems on  $\texttt{SmallGroup($3^6,99$)}$ and  $\texttt{SmallGroup($3^6,100$)}$ in Theorem \ref{Thmp^6} first alerted us to a mistake in the main result of \cite{DiazRuizViruel2007}. These fusion systems are the first in an infinite family constructed in \cite{parker2018fusion2}.
\end{Rem}

\subsection{Groups of order $3^7$}

\begin{Thm}\label{Thmp^7} Suppose that $\F$ is a saturated fusion system on a $3$-group of order $3^7$ with $O_3(\F)=1$ and $\F= O^3(\F)$. Then $\F$ is isomorphic to  one of the $88$ saturated fusion systems listed in the following table:

\begin{table}[H]
\renewcommand{\arraystretch}{1.4}
\begin{tabular}{|c|c|c|c|c|c|c|}
\hline
group \# & sec. rank & ab. ind. $p$? & \# s.f.s & reference & type  & citation(s) \\ \hline
$366$     & $3$  & yes           & $2$      & \ref{3^7.366}     &  $\PSL^\pm_4(q)$, $\nu_3(q \mp 1)= 2$ &  \cite{craven2017reduced}        \\
$386$     & $2$  & yes           & $7$      & \ref{3^7.386}     & $B(3,7;0,0,0,0)$             & \cite{DiazRuizViruel2007}, \cite{Oliver2014}                        \\
$2007$     & $5$  & no        & $3$      & \ref{3^7.2007}     &  $\mathrm{Suz}, \mathrm{Ly}$  & $-$  \\
$8705$     & $5$  & no            & $30$      & \ref{3^7.8705}     &  $\PSL_3(3)\times \PSp_4(3)$      & $-$ \\
$8707$     & $4$  & no            & $10$      & \ref{3^7.8707}     &  $\PSL_3(3) \times B(3,4;0,0,2,0) $     & $-$ \\

$8709$     & $4$  & no            & $34$      & \ref{3^7.8709}     &  $\PSL_3(3) \times B(3,4;0,0,0,0) $     & $-$ \\

$8713$    & $5$  & no            & $1$      & \ref{3^7.8713}     & $\mathrm{\mathrm{P\Gamma L}_6(4)}$           & $-$ \\
$9035$    & $5$  & no            & $1$      & \ref{3^7.9035}     &   $\mathrm {Co}_3$      & $-$
                          \\
                          \hline
\end{tabular}
\end{table}
\end{Thm}

\subsection{Isolated results}

In this short section we mention some known results for which we have been able to provide computer verification.

The following pair of fusion systems was constructed in \cite{HenkeShpectorov2015}. We confirm that these examples are saturated.

\begin{Thm}[{{\cite[Theorem A]{HenkeShpectorov2015}}}]\label{ellen}
Suppose that  $S$ is a  Sylow $3$-subgroup of $\PSp_4(9)$ and $A= J(S)$ is the Thompson subgroup of $S$. There is an elementary abelian subgroup $E$ of order $81$ such that $S=AE$ and a saturated fusion system $\F$ on $S$ in which $A$ and $E$ are $\F$-essential and $\Aut_\F(A) \cong 2.\PSL_3(4).2^2$ and $\Aut_\F(E)\cong  (8\circ \SL_2(9)).2= \SL_2(9).\mathrm Q_8$. Moreover $O^3(\F)=\F$,  $O_3(\F)=1$ and $O^{3'}(\F)$ has index $2$ in $\F$.
\end{Thm}

\cite[Theorem A]{HenkeShpectorov2015} indicates that there are several other saturated fusion systems $\F$ on a Sylow $3$-subgroup of $\PSp_4(9)$ for which $O^3(\F)=\F$ and $O_3(\F)=1$.  Apart from examples realized by overgroups of $\PSp_4(9)$ in its automorphism group, there  are examples in which $\E_\F=\{E^S\}$ and examples in which $\E_\F=\{A\} \cup \{E^S\}$ and $A$ is the 3-dimensional orthogonal module for $\PSL_2(9)$ of order $3^6$. These latter examples were first constructed by Clelland--Parker in \cite{ClellandParker2010}.

We have computationally verified the known theorems below.

\begin{Thm}[{{\cite[Theorems 4.1, 4.3, 5.1]{AOV2017}}}]\label{2^8}  All saturated fusion systems $\F$ with $O_2(\F)=1$ and $O^2(\F)=\F=O^{2'}(\F)$ on  $2$-groups of order at most $2^8$ are realizable.
\end{Thm}

For the execution of Theorem~\ref{2^8}, we need to exploit  \cite[Theorem B]{Oliver2013b} to handle $\texttt{SmallGroup($2^8,55683$)} \cong \Dih(8)\times \Dih(8) \times 2 \times 2$. As expected, our routine  locates just four proto-essential subgroups, but these subgroups support a daunting $268435456=2^{28}$  automizer sequences.
\begin{Thm}[{{\cite[Theorem 7.8]{OliverVentura2009}}}] \label{co3}
 Suppose that $\F$ is a saturated fusion system on a Sylow $2$-subgroup $S$ of $\Co_3$ with $\F= O^2(\F)$ and $O_2(\F)=1$. Then $\F= \F_{S}(\Co_3)$ or $\F= \mathrm{Sol}(3)$  is the smallest Solomon $2$-fusion system.
\end{Thm}

\begin{Thm}[{{\cite[Theorem 1.1]{parker2018fusion}}}]\label{g27}
 Suppose that $\F$ is a saturated fusion system on a Sylow $7$-subgroup $S$ of $\Gee_2(7)$ with  $O_7(\F)=1$. Then $\F$ is one of the 28 fusion systems listed in \cite[Table 5.1]{parker2018fusion} or $\F_S(\Gee_2(7))$.
\end{Thm}

The saturated fusion systems $\F$ with $O_p(\F)=1$ on a Sylow $p$-subgroup $S$ of $\Gee_2(p)$ are determined in  \cite{parker2018fusion} where it is shown that for $p \ge 11$ only $\F_{S}(\Gee_2(p))$ occurs. When $p \in\{5,7\}$ there are, in addition, fusion systems realized by sporadic simple groups; when $p=7$ we also obtain exotic examples. Theorems \ref{Thmp^6} and Theorem \ref{g27} give an independent confirmation of the calculations in \cite{parker2018fusion} of all the exceptional cases when $p \le 7$.

\section{Concluding observations}\label{s:conclude}
 In our examples, we observe that whenever $\F$ is a saturated fusion system on a direct product $S=S_1 \times S_2$ with $S_i$ non-abelian and $\F=O^{p'}(\F)$, $\F$ splits as a direct product $\F_1 \times \F_2$ with $\F_i$ a fusion system on $S_i$. {In \cite[Theorem A]{Oliver2013b},  Oliver considers saturated fusion systems on $2$-groups $S = S_1 \times S_2$ where $S_1$ and $S_2$ are non-trivial and  not decomposable  as direct products. Provided, for $j \in\{1,2\}$, $\Omega_1(Z(S_j)) \le S_j'$ and $S_{3-j}$ does not contain a subgroup isomorphic to $S_j\times S_j$, he demonstrates that a saturated fusion system $\F$ on $S$ with $\F=O^{2'}(\F)=O^2(\F)$ decomposes as $\F_1 \times \F_2$ where $\F_j$ is a saturated fusion system on $S_j$.
  If there is to be a generalization of Oliver's  result to odd primes, then the following example indicates that one would need to impose additional hypotheses:

\begin{Ex}Suppose that $p$ is odd.
A Sylow $p$-subgroup of $G=\Sym(2p^2)$ is isomorphic to $S\cong p\wr p \times p\wr p$; however $\F=\F_S(G)$   is not it a direct product of two fusion systems each on $p \wr p$. Indeed, the base group $E$ of $S$ of order $p^{2p}$ is elementary abelian and $\F$-essential with $\Aut_\F(E) \cong (p-1) \wr \Sym(2p)$ and as $E$ does not contain either direct factor, this example shows that \cite[Lemma 1.11(b)]{Oliver2013b} does not extend to any odd prime. We also remark that this is an example with $|\Out_S(E)|>p$ and $O_{p'}(\Out_\F(E))$ not centralized by $O^{p'}(\Out_\F(E))$.
\end{Ex}

Next recall that in  \cite{grazian2018fusion}, Grazian defines an \emph{$\F$-pearl} of a saturated fusion system $\F$ to be an $\F$-essential subgroup   which is either of order $p^2$ or non-abelian of order $p^3$.
We have the following remark  which pertains to $\F$-pearls  and particularly to the fusion systems in \ref{5^6.636} to \ref{5^6.642}.  Suppose that $S$ is a maximal class $p$-group. Then the \emph{$2$-step centralizer} is defined to be  $\gamma_1(S):=C_S(S'/[S,S,S,S])$. Suppose that $S$ has a self-centralizing subgroup $P$ of order $p^2$ with $P \not\le \gamma_1(S)$ and an element $\alpha \in \Aut(S)$ of order $p-1$ which leaves $P$ invariant and induces an element of determinant $1$ in $\Aut(P)\cong \GL_2(p)$. Form the fusion datum $\D =(\Q,\A)$ with $\mathcal Q= (S,P)$, $\A(S)= \langle \Inn(S), \alpha\rangle$ and $\A(P) =\SL_2(p)$. Then $\F=\F(\D)$ is saturated with $O_p(\F)= 1$ and $\mathfrak{foc}(\F)= [S,\alpha]$. A similar construction can be performed with non-abelian pearls. We speculate that many of the maximal class $p$-groups have such an automorphism \cite{DE17}. Indeed, $\F$-pearls can be attached to any saturated fusion system on a maximal class $p$-group which  has a class $P^\F$ of $\F$-centric elementary abelian subgroups of order $p^2$ under the described conditions.

We have the following theorem which is inspired by \cite[Lemma 3.7]{grazian2018fusion}:

\begin{Thm}\label{shrink} Suppose that $\F$ is a saturated fusion system on $S$ and $E$ is an $\F$-essential subgroup which is not   contained in any other $\F$-essential subgroup. Let $A= N_{\Aut_\F(S)}(E)$ and $C$ be a complement to $N_{\Inn(S)}(E)$ in $A$.
Define $S_1= N_S(E)$ and, for $i> 1$, $S_{i}= N_{S}(S_{i-1})$.
For $i \ge 1$, $C$ leaves $S_i$ invariant and induces $C_i\le \Aut(S_i)$. Define $\F_i$ to be the fusion system on $S_i$ given by  $\langle  \Inn(S_i)C_i, \Aut_\F(E)\rangle$.  If no proper subgroup of $E$ is $S_1$-centric, then, for each $i \ge 1$, $\mathcal F_i$ is saturated.
\end{Thm}

\begin{proof} Suppose that $i \ge 1$. Since $C$ leaves both $S$ and $S_1$ invariant it also leaves each $S_i$ invariant. Set $G_i= S_i\rtimes C_i$ and let $K= \Out_{G_i}(E)$ and $\Delta=\Out_\F(E)$. Then $G_i, K$ and $\Delta$ satisfy the hypothesis of \cite[Proposition 5.1]{BrotoLeviOliver2006}. It follows that $\F_i$ is saturated.
\end{proof}

\begin{Ex}\label{ex:prune}
We apply Theorem \ref{shrink} to some of our examples and make further observations.
\begin{enumerate}
 \item Let $G= \M$ be the monster finite simple group, $S \in \Syl_7(G)$ and $\F= \F_S(G)$.  Then $\F$ has an $\F$-pearl $P$ of order $49$ (\cite[Theorem 5.1]{parker2018fusion}).   Hence Theorem~\ref{shrink} implies that there are saturated fusion systems $\F_2$, $\F_3$ on groups $S_2$ and $S_3$. As $P$ is abelian, $O_7(\F_2)=O_7(\F_3)=1$.  In fact we have $\F_2=\F(7^4,7,1)$ and $\F_3=\F(7^5,37,1)$ is the Grazian fusion system; both are exotic.
 \item  Let $p$ be an odd prime,  $G= \PGL_p(q)$ and $S \in \Syl_p(G)$ where $p  \mid q-1$.  Set $\F= \F_S(G)$. $S$ is a maximal class group with a maximal abelian subgroup $A < S$ and $\E_\F=\{A,P\}$ where $P$ is an $\F$-pearl of order $p^2.$ We are thus in a situation where Theorem \ref{shrink} applies with $E=P$
and we obtain a string of saturated fusion systems $\F_i \subseteq \F$  on $p$-groups $S_i$,  all of which have an abelian subgroup of index $p$.
 \item  Let $\F$  be one of the exotic fusion systems in \cite{ParkerStroth2015}. Then $|S|=p^{p-1}$ and we obtain saturated fusion systems $\F_i$,  $2\le i\le p-3$ on $p$-groups of order $p^{i+2}$. The group $S$ has maximal class and has $2$-step centralizer $Q$, an extraspecial subgroup  of order $p^{p-2}$.  We have $|S_i \cap Q| =p^{i+1}$ and this is abelian if and only if $i\le (p-3)/2$. We speculate that almost all of the fusion systems $\F_i$ are exotic.
     \end{enumerate}
\end{Ex}

Recall the definition of $H_\F(P)$ for $P$ fully $\F$-normalized in $ S$ from the discussion before Lemma~\ref{l:davidprop}. The next two lemmas demonstrate that we can remove certain subgroups while preserving saturation.

\begin{Lem}\label{prune}
Suppose  that $\F$ is a saturated fusion system on  $S$ and $P$ is an $\F$-essential subgroup of $S$. Let $\mathcal C$ be a set of $\F$-conjugacy class representatives of $\F$-essential subgroups with $P\in \mathcal C$. Assume $P$ has the minimality property:
$$\mbox{if $Q<P$, then $Q$ is not $S$-centric. }$$
If $H_\F(P) \le K \le \Aut_\F(P)$, then
$$\mathcal G = \langle \Aut_\F(S),K, \Aut_\F(E) \mid E \in \mathcal C\setminus \{ P\}\rangle$$ is saturated. Furthermore,  $P$ is $\G$-essential if and only if $K > H_\F(P)$ and in this case $\Aut_\G(P)=K$.
\end{Lem}

\begin{proof} To prove that $\G$ is saturated, it suffices to show that every $\G$-centric subgroup is $\G$-saturated.  Let $T \le S$ be $\G$-centric.  Then the minimality property of $P$  implies that $T^\mathcal G= T^\mathcal F$. In particular, it follows that the set of $\F$-centric subgroups coincides with the set of $\G$-centric subgroups.

As $\F$ is saturated, there exists $R \in T^\F$ such that $R$ is fully $\F$-automized  and $\F$-receptive.  Because $T^\F= T^\mathcal G$, $R \in T^\mathcal G$.
Since $\Aut_\mathcal G(R) \le \Aut_\F(R)$ and  $\Aut_S(R) \le \Aut_\mathcal G(R)$, the fact that $\Aut_S(R) \in \Syl_p( \Aut_\F(R))$ implies that $\Aut_S(R) \in \Syl_p(\Aut_\mathcal G(R))$. Hence $R$ is fully $\mathcal G$-automized.

Assume that $Q \in R^\F= R^\mathcal G$ and $\theta \in \Hom_\mathcal G(Q,R)$. As $R$ is $\F$-receptive and  $\theta\in \Hom_\F(Q,R)$,  there exists an extension $\widetilde \theta \in \Hom_\F(N_\theta,S)$. We need to show that some such $\widetilde \theta$ can be found in $\Hom_\G(N_\theta,S)$. If $N_\theta=Q$, then we take $\widetilde \theta = \theta \in \Hom_\G(N_\theta,S)$ and there is nothing further to do.   Hence $N_\theta > Q$.
If $|N_\theta|> |P|$, then $\widetilde  \theta $ is a product only of maps from $\Aut_\F(S)$ and $\Aut_\F(E)$ with $E \in \mathcal C\setminus\{P\}$ by the Alperin-Goldschmidt Theorem \ref{t:alp}, and so $\widetilde \theta \in \Hom_\G(N_\theta,S)$ in this case. Thus $|R| < |N_\theta | \le |P|$. Assume that $\widetilde \theta \not \in \Hom_\G(N_\theta,S)$. Then $\widetilde \theta = \alpha_1   \alpha_2$ where $\alpha_1 \in \Hom_\G(N_\theta, P)$ and  $\alpha _2 \in \Hom_\F(N_\theta \alpha_1, S)$.  In particular, $Q\alpha_1 < N_\theta \alpha_1 \le P$.  Thus the minimality of $P$ now contradicts $R$ being $\G$-centric.  We deduce that $\widetilde \theta\in \Hom_\G(N_\theta, S)$. Hence $R$ is $\G$-receptive and this means that $T$ is $\G$-saturated. Using  \cite[Theorem I.3.10]{AschbacherKessarOliver2011} we have that $\G$ is saturated. As $P$ is fully  $\F$-normalized, by \cite[Proposition I.3.3 (b)]{AschbacherKessarOliver2011} $H_\F(P)=\Aut_\G(P)$ if and only if $P$ is not $\G$-essential.

 Let $$\G_0=    \langle \Aut_\F(S), H_\F(P),\Aut_\F(E) \mid E \in \mathcal C\setminus \{ P\}\rangle=   \langle \Aut_\F(S),\Aut_\F(E) \mid E \in \mathcal C\setminus \{ P\}\rangle.$$  Then $\G_0$ is saturated, $P$ is not $\G_0$-essential and $\Aut_{\G_0}(P)= H_\F(P)$. Assume that $K> H_\F(P)$. Then $P$ is $\G$-essential.
Suppose that $\theta \in \Aut_\G(P)$. Then $\theta$ is a composition  of maps from $\Aut_\F(S)$, $\Aut_\F(E)$, $E \in \mathcal C\setminus\{P\}$ and $K$. Thus $$\theta = \kappa_1 \alpha_1 \kappa_2\alpha_2\dots $$ with $\kappa_i\in K$  and $\alpha_i \in \Aut_{\G_0}(P)=H_\F(P)\le K$ and so $\theta \in K$. Hence $\Aut_\G(P)=K$.
\end{proof}

\begin{Lem}\label{pearlprune} Assume that $\F$ is a saturated fusion system on a $p$-group $S$, $\mathcal C$ is a set of $\F$-conjugacy class representatives of $\F$-essential subgroups, and $P \in\mathcal C$. If $P$ is an $\F$-pearl, then $\G=\langle \Aut_\F(S), \Aut_\F(E)\mid E \in \mathcal C\setminus\{P\}\rangle$ is saturated.
\end{Lem}

\begin{proof}

If $P$ is abelian, then Lemma~\ref{prune} gives the result.  Assume that $P$ is extraspecial. We claim that $P$ is not properly contained in any $\F$-essential subgroup. If $p$ is odd, then this follows from \cite[Theorem 3.6]{grazian2018fusion}. If $p=2$, then $P$ is quaternion of order $8$ and so $\Aut_\F(P)=\Aut(P) \cong \Sym(4)$ and $|\Out_S(P)|=2$ as $P$ is $\F$-essential.
  In particular, $\Aut_{Z_2(S)}(P)$ is contained in the centre of $\Aut_S(P)\cong \Dih(8)$.  Hence $Z_2(S) <P$ and  so $Z_2(S)$ has order $4$ and $|S:C_S(Z_2(S))|=2$.
Let $x \in P\setminus Z_2(S)$. Then $x \not \in C_S(Z_2(S))$ and $  C_{S}(x) \cap C_S(Z_2(S))= C_S(P) = Z(S)$.  As $|S:C_S(Z_2(S))|=2$ and $|Z(S)|=2$, it follows that $|C_S(x)|=4$. Now \cite[Lemmas 10.24 and 10.25]{GLS2} imply that $S$ is either quaternion or semidihedral. Since $P$ is non-abelian, any proper subgroup which properly contains $P$ is quaternion of order at least $16$. Such groups cannot be $\F$-essential as their automorphism groups are $2$-groups.

Suppose that $T$ is $\G$-centric. Assume $T^\G$ does not contain a proper subgroup of $P$.  Then $T^\F= T^\G$ and we know $\Aut_\G(T)= \Aut_\F(T)$ so long as $P \not \in T^\G$. In particular, $\Aut_\G(R)= \Aut_\F(R)$ for $R>T$ and hence the $\G$-surjectivity property holds for $T$. Therefore we may and do assume that $T < P$. Thus $|T|=p^2$ and, as  $T$ is $\G$-centric, $C_S(T)= T$, $P= N_S(T)$ and $T$   is fully $\G$-normalized. Since $P$ is $\F$-essential, $T$ is $\F$-conjugate to $Z_2(S)$ and so $T$ is not fully $\F$-normalized and in particular is not $\F$-essential. In addition, as $P= N_S(T)=TZ_2(S)$,   the only $\F$-essential subgroup containing $T$ is $P$.
  Indeed, if $E >  T$ is $\F$-essential, then $P=N_S(T)\ge  N_E(T)> T$ and so $P= N_E(T) \le E$. As $P$ is not properly contained in any $\F$-essential subgroup, we conclude that $P=E$.
    Hence  $$\Aut_\G(T)= \{ \phi|_T \mid \phi  \in N_{\Aut_\F(S)}(T)\}. $$
 Furthermore, for $\phi\in N_{\Aut_\F(S)}(T)$, we have $$P\phi = \langle T\phi, Z_2(S)\phi\rangle= \langle T,Z_2(S)\rangle = P.$$It follows that the restriction map $N_{\Aut_\G(P)}(T)\rightarrow N_{\Aut_\G(T)}(\Aut_P(T))$ is a surjection.  In particular, $T$ has the $\G$-surjectivity property and so we conclude that $\G$ is saturated by Theorem~\ref{t:surjprop}.
\end{proof}

\begin{Ex} Let $G= \mathrm E_8(2)$, $S \in \Syl_5(G)$ and $\F= \F_S(G)$.  Then $\F= \F(5^5,30,20)$ and the $\F$-conjugacy classes of $\F$-essential subgroups have representatives  $A$, which is elementary abelian, $E_0$ and $E_1$ which are extraspecial of order $5^3$ and exponent $5$. We have $\Out_\F(A) \cong (4\circ 2^{1+4}_+).\Sym(6) $ which is isomorphic to the complex reflection group $G_{31}$. As a Sylow $5$-subgroup $\Aut_S(A)$ of $\Aut_\F(A)$ is cyclic, the strongly $p$-embedded subgroup $H_A$ of $\Aut_\F(A)$ is just the normalizer of $\Aut_S(A)$. The over-groups in $G_{31}$ of $H_A$ are $4\circ 2^{1+4}_+.\Sym(5)_a$, $4\circ 2^{1+4}_+.\Sym(5)_b$, $4\circ 2^{1+4}_+:\Frob(20)$, $\GL_2(5)$ acting $(2/2)$ (acting with two non-trivial composition factors on $A$), $4 \times \Sym(5)$ one acting $(3/1)$ the other $(1/3)$ $H_A$ and $\Aut_\F(A)$. Our stipulation that the fusion systems has no normal $5$-subgroup then delivers, by Lemma~\ref{prune} (we say ``pruning" at $A$) fusion systems $\F(5^5,30,15) $, $\F(5^5,30,17)$, $\F(5^5,30,18)$, $\F(5^5,30,19)$  and $\F(5^5,30,21)$. The remaining two systems $\F(5^5,30,16)$ and $\F(5^5,30,22)$  cannot be obtained in this way and we see that $ \F(5^5,30,22)$ prunes to give $\F(5^5,30,16)$.  Since $E_0$ and $E_1$ are extraspecial, Lemma~\ref{pearlprune} applies and $E_0$ as well as $E_1$ can also be pruned individually and together. This gives the fusion systems $\F(5^5,30, 33)$ through $\F(5^5,30,48)$. In this way we see how a group fusion system can be pruned to deliver a plethora of exotic systems and explain many of the systems that appear on a given group.
\end{Ex}

Together, Theorem~\ref{shrink}, and Lemmas~\ref{prune} and \ref{pearlprune} allow us to construct new saturated fusion systems from fusion systems of groups.

Based on the computations in this paper and observations about local structure in finite simple groups, we make the following conjecture:

\begin{Conj}\label{c:p6} Suppose that $p \ge 5$, $S$ is a $p$-group of order $p^6$ and $\F$ is a saturated fusion system on $S$ with $O_p(\F)=1$ and $O^p(\F)=\F$.Then either $S$ has maximal class or $S$ is a Sylow subgroup of $\PSL_3(p) \times \PSL_3(p)$, $\PSL_3(p^2)$, $\PSL_4(p)$ or $\PSU_4(p)$. Furthermore, if $p \ge 11$ and  $S$ has maximal class, then either $S$ is a Sylow $p$-subgroup of $\Gee_2(p)$ or $S$ has an abelian subgroup of index $p$ (perhaps even $S\cong   B(p,6;0,0,0,0)$)  and $\F$ is obtained by pruning one of the fusion systems  in \cite{ClellandParker2010} which has  $\Aut_\F(\gamma_1(S)) \cong \PSL_2(p)$  acting irreducibly on $\gamma_1(S)$.
\end{Conj}

To prove the first part of Conjecture \ref{c:p6}, by the results of the present paper we may assume that $p \ge 7$. We may also assume that $S$ has no abelian subgroup of index $p$ and by \cite[Theorem C]{grazian2018fusion2} that the sectional rank of $S$ is at least $4$. By \cite{moncho2018fusion}, we can also assume that $S$ does not have an extraspecial subgroup of index $p$ as otherwise $S$ is one of the groups listed. By   \cite[Lemma 1.5]{grazian2018fusion}, we may assume that $\F$ has no $\F$-pearls, for otherwise $S$ has maximal class. If $E$ is $\F$-essential and $|\Out_S(E)| \ge p^2$, then Lemma~\ref{V bound} implies that $E$ is elementary abelian of order $p^4$, is normal in $S$ and $|\Out_S(E)|= p^2$. Since a Sylow $p$-subgroup of $\GL_4(p)$ has exponent $p$, this leads to $O^{p'}(\Aut_\F(E)) \cong \SL_2(p^2)$ or $\PSL_2(p^2)$ using \cite[Theorem 7.6.1]{GLS3}. In the first case, this means that $S$ is a Sylow $p$-subgroup of $\PSL_3(p^2)$ and in the second that $S$ is isomorphic to a Sylow $p$-subgroup of $\PSU_4(p)$. Hence $|\Out_S(E)|=p$ for all $\F$-essential subgroups $E$ of $\F$.

To prove the second part of Conjecture \ref{c:p6}, results in preparation by Grazian and Parker \cite{GrazianParker} show that, provided $S$ is not isomorphic to the Sylow $p$-subgroup of $\Gee_2(p)$, the $\F$-essential subgroups are either $\F$-pearls or perhaps  $\gamma_1(S)$. Furthermore, if $\gamma_1(S)$ is non-abelian, then the $\F$-essential subgroups are all abelian $\F$-pearls. So this part of the conjecture boils down to an inspection of the isomorphism types of maximal class $p$-groups with an abelian maximal subgroup which support a fusion system, and the determination of the maximal class $p$-groups which support a fusion system with an abelian $\F$-pearl.

 The fusion systems listed in Tables~\ref{5^6.630t} and ~\ref{5^6.632t} show that the second part of the conjecture does not hold when $p=5$. In fact the  table in Theorem~\ref{Thmp^6}  could lead us to speculate that the Blackburn groups $B(p,6;0,0, i,0)$ and $B(p,6;0,1,i,0)$ where $0\le i \le p-1$ all support saturated fusion systems with $O_p(\F)=1$ and $O^p(\F)=\F$. However,  we don't believe this is the case even when $p=7$. We have computed that the maximal class $7$-groups which have a saturated  fusion system $\F$ with an abelian $\F$-pearl  are $B(7,6;0,0,0,0)$ and the Sylow $7$-subgroup of $\Gee_2(7)$. For $p=7$, the group fusion system on $\PSL_7(8)$ on one of its Sylow $7$-subgroups cannot be obtained from  one of the systems in  \cite{ClellandParker2010}. This is the origin of the assumption that $p \ge 11$.

We make the following conjecture concerning Lie type groups in their defining characteristic:

\begin{Conj}\label{c:lierk3}
 Suppose that $p$ is a prime and $\mathcal S$ is the collection of pairs $(G,S)$ where $G$ is a simple Lie type group  defined in characteristic $p$ which is not isomorphic to $\PSp_4(p^a)$ and $S \in \Syl_p(G)$. Then, for all but finitely many exceptions, if $(G,S) \in \mathcal S$ and $\F$ is a saturated fusion system on $S$ with   $O_p(\F)=1$, then $\F = \F_S(H)$ for some $G \le H \le \Aut(G)$.
\end{Conj}

To prove Conjecture \ref{c:lierk3} we can assume that both the rank of $G$ and the field are large, although it would be desirable to have a complete list of exceptions.

The work of Henke--Shpectorov \cite{HenkeShpectorov2015} and Clelland--Parker \cite{ClellandParker2010} shows that for $p$ odd the Sylow $p$-subgroup of $\PSp_4(p^n)$ supports exotic fusion systems, and so this restriction is required in Conjecture \ref{c:lierk3}. All the other known sporadic groups or exotic fusion systems on such $p$-groups only occur for small values of $p$. For example, the conjecture holds if $G=\Gee_2(p), \PSU_4(p), \SL_3(p^n)$ by \cite{parker2018fusion, moncho2018fusion, clelland2007saturated}.  For $q=p^n> p$, Conjecture~\ref{c:lierk3} holds for $\mathrm U_4(q)$ and $\Gee_2(q)$ by work of van Beek \cite{martin}. Further, evidence for the validity of Conjecture \ref{c:lierk3} is provided by \cite{onofrei2011saturated, PPSS19}.

 \appendix

\section{A compendium of saturated fusion systems on small $p$-groups of odd order}\label{s:appendix}

Throughout this appendix $p$ is odd. We give detailed descriptions of all the fusion systems appearing in Theorems~\ref{Thmp^4}, \ref{Thmp^5}, \ref{Thmp^6} and \ref{Thmp^7}.

\subsection{Notation}\label{notation}
Here we introduce some notational conventions to describe fusion systems on certain families of groups.

 Suppose first that $S$ is $p$-group which possesses a unique abelian subgroup $A$ of index $p$ and let $\F$ be a saturated fusion system on $S$. Define $A_0= Z(S)[S,S]$. Assuming that $|Z(S)\cap [S,S]|=p$, by  \cite[Lemma 2.2(d)]{craven2017reduced} we  have $|A:A_0|=p$  and we may choose $\a \in A \backslash A_0$ and $\x \in S \backslash A$ so that $A_0\langle \x \rangle$ and $[S,S]\langle \a \rangle$ are each normalized by $\Aut_\F(S)$. We may also choose $\x$ to have order $p$ if some element of $S \backslash A$ does. For each $i=0,1,\ldots, p-1$ we set (cf.  \cite[Notation 2.4]{craven2017reduced}): $$V_i=Z(S)\langle \x\a^i \rangle \mbox{ and } E_i=Z_2(S)\langle \x\a^i \rangle. $$

We adopt similar conventions when $S$ is a maximal class $p$-group.  The metabelian $p$-groups of maximal class $r-1$  which satisfy $[\gamma_1(S),\gamma_2(S)] \le \gamma_{r-2}(S)$ have been classified by Blackburn in \cite[Section 4, pages 82, 83]{blackburn1958special}. For $r \ge 4$, and $\alpha,\beta, \gamma, \delta \in \{0,1,2,\ldots, p-1\}$, define $$B(p,r; \alpha,\beta, \gamma, \delta)=\langle s, s_1,\dots,s_{r-1}\mid \textbf{R1},\textbf{R2}, \textbf{R3}, \textbf{R4}, \textbf{R5}, \textbf{R6},\textbf{R7}\rangle$$ where, defining $s_r=s_{r+1}=\cdots = s_{r+p-2}=1$,
 the relations are as follows:
\begin{enumerate}
\item[\textbf{R1}:] $s_i=[s_{i-1},s]$ for $i\in\{2, \dots, r\}$;
\item[\textbf{R2}:] $[s_1,s_2]=s_{r-2}^\alpha s_{r-1}^\beta$;
\item[\textbf{R3}:] $[s_1,s_3]=s_{r-1}^\alpha$;
\item [\textbf{R4}:] $[s_1,s_i]=1$ for $i\in\{4, \dots, r-1\}$;
\item [\textbf{R5}:] $s^p=s_{r-1}^\delta$;
\item [\textbf{R6}:] $s_1^p s_2^{p \choose 2} s_{3}^{p \choose 3} \cdots s_p = s_{r-1}^\gamma$;
\item [\textbf{R7}:] $s_i^p s_{i+1}^{p \choose 2} s_{i+2}^{p \choose 3} \cdots s_{i+p-1} = 1$ for $i \in \{2, \dots, r-1\}$.
\end{enumerate}

In this case, for $i\in \{0,1,2,\ldots,p-1\}$, we set:
$$V_i=\langle s s_1^i,s_{r-1} \rangle \mbox{ and } E_i=\langle s s_1^i,s_{r-2},s_{r-1} \rangle.$$ Note that we may take $\x=s$ and $\a=s_1$ when $\langle s_1, s_2 \ldots s_{r-1} \rangle$ is abelian. For maximal class $p$-groups $S$, we define $\gamma_2(S)=[S,S]$ and for $j \ge 3$, $\gamma_j(S)=[\gamma_{j-1}(S),S]$. The subgroup $\gamma_1(S)$ defined by $C_S(\gamma_2(S)/\gamma_4(S))$ is a $2$-step centralizer. In \cite{blackburn1958special} Blackburn provides presentations for metabelian $p$-groups $S$ of maximal class $r-1$  which satisfy $[\gamma_1(S),\gamma_2(S)] \le \gamma_{r-2}(S)$. Especially, he gives a complete classification of all maximal class $3$-groups and all maximal class groups of order at most $p^6$.

 As the size of $\Out_\F(S)$ grows, some of the subgroups become $\F$-conjugate.  We indicate that $X$ and $Y$ are $\F$-conjugate  by $X\sim_\F Y$.

In the description of the groups in our tables, we use the following notation which is similar to that given in \cite{ATLAS}. The symmetric group of degree $n$ is $\Sym(n)$, and $\Alt(n)$ is the corresponding alternating group.  We denote by $\Frob(20)$ the Frobenius group of order 20, $\Dih(n)$ and $\SDih(n)$ are the dihedral and semidihedral groups of order $n$ and, for $r$ a natural number, $r^n$ represents the homocyclic  $r$-group of order $r^n$ with $n$ suppressed if the group is cyclic. For the classical groups we use standard notation so for example $\PSU_5(4)$ is the projective special unitary group in dimension $5$ defined over the field of order 16. The groups $2^{1+4}_+$ and $2^{1+4}_-$ are the extraspecial $2$-groups of order $32$, the first one of plus type the second of minus type.
By $\mathrm{P\Gamma U}_5(4)$ we include the group of field automorphisms on $\mathrm{PGU}_5(4)$ (which contributes a cyclic group of order $4$).  The notation for the sporadic simple groups is standard. For extensions and quotients we have the following conventions.  The group $X=A^{\;.}B$ is a non-split extension with normal subgroup $A$ and $X/A \cong B$.   The group  $X=A{:}B$ is the split extension and $A.B$ is an extension of undetermined type. We write $A\wr B$ for the wreath product of $A$ by $B $ normally with transparent action. By $A \circ B$ we represent the central product of $A$ and $B$. Thus $4\circ 2^{1+4}_+ \cong 4 \circ 2^{1+4}_-$ has centre of order $4$. By $\frac{1}{n} A$ we mean an unspecified subgroup of index $n$ in $A$. From time to time we meet groups with the same outward appearance and we indicate that they are non-isomorphic by introducing subscripts.  For example the complex reflection group $G_{31}\cong (4\circ 2^{1+4}_+).\Sym(6)$ has two non-isomorphic subgroups with shape  $(4\circ 2^{1+4}_+).\Sym(5)$.  We denote one by    $(4\circ 2^{1+4}_+).\Sym(5)_a$ and the other  $(4\circ 2^{1+4}_+).\Sym(5)_b$. For  automorphism groups of $\PSU_4(3)$, we have followed {\sc Atlas} \cite{ATLAS} conventions.

In the first column of every table is a number $j$ which allows us to specify a fusion system $\F(p^n,i,j)$ on $\texttt{SmallGroup($p^n,i$)}$. The last column indicates whether or not a particular fusion system is exotic. Mostly this is completed with reference to a citation listed in the final column of the appropriate table in Section \ref{s:results}.
To describe the exotic fusion systems discovered in  \cite[Theorem 5.10]{DiazRuizViruel2007}, we use the notation  $\F_{\mathrm {DRV}}(3^n,i)$ to refer to the fusion system  called $\F(3^n,i)$ in that paper. When $|S|=3^7$, we have fusion systems on direct products $S=S_1 \times S_2$ where $S_1 \cong 3^{1+2}$ and $S_2$ is a Blackburn group, and one of the direct factors can support an exotic fusion system. In these cases,   computer calculations reveal that $S_1$ and $S_2$ are strongly closed in $\F$.
Hence  \cite[Proposition I.6.7]{AschbacherKessarOliver2011} implies that the projections  $\F_1$ and $\F_2$ on $S_1$ and $S_2$ are saturated with $\F$ isomorphic to a subfusion system of $\F_1\times \F_2$. Furthermore, we obtain $O^p(\F_i)=\F_i$ and $O_p(\F_i)=1$ for $i=1,2$ and so all our examples can be constructed from smaller cases listed in our tables. Notice that if $\F=\F_S(G)$ for some finite group $G$ then $N_\F(S_1) = \F_S(N_G(S_1))$ and $\F_2 \cong
\F_{S/S_1}(N_G(S_1)/S_1)$. In particular, if $\F_2$ is exotic, then $\F$ is exotic.

\subsection{Groups of order $p^3$} We start with possibly the most well-known result in the subject of fusion systems. If $S$ is a non-abelian $p$-group of order $p^3$ which supports a saturated fusion system $\F$ with $O_p(\F)=1$, then $S$ is extraspecial of exponent $p$.  The fusion systems are described in the celebrated paper by Ruiz and Viruel \cite[Tables 1.1 and 1.2]{RuizViruel2004}. Famously, there are three exotic systems on the group of order $7^3$.

\subsection{Groups of order $p^4$} By Theorem \ref{Thmp^4}, $S$ has an abelian subgroup of index $p$ so we may adopt the notation of Section \ref{notation}. We give detailed descriptions of the fusion systems on these groups.

\subsubsection{$\sgb(3^4,7)$}\label{3^4.7}
This group is isomorphic to a Sylow $3$-subgroup of $\Alt(9)$.
\begin{table}[H]
\caption{Saturated fusion systems on a Sylow $3$-subgroup of $\Alt(9)$}
\label{3^4.7t}
\renewcommand{\arraystretch}{1.4}
\begin{tabular}{|c|c|c|c|c|c|}
\hline
$\F$      & $\Out_\F(A)$    & $\Out_\F(V_0)$ & $\Out_\F(E_0)$ & $\Out_\F(S)$ & Example(s)                                                   \\ \hline
$ 1$    & $\Sym(4)$       & $\SL_2(3)$   & $-$               & $2$          & $\Alt(9)$                          \\
$ 2$    & $-$             & $\SL_2(3)$   & $-$               & $2$          & $-$                           \\
$ 3$ & $\Sym(4)$       & $-$          & $\SL_2(3)$        & $2$          & $ \PSU_4(2)$                       \\
$ 4$    & $2 \wr \Sym(3)$ & $-$          & $\GL_2(3)$        & $2 \times 2$          & $\mathrm{PSU}_4(2){:}2, \PSL_6(2)$   \\
$ 5$    & $2 \wr \Sym(3)$ & $\GL_2(3)$    & $-$               & $2 \times 2$          & $\Sym(9)$                     \\
$ 6$ & $-$             & $\GL_2(3)$   & $-$               & $2 \times 2$          & $-$                           \\ \hline
\end{tabular}
\end{table}

\subsubsection{$\sgb(3^4,8)$}\label{3^4.8}
This is the group  $B(3,4;0,0,2,0)$.   We indicate in the Example column how these fusion systems correspond to the exotic examples  discovered in \cite[Theorem 5.10, Table 3]{DiazRuizViruel2007}.

\begin{table}[H]
\caption{Saturated fusion systems on the group $B(3,4;0,0,2,0)$}
\label{3^4.8t}
\renewcommand{\arraystretch}{1.4}
\begin{tabular}{|c|c|c|c|}
\hline
$\F$          & $\Out_\F(V_0)$  & $\Out_\F(S)$ & Example                   \\ \hline
$ 1$         & $\SL_2(3)$                 & $2$          & $\F_{\mathrm{DRV}}(3^4,3)$                       \\
$ 2$         & $\GL_2(3)$                 & $2 \times 2$          & $\F_{\mathrm{DRV}}(3^4,3).2$                     \\ \hline
\end{tabular}
\end{table}

\subsubsection{$\sgb(3^4,9)$}\label{3^4.9}
This is the group $B(3,4;0,0,0,0)$.   Where relevant, we indicate in the Example  column how these fusion systems correspond to the exotic examples  discovered in \cite[Theorem 5.10, Table 2]{DiazRuizViruel2007}.

\begin{table}[H]
\caption{Saturated fusion systems on the group $B(3,4;0,0,0,0)$}
\label{3^4.9t}
\renewcommand{\arraystretch}{1.4}
\begin{tabular}{|c|c|c|c|c|c|c|}
\hline
$\F$      & $\Out_\F(V_0)$ & $\Out_\F(V_1)$     & $\Out_\F(V_2)$ & $\Out_\F(E_0)$ & $\Out_\F(S)$  & Example           \\ \hline
$ 1$    & $\SL_2(3)$               & $\SL_2(3)$         & $\SL_2(3)$                & $-$            & $2$          & $ \PSU_3(8)$  \\
$ 2$    & $-$                 & $\SL_2(3)$         & $\SL_2(3)$               & $-$            & $2$          & $\F_{\mathrm{DRV}}(3^4,2)$                   \\
$ 3$ & $\SL_2(3)$                 & $-$         & $-$               & $-$            & $2$          &  $\F_{\mathrm{DRV}}(3^4,1)$                   \\

$ 4$ & $-$                 & $\SL_2(3)$         & $V_1\sim_\F V_2$                & $\GL_2(3)$     & $2 \times 2$ & $^3\mathrm D_4(2)$             \\
$ 5$ & $\GL_2(3)$                 & $\SL_2(3)$ & $V_1\sim_\F V_2$               & $-$            & $2 \times 2$ & $\PSU_3(8).2$                \\
$ 6$    & $-$                 & $\SL_2(3)$         & $V_1\sim_\F V_2$                & $-$            & $2 \times 2$ & $\F_{\mathrm{DRV}}(3^4,2).2$                 \\
$ 7$    & $\GL_2(3)$                 & $-$         & $-$               & $-$            & $2 \times 2$ &  $\F_{\mathrm{DRV}}(3^4,1).2$                    \\
\hline
\end{tabular}
\end{table}

\subsubsection{$\sgb(5^4,7)$}\label{5^4.7}
This group is isomorphic to a Sylow $5$-subgroup $S$ of $\PSp_4(5)$. It has maximal   class, exponent $5$ and a subgroup $A \le S$ which is abelian of order $5^3$. We   adopt the notation of Section~\ref{notation}.

\begin{landscape}
\scriptsize
\begin{table}[]
\caption{Saturated fusion systems on a Sylow $5$-subgroup of $\PSp_4(5)$}
\label{5^4.7t}
\renewcommand{\arraystretch}{1.4}
\begin{tabular}{|c|c|c|c|c|c|c|c|c|c|}
\hline
$\F$      & $\Out_\F(A)$           & $\Out_\F(V_0)$ & $\Out_\F(V_1)$ & $\Out_\F(V_2)$ & $\Out_\F(V_3)$ & $\Out_\F(V_4)$ & $\Out_\F(E_0)$ & $\Out_\F(S)$ & Example    \\ \hline

$ 1$    & $\Sym(5)$                                           & $\SL_2(5)$       & $\SL_2(5)$ & $\SL_2(5)$ & $\SL_2(5)$ & $\SL_2(5)$       & $-$            & $4$                                       & $\PSU_5(4)$   \\

$ 2$    & $\Sym(5)$                                          & $\SL_2(5)$       & $\SL_2(5)$ & $\SL_2(5)$ & $\SL_2(5)$ & $-$       & $-$ & $4$                                       & $-$           \\

$ 3$    & $-$                    & $\SL_2(5)$       & $\SL_2(5)$ & $\SL_2(5)$ & $\SL_2(5)$ & $\SL_2(5)$       & $-$          & $4$                                       & $-$           \\

$ 4$    & $\Sym(5)$                                         & $\SL_2(5)$       & $\SL_2(5)$ & $\SL_2(5)$ & $-$ & $-$       & $-$           & $4$                                       & $-$           \\

$ 5$    & $-$                                            & $\SL_2(5)$       & $\SL_2(5)$ & $\SL_2(5)$ & $\SL_2(5)$ & $-$       & $-$           & $4$                                       & $-$           \\

$ 6$    & $\Sym(5)$                                        & $\SL_2(5)$       & $\SL_2(5)$ & $-$ & $-$ & $-$       & $-$       & $4$                                       & $-$           \\

$ 7$    & $-$                                                 &  $\SL_2(5)$       & $\SL_2(5)$ & $\SL_2(5)$ & $-$ & $-$       & $-$           & $4$                                       & $-$           \\

$ 8$    & $-$                                                  & $\SL_2(5)$       & $\SL_2(5)$ & $-$ & $-$ & $-$       & $-$           & $4$                                       & $-$           \\

$ 9$    & $\Sym(5)$                                            & $\SL_2(5)$       & $-$ & $-$ & $-$ & $-$       & $-$          & $4$                                       & $-$           \\

$ {10}$ & $-$                                                 & $\SL_2(5)$       & $-$ & $-$ & $-$ & $-$       & $-$        & $4$                                       & $-$           \\

$ {11}$ & $2 \times \Sym(5)$                                   & $\SL_2(5).2$       & $\SL_2(5)$ & $\SL_2(5)$ &  $V_2\sim_\F V_3$ &  $V_1\sim_\F V_4$        & $-$            & $4 \times 2$                              & $\PSU_5(4).2$           \\

$ {12}$ & $-$                                                  & $\SL_2(5).2$       & $\SL_2(5)$ & $\SL_2(5)$ & $V_2\sim_\F V_3$&  $V_1\sim_\F V_4$        & $-$ & $4 \times 2$                              & $-$           \\

$ {13}$ & $2 \times \Sym(5)$                                  & $\SL_2(5).2$       & $\SL_2(5)$ & $-$ & $-$ &  $V_1\sim_\F V_4$        & $-$          & $4 \times 2$                              & $-$           \\

$ {14}$ & $2 \times \Sym(5)$                                 & $-$       & $\SL_2(5)$ & $\SL_2(5)$ & $V_2\sim_\F V_3$& $V_1\sim_\F V_4$       & $-$         & $4 \times 2$                              & $-$           \\

$ {15}$ & $-$                                               & $\SL_2(5).2$       & $\SL_2(5)$ & $-$ & $-$ &  $V_1\sim_\F V_4$      & $-$        & $4 \times 2$                              & $-$          \\

$ {16}$ & $2 \times \Sym(5)$                                   & $\SL_2(5).2$       & $-$ & $-$ & $-$ & $-$       & $-$                 & $4 \times 2$             & $-$           \\

$ {17}$ & $2 \times \Sym(5)$                                 & $-$       & $\SL_2(5)$ & $-$ & $-$ &  $V_1\sim_\F V_4$       & $-$            & $4 \times 2$                              & $-$           \\

$ {18}$ & $-$                                                 & $-$       & $\SL_2(5)$ & $\SL_2(5)$ & $V_2\sim_\F V_3$ & $V_1\sim_\F V_4$        & $-$             & $4 \times 2$                              & $-$           \\

$ {19}$ & $-$                                                & $-$       & $\SL_2(5)$ & $-$ & $-$ &  $V_1\sim_\F V_4$        & $-$              & $4 \times 2$                              & $-$           \\

$ {20}$ & $-$                                                 & $\SL_2(5).2$       & $-$ & $-$ & $-$ & $-$       & $-$             & $4 \times 2$                              & $-$          \\

$ {21}$ & $\GL_2(5)/\{\pm I\}$ & $-$       & $-$       & $-$ & $-$ &  $-$       & $\SL_2(5).2$ &    $4 \times 2$                              & $\PSp_4(5)$  \\

$ {22}$ & $4 \times \Sym(5)$                                & $-$       & $\SL_2(5)$ &$V_1\sim_\F V_2$ & $V_1\sim_\F V_3$& $V_1\sim_\F V_4$       & $\GL_2(5)$       & $4 \times 4$                              & $\Co_1$     \\

$ {23}$ & $4 \times \Sym(5)$                                 & $\GL_2(5)$        & $\SL_2(5)$ &$V_1\sim_\F V_2$ & $V_1\sim_\F V_3$& $V_1\sim_\F V_4$       & $-$        & $4 \times 4$                              & $ \PSU_5(4).4$          \\

$ {24}$ & $-$                                                 & $\GL_2(5)$      & $\SL_2(5)$ &$V_1\sim_\F V_2$ & $V_1\sim_\F V_3$& $V_1\sim_\F V_4$       & $-$       & $4 \times 4$                              & $-$           \\

$ {25}$ & $4 \times \Sym(5)$                                  & $\GL_2(5)$       & $-$ & $-$ & $-$ & $-$       & $-$      & $4 \times 4$                              & $-$           \\

$ {26}$ & $4 \times \Sym(5)$                                & $-$       & $-$ & $-$ & $-$ & $-$       & $\GL_2(5)$         & $4 \times 4$                              & $\PSp_4(5).2$ \\

$ {27}$ & $-$                                                 & $-$        & $\SL_2(5)$ &$V_1\sim_\F V_2$ & $V_1\sim_\F V_3$& $V_1\sim_\F V_4$       &$\GL_2(5)$      & $4 \times 4$                              & $-$           \\

$ {28}$ & $4 \times \Sym(5)$                                  & $-$       & $\SL_2(5)$ &$V_1\sim_\F V_2$ & $V_1\sim_\F V_3$& $V_1\sim_\F V_4$       &$-$     & $4 \times 4$                              & $-$          \\

$ {29}$ & $-$                                                 & $-$       & $\SL_2(5)$ &$V_1\sim_\F V_2$ & $V_1\sim_\F V_3$& $V_1\sim_\F V_4$       &$-$    & $4 \times 4$                              & $-$          \\

$ {30}$ & $-$                                                 & $\GL_2(5)$       & $-$ & $-$ & $-$ & $-$       & $-$      & $4 \times 4$                              & $-$          \\ \hline
\end{tabular}
\end{table}
\end{landscape}

\subsubsection{$\sgb(7^4,7)$}\label{7^4.7}

This group is isomorphic to a Sylow $7$-subgroup of $\PSp_4(7)$. It has maximal   class, exponent $7$ and a subgroup $A \le S$ which is abelian of order $7^3$.  Again we use  the notation of Section \ref{notation}.

\begin{table}[H]
\caption{Saturated fusion systems on a Sylow $7$-subgroup of $\PSp_4(7)$}
\label{7^4.7t}
\renewcommand{\arraystretch}{1.4}
\begin{tabular}{|c|c|c|c|c|c|}
\hline
$\F$   & $\Out_\F(A)$         & $\Out_\F(V_0)$        & $\Out_\F(E_0)$        & $\Out_\F(S)$ & Example    \\ \hline
$ 1$ & $-$                  & $\SL_2(7)$          & $-$                 & $6$          & $-$           \\
$ 2$ & $-$                  & $\SL_2(7).2$        & $-$                 & $6 \times 2$ & $-$          \\
$ 3$ & $3 \times \PGL_2(7)$ & $-$                 & $\SL_2(7) \times 3$ & $6 \times 3$ & $\PSp_4(7)$  \\
$ 4$ & $3 \times \PGL_2(7)$ & $\SL_2(7) \times 3$ & $-$                 & $6 \times 3$ & $-$          \\
$ 5$ & $-$                  & $\SL_2(7) \times 3$ & $-$                 & $6 \times 3$ & $-$         \\
$ 6$ & $6 \times \PGL_2(7)$ & $-$                 & $\GL_2(7)$          & $6 \times 6$ & $\PSp_4(7).2$ \\
$ 7$ & $6 \times \PGL_2(7)$ & $\GL_2(7)$          & $-$                 & $6 \times 6$ & $-$          \\
$ 8$ & $-$                  & $\GL_2(7)$          & $-$                 & $6 \times 6$ & $-$          \\ \hline
\end{tabular}
\end{table}

\subsection{Groups of order $p^5$} By Corollary \ref{p^5-abelian subgroup} if  $S$ is not isomorphic with $\texttt{SmallGroup$(7^5,37)$}$, then it has an abelian subgroup of index $p$ and we may adopt the notation of Section \ref{notation}.

\subsubsection{$\sgb(3^5,26)$}\label{3^5.26} This is group $B(3,5;0,0,0,0)$ which is isomorphic to a Sylow $3$-subgroup of $\mathrm{PGU}_3(8)$. Where appropriate, we indicate in the Example column how these fusion systems correspond to the exotic examples  discovered in \cite[Theorem 5.10, Table 6]{DiazRuizViruel2007}. Here we point out a small error: $\F(3^5,26,3)$ is exotic, but labelled $3^3\mathrm D_4(2)$ in \cite[Theorem 5.10, Table 6]{DiazRuizViruel2007}.

\begin{table}[H]
\caption{Saturated fusion systems on a  Sylow $3$-subgroup of $\mathrm{PGU}_3(8)$}
\label{3^5.26t}
\renewcommand{\arraystretch}{1.4}
\begin{tabular}{|c|c|c|c|c|c|c|c|}
\hline
$\F$   & $\Out_\F(A)$ & $\Out_\F(V_0)$ & $\Out_\F(E_0)$ & $\Out_\F(E_1)$ & $\Out_\F(E_2)$ & $\Out_\F(S)$ & Example           \\ \hline
$ 1$ & $\GL_2(3)$    & $-$            & $\GL_2(3)$     & $\SL_2(3)$     & $E_1\sim_\F E_2$     & $2 \times 2$ & $^2\mathrm F_4(8)$              \\
$ 2$ & $\GL_2(3)$    & $\GL_2(3)$     & $-$            & $\SL_2(3)$     & $E_1\sim_\F E_2$     & $2 \times 2$ & $\F_{\mathrm{DRV}}(3^5,4)$                      \\
$ 3$ & $-$           & $\GL_2(3)$     & $-$            & $\SL_2(3)$     & $E_1\sim_\F E_2$   & $2 \times 2$ & $-$               \\
$ 4$ & $\GL_2(3)$    & $-$            & $\GL_2(3)$     &     $-$        &   $-$          & $2 \times 2$ & $\F_{\mathrm{DRV}}(3^5,2)$                      \\
$ 5$ & $\GL_2(3)$    & $-$            & $-$            & $\SL_2(3)$     & $E_1\sim_\F E_2$     & $2 \times 2$ & $\F_{\mathrm{DRV}}(3^5,1)$                      \\
$ 6$ & $\GL_2(3)$    & $\GL_2(3)$     & $-$            &   $-$          &   $-$          & $2 \times 2$ & $\F_{\mathrm{DRV}}(3^5,3)$                      \\
$ 7$ & $-$           & $\GL_2(3)$     & $-$            &     $-$        &   $-$          & $2 \times 2$ & $\PGL_3(19).2$  \\ \hline
\end{tabular}

\end{table}

\subsubsection{$\sgb( 3^5,51)$}\label{3^5.51}
This group is isomorphic to a Sylow $3$-subgroup of $\Alt(12)$. Thus it is isomorphic to $(3 \wr 3 ) \times 3$ and has an abelian subgroup $A$ of index $3$.  The subgroup $V_0$ is self-centralizing and elementary abelian of order $27$ and $E_0$ is isomorphic to $3 \times 3^{1+2}_+$.
\begin{table}[H]
\caption{Saturated fusion systems on a Sylow $3$-subgroup of $\Alt(12)$}
\label{3^5.51t}
\renewcommand{\arraystretch}{1.4}
\begin{tabular}{|c|c|c|c|c|c|}
\hline
$\F$      & $\Out_\F(A)$                   & $\Out_\F(V_0)$        & $\Out_\F(E_0)$        & $\Out_\F(S)$          & Example  \\ \hline
$ 1$    & $\frac{1}{2}(2 \wr \Sym(4))_a$   & $-$                 & $\GL_2(3)$          & $2 \times 2$          &   $\Omega_8^+(2)$         \\
$ 2$    & $\Sym(5)$                      & $-$                 & $2 \times \SL_2(3)$ & $2 \times 2$          &   $\PSU_5(2)$         \\
$ 3$ & $\GL_2(3)$                     & $\GL_2(3)$          & $-$                 & $2 \times 2$          &   $-$         \\
$ 4$ & $2 \times \Alt(5)$             & $\GL_2(3)$          & $-$                 & $2 \times 2$          &    $-$        \\
$5$ & $\frac{1}{2}(2 \wr \Sym(4))_b$ & $\GL_2(3)$          & $-$                 & $2 \times 2$          & $\Alt(12)$  \\
$6$ & $2 \wr \Sym(4)$                & $-$                 & $2 \times \GL_2(3)$ & $2 \times 2 \times 2$ &  $\mathrm O_8^+(2)$         \\
$7$ & $2 \times \Sym(5)$             & $-$                 & $2 \times \GL_2(3)$ & $2 \times 2 \times 2$ &  $\PSU_5(2).2$         \\
$ 8$ & $2 \wr \Sym(4)$                & $2 \times \GL_2(3)$ & $-$                 & $2 \times 2 \times 2$ & $\Sym(12)$  \\
$ 9$ & $2 \times \Sym(5)$             & $2 \times \GL_2(3)$ & $-$                 & $2 \times 2 \times 2$ &    $-$     \\ \hline
\end{tabular}
\end{table}

\subsubsection{$\sgb(5^5,30)$}\label{5^5.30}
This group is isomorphic to a Sylow $5$-subgroup $S$ of $\mathrm{PGU}_5(4)$. It has maximal class and has unique subgroup $A$ of index $5$ so that the notation in Section \ref{notation} applies.

We make some remarks concerning the $5$-fusion system $\F$ of $\mathrm E_8(2)$. Setting $Z=Z(S)$, $\G=C_\F(Z)$ is isomorphic to the $5$-fusion system of $\SU_5(4)$. Any $\G$-essential subgroup is $\F$-centric and it follows that $\F$ is $\F(5^5,30,{20})$ in Table~\ref{5^5.30ta} (note also that $\Aut_\F(A) \cong (4\circ 2^{1+4}_+).\Sym(6)$ is isomorphic to the complex reflection group $G_{31}$ in the standard Shephard-Todd enumeration.)

 \begin{landscape}
{\scriptsize
\begin{longtable}{|c|c|c|c|c|c|c|c|c|c|}
\caption{Saturated fusion systems on a Sylow $5$-subgroup $\mathrm{PGU}_5(4)$}\label{5^5.30ta}
\\
\hline
$\F$      & $\Out_\F(A)$ & $\Out_\F(V_0)$ & $\Out_\F(E_0)$ & $\Out_\F(E_1)$ & $\Out_\F(E_2)$ & $\Out_\F(E_3)$ & $\Out_\F(E_4)$         & $\Out_\F(S)$ & Example       \\ \hline
$ 1$    & $\GL_2(5)/\{\pm I\}\cong \Alt(5):4$                      & $-$            & $\SL_2(5).2$   & $\SL_2(5)$ & $\SL_2(5)$ & $E_2\sim_\F E_3$ & $E_1\sim_\F E_4 $& $4 \times 2$ &    $-$ \\
$ 2$    & $2^{1+4}_-{:}\Frob(20)$                     & $-$            & $\SL_2(5).2$   & $\SL_2(5)$ & $\SL_2(5)$ & $E_2\sim_\F E_3$ & $E_1\sim_\F E_4$ &$4 \times 2$ &   $-$                       \\
$ 3$    & $\GL_2(5)/\{\pm I\}\cong  \Alt(5):4$                       & $-$            & $\SL_2(5).2$   & $\SL_2(5)$ & $-$ & $-$ & $E_1\sim_\F E_4 $ & $4 \times 2$ &     $-$                     \\

$ 4$    & $2^{1+4}_-{:}\Frob(20)$                        & $-$            & $\SL_2(5).2$   & $\SL_2(5)$ & $-$ & $-$ & $E_1\sim_\F E_4 $ &   $4 \times 2$         &   $-$           \\

$ 5$    & $\GL_2(5)/\{\pm I\}\cong \Alt(5):4$                       & $-$            & $-$   & $\SL_2(5)$ & $\SL_2(5)$ &  $E_2\sim_\F E_3$ & $E_1\sim_\F E_4 $ & $4 \times 2$ &                         $-$ \\

$ 6$    &$2^{1+4}_-{:}\Frob(20)$                       &   $-$            & $-$   & $\SL_2(5)$ & $\SL_2(5)$ &  $E_2\sim_\F E_3$ & $E_1\sim_\F E_4 $   & $4 \times 2$ &      $-$            \\

$ 7$    & $\GL_2(5)/\{\pm I\}\cong \Alt(5):4$                    &   $-$            & $\SL_2(5).2$   & $-$ & $-$ & $-$ & $-$         & $4 \times 2$ &            $-$              \\

$ 8$    &$2^{1+4}_-{:}\Frob(20)$                       & $-$            & $\SL_2(5).2$   & $-$ & $-$ & $-$ & $-$         & $4 \times 2$ &       $-$                \\

$ 9$    & $\GL_2(5)/\{\pm I\}\cong \Alt(5):4$                     & $-$            & $-$   & $\SL_2(5)$ & $-$ & $-$ & $E_1\sim_\F E_4 $        & $4 \times 2$ &     $-$                   \\

$ {10}$ & $2^{1+4}_-{:}\Frob(20)$                      & $-$            & $-$   & $\SL_2(5)$ & $-$ & $-$ & $E_1\sim_\F E_4 $     & $4 \times 2$ &             $-$             \\

$ {11}$ & $\GL_2(5)/\{\pm I\}\cong \Alt(5):4$                    & $\SL_2(5).2$   & $-$     & $-$ &$-$ & $-$  & $-$                    & $4 \times 2$ &    $-$                     \\

$ {12}$ & $2 \times \Sym(5), (1/3)$                        & $\SL_2(5).2$   & $-$     & $-$ &$-$ & $-$  & $-$                   & $4 \times 2$ &         $\mathrm{PGU}_5(4).2$                \\

$ {13}$ & $2^{1+4}_-{:}\Frob(20)$                         & $\SL_2(5).2$   & $-$     & $-$ &$-$ & $-$  & $-$        & $4 \times 2$ &          $-$                \\

$ {14}$ & $-$                                       & $\SL_2(5).2$   & $-$     & $-$ &$-$ & $-$  & $-$          & $4 \times 2$ &          $-$            \\

$ {15}$ & $(4 \circ 2^{1+4}_+).\Sym(5)_b$                       & $-$            & $\GL_2(5)$   &$\SL_2(5)$ & $E_1\sim_\F E_2$ & $E_1\sim_\F E_3$  &$E_1\sim_\F E_4$   & $4 \times 4$ &                        $-$ \\

$ {16}$ & $\GL_2(5), (4)$                           & $-$            & $\GL_2(5)$   &$\SL_2(5)$ & $E_1\sim_\F E_2$ & $E_1\sim_\F E_3$  &$E_1\sim_\F E_4$   & $4 \times 4$ &  $-$          \\

$ {17}$ & $4 \times \Sym(5)$                        & $-$             & $\GL_2(5)$   &$\SL_2(5)$ & $E_1\sim_\F E_2$ & $E_1\sim_\F E_3$  &$E_1\sim_\F E_4$   & $4 \times 4$ &    $-$                   \\
$ {18}$ & $(4 \circ 2^{1+4}_+):\Sym(5)_a$ & & $\GL_2(5)$   &$\SL_2(5)$  & $E_1\sim_\F E_2$ & $E_1\sim_\F E_3$  &$E_1\sim_\F E_4$   &$4 \times 4$& $-$ \\
$ {19}$ & $\GL_2(5), (2 / 2)$                       & $-$             & $\GL_2(5)$   &$\SL_2(5)$  & $E_1\sim_\F E_2$ & $E_1\sim_\F E_3$  &$E_1\sim_\F E_4$   & $4 \times 4$ & $-$            \\

$ {20}$ & $(4 \circ 2^{1+4}_+).\Sym(6) \cong G_{31}$               & $-$             & $\GL_2(5)$   &$\SL_2(5)$ & $E_1\sim_\F E_2$ & $E_1\sim_\F E_3$  &$E_1\sim_\F E_4$   & $4 \times 4$ & $\mathrm E_8(2)$                        \\

$ {21}$ & $(4\circ 2^{1+4}_+).\Frob(20)$                         & $-$             & $\GL_2(5)$   &$\SL_2(5)$  & $E_1\sim_\F E_2$ & $E_1\sim_\F E_3$  &$E_1\sim_\F E_4$   & $4 \times 4$ &                        $-$ \\

$ {22}$ & $4.\Sym(6)$                               & $-$             & $\GL_2(5)$   &$\SL_2(5)$  & $E_1\sim_\F E_2$ & $E_1\sim_\F E_3$  &$E_1\sim_\F E_4$   & $4 \times 4$ &                        $-$ \\

$ {23}$ & $(4 \circ 2^{1+4}_+):\Sym(5)_a$               & $\GL_2(5)$     & $-$      & $\SL_2(5)$  & $E_1\sim_\F E_2$ & $E_1\sim_\F E_3$  &$E_1\sim_\F E_4$   &$4 \times 4$ &                         $-$ \\

$ {24}$ &  $(4 \circ 2^{1+4}_+).\Sym(5)_b$                  & $\GL_2(5)$     & $-$      & $\SL_2(5)$  & $E_1\sim_\F E_2$ & $E_1\sim_\F E_3$  &$E_1\sim_\F E_4$   & $4 \times 4$ & $-$         \\

$ {25}$ & $\GL_2(5), (4)$                           & $\GL_2(5)$     & $-$      & $\SL_2(5)$  & $E_1\sim_\F E_2$ & $E_1\sim_\F E_3$  &$E_1\sim_\F E_4$   & $4 \times 4$ &    $-$                     \\

$ {26}$ & $4 \times \Sym(5), (3/1)$                        & $\GL_2(5)$     & $-$      & $\SL_2(5)$  & $E_1\sim_\F E_2$ & $E_1\sim_\F E_3$  &$E_1\sim_\F E_4$   & $4 \times 4$ &                       $-$ \\

$ {27}$ & $4 \times \Sym(5), (1/3)$                        & $\GL_2(5)$     & $-$      & $\SL_2(5)$  & $E_1\sim_\F E_2$ & $E_1\sim_\F E_3$  &$E_1\sim_\F E_4$   &$4 \times 4$ &    $-$          \\

$ {28}$ & $\GL_2(5), (2 / 2)$                       & $\GL_2(5)$     & $-$      & $\SL_2(5)$  & $E_1\sim_\F E_2$ & $E_1\sim_\F E_3$  &$E_1\sim_\F E_4$   & $4 \times 4$ &         $-$                 \\

$ {29}$ & $(4 \circ 2^{1+4}_+).\Sym(6) \cong G_{31}$               & $\GL_2(5)$    & $-$      & $\SL_2(5)$  & $E_1\sim_\F E_2$ & $E_1\sim_\F E_3$  &$E_1\sim_\F E_4$   & $4 \times 4$ &                         $-$ \\

$ {30}$ & $(4\circ 2^{1+4}_+).\Frob(20)$                        & $\GL_2(5)$     & $-$      & $\SL_2(5)$  & $E_1\sim_\F E_2$ & $E_1\sim_\F E_3$  &$E_1\sim_\F E_4$   &$4 \times 4$ &                      $-$ \\

$ {31}$ & $4.\Sym(6)$                               & $\GL_2(5)$   & $-$      & $\SL_2(5)$ & $E_1\sim_\F E_2$ & $E_1\sim_\F E_3$  &$E_1\sim_\F E_4$   & $4 \times 4$ &       $-$            \\

$ {32}$ & $-$                                       & $\GL_2(5)$     & $-$      & $\SL_2(5)$  & $E_1\sim_\F E_2$ & $E_1\sim_\F E_3$  &$E_1\sim_\F E_4$   & $4 \times 4$ &     $-$                     \\

$ {33}$ & $(4 \circ 2^{1+4}_+):\Sym(5)_a$               & $-$            & $\GL_2(5)$  & $-$ &$-$&  $-$   & $-$     & $4 \times 4$ &     $-$                     \\

$ {34}$ & $(4 \circ 2^{1+4}_+).\Sym(5)_b$               & $-$            & $\GL_2(5)$  & $-$ &$-$&  $-$   & $-$      & $4 \times 4$ &    $-$                 \\

$ {35}$ & $\GL_2(5), (4)$                           & $-$            & $\GL_2(5)$  & $-$ &$-$&  $-$   & $-$   & $4 \times 4$ &   $-$          \\

$ {36}$ & $4 \times \Sym(5)$                        & $-$            & $\GL_2(5)$  & $-$ &$-$&  $-$   & $-$    & $4 \times 4$ &    $-$               \\

$ {37}$ & $\GL_2(5), (2 / 2)$                       & $-$            & $\GL_2(5)$  & $-$ &$-$&  $-$   & $-$     & $4 \times 4$ &    $-$           \\

$ {38}$ & $(4 \circ 2^{1+4}_+).\Sym(6) \cong G_{31}$               & $-$          & $\GL_2(5)$  & $-$ &$-$&  $-$   & $-$    & $4 \times 4$ &   $-$         \\

$ {39}$ & $(4\circ 2^{1+4}_+).\Frob(20)$                      & $-$           & $\GL_2(5)$  & $-$ &$-$&  $-$   & $-$     & $4 \times 4$ &    $-$                 \\

$ {40}$ & $4.\Sym(6)$                               & $-$         & $\GL_2(5)$  & $-$ &$-$&  $-$   & $-$    & $4 \times 4$ &  $-$      \\

$ {41}$ & $(4 \circ 2^{1+4}_+):\Sym(5)_a$               & $-$           & $-$  & $\SL_2(5)$  & $E_1\sim_\F E_2$ & $E_1\sim_\F E_3$  &$E_1\sim_\F E_4$   & $4 \times 4$ &      $-$               \\

$ {42}$ & $(4 \circ 2^{1+4}_+).\Sym(5)_b$                & $-$            & $-$  & $\SL_2(5)$  & $E_1\sim_\F E_2$ & $E_1\sim_\F E_3$  &$E_1\sim_\F E_4$   & $4 \times 4$ &  $-$               \\

$ {43}$ & $\GL_2(5), (4)$                           & $-$            & $-$  & $\SL_2(5)$  & $E_1\sim_\F E_2$ & $E_1\sim_\F E_3$  &$E_1\sim_\F E_4$   & $4 \times 4$ &   $-$         \\

$ {44}$ & $4 \times \Sym(5)$                        & $-$            &$-$  & $\SL_2(5)$  & $E_1\sim_\F E_2$ & $E_1\sim_\F E_3$  &$E_1\sim_\F E_4$   &$4 \times 4$ &    $-$                  \\

$ {45}$ & $\GL_2(5), (2 / 2)$                       & $-$            & $-$  & $\SL_2(5)$ & $E_1\sim_\F E_2$ & $E_1\sim_\F E_3$  &$E_1\sim_\F E_4$   & $4 \times 4$ &   $-$            \\

$ {46}$ & $(4 \circ 2^{1+4}_+).\Sym(6) \cong G_{31}$               & $-$            & $-$  & $\SL_2(5)$  & $E_1\sim_\F E_2$ & $E_1\sim_\F E_3$  &$E_1\sim_\F E_4$   & $4 \times 4$ &                    $-$ \\

$ {47}$ & $(4\circ 2^{1+4}_+).\Frob(20)$                        & $-$            & $-$  & $\SL_2(5)$  & $E_1\sim_\F E_2$ & $E_1\sim_\F E_3$  &$E_1\sim_\F E_4$   &$4 \times 4$ &                     $-$ \\

$ {48}$ & $4.\Sym(6)$                               & $-$            &$-$  & $\SL_2(5)$  & $E_1\sim_\F E_2$ & $E_1\sim_\F E_3$  &$E_1\sim_\F E_4$   & $4 \times 4$ &    $-$                \\

$ {49}$ & $(4 \circ 2^{1+4}_+):\Sym(5)_a$               & $\GL_2(5)$     & $-$   & $-$ & $-$ & $-$   & $-$            & $4 \times 4$ &           $-$         \\

$ {50}$ &$(4 \circ 2^{1+4}_+).\Sym(5)_b$             & $\GL_2(5)$    & $-$   & $-$ & $-$ & $-$   & $-$                & $4 \times 4$ &      $-$             \\

$ {51}$ & $\GL_2(5), (4)$                           & $\GL_2(5)$    & $-$   & $-$ & $-$ & $-$   & $-$                   & $4 \times 4$ &       $-$      \\

$ {52}$ & $4 \times \Sym(5), (3/1)$                 & $\GL_2(5)$     & $-$   & $-$ & $-$ & $-$   & $-$              & $4 \times 4$ &       $-$         \\

$ {53}$ & $4 \times \Sym(5), (1/3)$                 & $\GL_2(5)$     & $-$   & $-$ & $-$ & $-$   & $-$            & $4 \times 4$ &    $\mathrm {P G U}_5(4).4$                \\

$ {54}$ & $\GL_2(5), (2 / 2)$                       & $\GL_2(5)$     & $-$   & $-$ & $-$ & $-$   & $-$                & $4 \times 4$ &     $-$            \\

$ {55}$ & $(4 \circ 2^{1+4}_+).\Sym(6) \cong G_{31}$               & $\GL_2(5)$   & $-$   & $-$ & $-$ & $-$   & $-$               & $4 \times 4$ &   $-$              \\

$ {56}$ & $(4\circ 2^{1+4}_+).\Frob(20)$                         & $\GL_2(5)$     & $-$   & $-$ & $-$ & $-$   & $-$             & $4 \times 4$ &    $-$              \\

$ {57}$ & $4.\Sym(6)$                               & $\GL_2(5)$     & $-$   & $-$ & $-$ & $-$   & $-$              & $4 \times 4$ &   $-$              \\

$ {58}$ & $-$                                       & $\GL_2(5)$     & $-$   & $-$ & $-$ & $-$   & $-$              & $4 \times 4$ &   $-$              \\
\hline
\end{longtable}}
\end{landscape}

\subsubsection{$\sgb(7^5,32)$}\label{7^5.32}

The group $\texttt{SmallGroup($7^5,32$)}$ has maximal   class and has a unique abelian subgroup of order $7^4$ which is denoted $A$. It is isomorphic to the quotient of a Sylow $7$-subgroup of $\Gee_2(7)$ by its centre from which we recognise it as a Sylow $7$-subgroup of $A:\SL_2(7)$ where $\SL_2(7)$ acts irreducibly on $A$ as the symmetric cube of the natural $\SL_2(7)$-module.

\begin{table}[H]
\caption{Saturated fusion systems on a Sylow $7$-subgroup of $7^4:\SL_2(7)$}
\label{7^5.32t}
\renewcommand{\arraystretch}{1.4}
\begin{tabular}{|c|c|c|c|c|c|}
\hline
          & $\Out_\F(A)$         & $\Out_\F(V_0)$ & $\Out_\F(E_0)$ & $\Out_\F(S)$ & Example \\ \hline
$ 1$    & $-$                  & $\SL_2(7)$   & $-$          & $6$         &   $-$   \\
$ 2$    & $-$                  & $\SL_2(7)$.2 & $-$          & $6 \times 2$ &   $-$  \\
$ 3$    & $2.\Sym(7)$          & $-$          & $\SL_2(7)$.2 & $6 \times 2$ &  $-$   \\
$ 4$    & $\SL_2(7)$.2         & $-$          & $\SL_2(7)$.2 & $6 \times 2$ & $-$    \\
$ 5$    & $-$                  & $\SL_2(7).3$ & $-$          & $6 \times 3$ & $-$    \\
$ 6$    & $2.\Sym(7) \times 3$ & $\GL_2(7)$   & $-$          & $6 \times 6$ & $-$    \\
$ 7$    & $\GL_2(7)$           & $\GL_2(7)$   & $-$          & $6 \times 6$ &  $-$    \\
$ 8$    & $2.\Sym(7) \times 3$ & $-$          & $\GL_2(7)$   & $6 \times 6$ &  $-$    \\
$ {9}$ & $\GL_2(7)$           & $-$          & $\GL_2(7)$   & $6 \times 6$ &   $-$  \\
$ {10}$ & $-$         & $\GL_2(7)$   & $-$          & $6 \times 6$ &  $-$    \\ \hline
\end{tabular}
\end{table}

\subsubsection{$\sgb(7^5,37)$}\label{7^5.37}
 This is the maximal class group $S=B(7,5;0,1,0,0)$ and the notation in Section \ref{notation} applies. $S$ is isomorphic to a maximal subgroup of a Sylow $7$-subgroup of $\Gee_2(7)$ which is not a unipotent radical subgroup of $\Gee_2(7)$.   There is a single reduced fusion system $\F(7^5,37,1)$ on $S$, originally discovered by Grazian \cite{grazian2018fusion}. This fusion system can be constructed from the Monster finite simple group using Theorem \ref{shrink} (see Example \ref{ex:prune}(1).)

\begin{table}[H]
\caption{Saturated fusion systems on the group $B(7,5;0,1,0,0)$}
\label{7^5.37t}
\renewcommand{\arraystretch}{1.4}
\begin{tabular}{|c|c|c|c|}
\hline
$\F$          & $\Out_\F(V_0)$  & $\Out_\F(S)$ & Example                 \\ \hline
$1$         & $\SL_2(7)$                 & $6$          & $-$                                       \\ \hline
\end{tabular}
\end{table}

\subsection{Groups of order $p^6$} We describe $p$-fusion systems on $p$-groups of order $p^6$ when $p \in \{3,5\}$. By Theorem \ref{Thmp^6}, we may adopt the notation of Section \ref{notation} provided $S$ is not isomorphic with  $\texttt{SmallGroup($3^6,i$)}$ for $i \in \{149,307, 321, 453,469\}$ or $\texttt{SmallGroup($5^6,i$)}$ for $i \in \{240,276, 609, 616,643\}$ as these groups have no abelian subgroup of index $3$ or $5$ respectively.
\subsubsection{$\sgb(3^6,95)$}\label{3^6.95} This is the maximal class group $B(3,6;0,0,0,0)$. It is isomorphic to a Sylow $3$-subgroup of $\PSL_3(109)$. Where appropriate, we indicate in the Example column how these fusion systems correspond to the exotic examples  discovered in \cite[Theorem 5.10, Table 4]{DiazRuizViruel2007}.

\begin{table}[H]
\caption{Saturated fusion systems on a Sylow $3$-subgroup of $\PSL_3(109)$}
\label{3^6.95t}
\renewcommand{\arraystretch}{1.4}
\begin{tabular}{|c|c|c|c|c|c|c|}
\hline
$\F$ & $\Out_\F(V_0)$ & $\Out_\F(V_1)$ & $\Out_\F(V_2)$ & $\Out_\F(E_0)$ & $\Out_\F(S)$ & Example            \\ \hline
$ 1$                     & $\SL_2(3)$     & $\SL_2(3)$     & $\SL_2(3)$     & $-$            & $2$          & $\PSL_3(109)$    \\
$ 2$                     & $-$            & $\SL_2(3)$     & $\SL_2(3)$     & $-$            & $2$          & $\F_{\mathrm{DRV}}(3^6,2)$                      \\
$ 3$                     & $\SL_2(3)$     & $-$            & $-$            & $-$            & $2$          & $\F_{\mathrm{DRV}}(3^6,1)$                       \\
$ 4$                     & $-$            & $\SL_2(3)$     & $\SL_2(3)$     & $\GL_2(3)$     & $2 \times 2$ & $^3\mathrm D_4(8)$               \\
$ 5$                     & $\GL_2(3)$     & $\SL_2(3)$     & $\SL_2(3)$     & $-$            & $2 \times 2$ & $\PSL_3(109).2$ \\
$ 6$                     & $-$            & $\SL_2(3)$     & $\SL_2(3)$     & $-$            & $2 \times 2$ & $\F_{\mathrm{DRV}}(3^6,2).2$             \\
$ 7$                     & $\GL_2(3)$     & $-$            & $-$            & $-$            & $2 \times 2$ & $\F_{\mathrm{DRV}}(3^6,1).2$                 \\ \hline
\end{tabular}
\end{table}

\subsubsection{$\sgb(3^6,97)$}\label{3^6.97} This is the maximal class group $B(3,6;0,0,1,0)$. We indicate in the Example column how these fusion systems correspond to the exotic examples  discovered in \cite[Theorem 5.10, Table 5]{DiazRuizViruel2007}.

\begin{table}[H]
\caption{Saturated fusion systems on the group $B(3,6;0,0,1,0)$}
\label{3^6.97t}
\renewcommand{\arraystretch}{1.4}
\begin{tabular}{|c|c|c|c|}
\hline
$\F$   & $\Out_\F(V_0)$ & $\Out_\F(S)$ & Example          \\ \hline
$ 1$ & $\SL_2(3)$     & $2$          & $\F_{\mathrm{DRV}}(3^6,3)$      \\
$ 2$ & $\GL_2(3)$     & $2 \times 2$ & $\F_{\mathrm{DRV}}(3^6,3).2$         \\ \hline
\end{tabular}
\end{table}

\subsubsection{$\sgb(3^6,98)$}\label{3^6.98} This is the maximal class group  $B(3,6;0,0,2,0)$. We indicate in the Example column how these fusion systems correspond to the exotic examples  discovered in \cite[Theorem 5.10, Table 5]{DiazRuizViruel2007}.

\begin{table}[H]
\caption{Saturated fusion systems on  the group $B(3,6;0,0,2,0)$}
\label{3^6.98t}
\renewcommand{\arraystretch}{1.4}
\begin{tabular}{|c|c|c|c|}
\hline
$\F$   & $\Out_\F(V_0)$ & $\Out_\F(S)$ & Example       \\ \hline
$ 1$ & $\SL_2(3)$     & $2$          & $\F_{\mathrm{DRV}}(3^6,4)$      \\
$ 2$ & $\GL_2(3)$     & $2 \times 2$ & $\F_{\mathrm{DRV}}(3^6,4).2$         \\ \hline
\end{tabular}
\end{table}

\subsubsection{$\sgb(3^6,99)$}\label{3^6.99} This is $B(3,6;0,1,0,0)$. It supports a unique saturated fusion system \cite[Theorem 1.1 (i)]{parker2018fusion2} and it is exotic.

\begin{table}[H]
\caption{Saturated fusion systems on the group $B(3,6;0,1,0,0)$}
\label{3^6.99t}
\renewcommand{\arraystretch}{1.4}
\begin{tabular}{|c|c|c|c|}
\hline
$\F$      & $\Out_\F(V_0)$ & $\Out_\F(S)$ & Example       \\ \hline
$ 1$ & $\SL_2(3)$     & $2$          & $-$      \\ \hline
\end{tabular}
\end{table}

\subsubsection{$\sgb(3^6,100)$}\label{3^6.100} This is $B(3,6;0,1,1,0)$. Saturated fusion systems on this group were classified in \cite[Theorem 1.1 (ii)]{parker2018fusion2}. They are all exotic.

\begin{table}[H]
\caption{Saturated fusion systems on the group $B(3,6;0,1,1,0)$}
\label{3^6.100t}
\renewcommand{\arraystretch}{1.4}
\begin{tabular}{|c|c|c|c|c|}
\hline
$\F$      & $\Out_\F(V_1)$ & $\Out_\F(V_2)$ & $\Out_\F(S)$ & Example        \\ \hline
$1$ & $\SL_2(3)$     & $\SL_2(3)$     & $2$          & $-$     \\
$2$ & $\SL_2(3)$     & $-$            & $2$          & $-$    \\
$ 3$ & $-$            & $\SL_2(3)$     & $2$          & $-$      \\ \hline
\end{tabular}
\end{table}

\subsubsection{$\sgb(3^6,149)$}\label{3^6.149}
This group is isomorphic to a Sylow $3$-subgroup $S$ of $\Gee_2(3)$. The fusion systems were classified in \cite[Theorem 7.2]{parker2018fusion} and we adopt the notation used there. In particular, $Q_1$ and $Q_2$ are the unipotent radicals of the maximal parabolic subgroups containing $S$.
\begin{table}[H]
\caption{Saturated fusion systems on a Sylow $3$-subgroup of $\Gee_2(3)$}
\label{3^6.149t}
\renewcommand{\arraystretch}{1.4}
\begin{tabular}{|c|c|c|c|c|}
\hline
$\F$      & $\Out_\F(Q_1)$ & $\Out_\F(Q_2)$ & $\Out_\F(S)$  & Example \\ \hline
$ 1$ & $\GL_2(3)$     & $\GL_2(3)$     & $2 \times 2$ & $\Gee_2(3)$    \\
$2$ & $\GL_2(3)$     &$Q_1\sim_\F Q_2$   & $\Dih(8)$         & $\Gee_2(3).2$ \\ \hline
\end{tabular}
\end{table}

\subsubsection{$\sgb(3^6,307)$}\label{3^6.307}
This group is isomorphic to a Sylow $3$-subgroup of $\PSL_4(3)$.
We label the three unipotent radicals of proper parabolic subgroups in this group by $Q_1,Q_2,Q_3$. The fusion systems are listed in Table \ref{3^6.307t}. The group in the line of  $\F (3^6,307,8)$ of shape $(4\circ\SL_2(5)).2 $  has centre of order $2$ and, in particular, is not isomorphic to $\GL_2(5)$.

\begin{table}[H]
\caption{Saturated fusion systems on a Sylow $3$-subgroup of $\PSL_4(3)$}
\label{3^6.307t}
\renewcommand{\arraystretch}{1.4}
\begin{tabular}{|c|c|c|c|c|c|}
\hline
$\F$      & $\Out_\F(Q_1)$      & $\Out_\F(Q_2)$             & $\Out_\F(Q_3)$      & $\Out_\F(S)$          & Example     \\ \hline
$1$ & $\GL_2(3)$          & $\GL_2(3)$                 & $\GL_2(3)$          & $2 \times 2$          & $\PSL_4(3)$     \\
$ 2$ & $\GL_2(3)$          & $\GL_2(3) \times 2$        & $Q_1\sim_\F Q_3$          & $2 \times 2 \times 2$ & $\PSL_4(3).2$   \\
$ 3$ & $\GL_2(3)$          & $\GL_2(3) \times 2$        & $Q_1\sim_\F Q_3$             & $2 \times 2 \times 2$ & $\Omega_8^+(2) :\Sym(3)$ \\
$ 4$ & $\GL_2(3)$          & $(\mathrm Q_8 \times \mathrm Q_8).\Sym(3)$ & $Q_1\sim_\F Q_3$             & $2 \times 2 \times 2$ &    $\mathrm F_4(2)$             \\
$ 5$ & $\GL_2(3)$          & $4\circ \GL_2(3)$               &$Q_1\sim_\F Q_3$           & $2 \times 4$          & $\PSL_4(3).2$   \\
$ 6$ & $\GL_2(3)$          & $4\circ\SL_2(5)$                & $Q_1\sim_\F Q_3$            & $2 \times 4$          &     $\HN$            \\
$ 7$ & $\GL_2(3) \times 2$ & $\GL_2(3) \times 2$        & $\GL_2(3) \times 2$ & $2 \times 2 \times 2$ & $\PSL_4(3).2$   \\
$ 8$ & $\GL_2(3) \times 2$ & $(4\circ\SL_2(5)).2 $              & $Q_1\sim_\F Q_3$   & $2 \times \Dih(8)$        &   $\Aut(\HN)$              \\
$ 9$ & $\GL_2(3) \times 2$ & $(\mathrm Q_8 \times \mathrm Q_8).\Dih(12)$ & $Q_1\sim_\F Q_3$    & $2 \times \Dih(8)$        &   $\Aut(\mathrm F_4(2))$              \\
$ 10$ & $\GL_2(3) \times 2$ & $\GL_2(3).2^2$              &$Q_1\sim_\F Q_3$ & $2 \times \Dih(8)$        & $\PSL_4(3).2^2$ \\ \hline
\end{tabular}
\end{table}

\subsubsection{$\sgb(3^6,321)$}\label{3^6.321}
This group is isomorphic to a Sylow $3$-subgroup of $\PSU_4(3)\cong \Omega_6^-(3)$. It is also a Sylow $3$-subgroup of the semi-direct product $Q:\Alt(6)$ where $Q \cong 3^4$ is identified with the heart of the natural module (a 1-space stabilizer in the natural action of $\Omega_6^-(3)$ on its natural module). If $\F=O^3(\F)$ and $O_3(\F)=1$,  there are always two classes of essential subgroups given by $\{Q,R\}$ where $R \cong 3^{1+4}_+$ is the unique extraspecial subgroup of $S$ of index $3$. These are unipotent radicals of maximal parabolic subgroups if $\PSU_4(3)$. The fusion systems on this group were determined in \cite{BFM2019}.

\begin{table}[H]
\caption{Saturated fusion systems on a Sylow $3$-subgroup of $\PSU_4(3)$}
\label{3^6.321t}
\renewcommand{\arraystretch}{1.4}
\begin{tabular}{|c|c|c|c|c|}
\hline
$\F$      & $\Out_\F(Q)$       & $\Out_\F(R)$               & $\Out_\F(S)$       & Example           \\ \hline
$ 1$ & $\Alt(6)$          & $2.\Sym(4)$                & $4$                & $\PSU_4(3)$           \\
$ 2$ & $\Sym(6)$          & $\GL_2(3).2$               & $\mathrm \Dih(8)$              & $\PSU_4(3).2_2$       \\
$ 3$ & $\Sym(6)$          & $(\mathrm  Q_8 \times \mathrm Q_8).\Sym(3)$ & $\mathrm \Dih(8)$              & $\PSU_6(2)$           \\
$4$ & $2 \times \Alt(6)$ & $2.\Sym(4):2$              & $2 \times 4$       & $\PSU_4(3).2_1$       \\
$ 5$ & $\M_{10}$           & $\SL_2(3).2^2$             & $\mathrm Q_8$              & $\PSU_4(3).2_3$       \\
$ 6$ & $\M_{10}$           & $2.\Sym(5)$                & $\mathrm Q_8$              & $\McL$               \\
$ 7$ & $2 \times \Sym(6)$ & $(\mathrm Q_8 \times \mathrm   Q_8).D_{12}$  & $2 \times \mathrm \Dih(8)$     & $\PSU_6(2).2$        \\
$ 8$ & $2 \times \Sym(6)$ & $\GL_2(3).2^2$             & $2 \times \mathrm \Dih(8)$     & $\PSU_4(3).2^2_{122}$ \\
$ 9$ & $2 \times \M_{10}$  & $\SL_2(5).2^2$             & $2 \times \mathrm  Q_8$     & $\McL.2$              \\
$ 10$ & $2 \times \M_{10}$  & $2.\Sym(4):2^2$            & $2 \times \mathrm Q_8$     & $\PSU_4(3).2^2_{133}$ \\
$ 11$ & $\Alt(6):4$        & $\GL_2(3):4$               & $2 \times 8$       & $\PSU_4(3).4$         \\
$ 12$ & $\Alt(6).\Dih(8)$    & $2^{1+4}_-.\Sym(5)$            & $2 \times\SDih(16)$ & $\Co_2$               \\
$ 13$ & $\Alt(6).\Dih(8)$    & $\SL_2(3)\circ\SDih(16)$          & $2 \times \SDih(16)$ & $\PSU_4(3).\mathrm \Dih(8)$       \\ \hline
\end{tabular}
\end{table}

\subsubsection{$\sgb(3^6,453)$}\label{3^6.453}
This group is isomorphic with a Sylow $3$-subgroup $S$ of $G \times G$ where $G=\PSL_3(3)$. Also set $H= {}^2\mathrm F_4(2)'$. We have $\Out(G \times G) \cong \Out(H \times H) \cong \Dih(8)$ and $\Out(G \times H) \cong C_2 \times C_2$.  The fusion systems on this group are one-to-one correspondence with classes of subgroups of these outer automorphism groups and are not individually listed.

\subsubsection{$\sgb(3^6,469)$}\label{3^6.469}

This group is isomorphic to a Sylow $3$-subgroup of $G=\PSL_3(9)$.  We have $\Out(G) \cong C_2 \times C_2$ and there are five fusion systems and they are realized in the appropriate subgroup by subgroups of $\Aut(\PSL_3(9))$ containing $\PSL_3(9)$. This result agrees with one of the main theorems in Clelland's PhD thesis \cite[Theorem 4.5.1]{clelland2007saturated}.

\subsubsection{$\sgb(3^6,479)$}\label{3^6.479}

This group is isomorphic with a Sylow $3$-subgroup of $\Alt(15)$.
\begin{table}[H]
\caption{Saturated fusion systems on a Sylow $3$-subgroup of $\Alt(15)$}
\label{3^6.479t}
\renewcommand{\arraystretch}{1.4}
\begin{tabular}{|c|c|c|c|c|c|}
\hline
$\F$      & $\Out_\F(A)$                 & $\Out_\F(V_0)$          & $\Out_\F(E_0)$          & $\Out_\F(S)$            & Example \\ \hline
$ 1$ & $\frac{1}{2}(2 \wr \Sym(5))$ & $4:\GL_2(3)$          & $-$                   & $2 \times \Dih(8)$          & $\Alt(15)$  \\
$ 2$ & $\frac{1}{2}(2 \wr \Sym(5))$ & $-$                   & $2^2:\GL_2(3)$        & $2 \times \mathrm \Dih(8)$    &     $\Omega^-_{10}(2)$     \\
$ 3$ & $2 \wr \Sym(5)$              & $-$                   & $\GL_2(3) \times\mathrm \Dih(8)$ & $2 \times 2 \times \mathrm \Dih(8)$ &$\Omega_{10}^-(2).2$,   $\Sp_{10}(2)$        \\
$ 4$ & $2 \wr \Sym(5)$              & $\GL_2(3) \times \mathrm \Dih(8)$ & $-$                   & $2 \times 2 \times \mathrm \Dih(8)$ & $\Sym(15)$  \\ \hline
\end{tabular}
\end{table}

\subsubsection{$\sgb(5^6,240)$}\label{5^6.240} This group is isomorphic to $S_1 \times S_2$ where $S_1 \cong S_2 \cong 5^{1+2}_+$.  We obtain fusion systems from conjugacy classes of subgroups of $\Out(\Th \times \Th)$, $\Out(\PSL_3(5) \times \Th)$ and $\Out(\PSL_3(5) \times \PSL_3(5))$ of which there are $2,2$ and $8$. This accounts for all fusion systems and they are not individually listed.

\subsubsection{$\sgb(5^6,276)$}\label{5^6.276}  This group is isomorphic to a Sylow $5$-subgroup of $\PSL_3(25)$, and the fusion systems are in bijection with classes of subgroups of $\Out(\PSL_3(25)) \cong \Dih(12)$.  They are realized  by subgroups of $\Aut(\PSL_3(25))$ containing $\PSL_3(25)$. This agrees with \cite{clelland2007saturated}.

\subsubsection{$\sgb(5^6,609)$}\label{5^6.609}
This group is isomorphic to a Sylow $5$-subgroup of $\PSL_4(5)$ and the fusion systems are in bijection with classes of subgroups of $\Out(\PSL_4(5)) \cong \Dih(8)$.

\subsubsection{$\sgb(5^6,616)$}\label{5^6.616}
This group is isomorphic to a Sylow $5$-subgroup of $\PSU_4(5)$ and the fusion systems are in bijection with classes of subgroups of $\Out(\PSU_4(5)) \cong C_2 \times C_2$.

\subsubsection{$\sgb(5^6,630)$}\label{5^6.630}
This is $B(5,6;0,0,0,0)$, it has an elementary abelian maximal subgroup $A$. The fusion systems are listed below:
\begin{table}[H]
\caption{Saturated fusion systems on the groups $B(5,6;0,0,0,0)$, $B(5,6;0,0,3,0)$ and $B(5,6;0,0,4,0)$}
\label{5^6.630t}
\renewcommand{\arraystretch}{1.4}
\begin{tabular}{|c|c|c|c|c|}
\hline
$\F$      & $\Out_\F(A)$ & $\Out_\F(V_0)$     &  $\Out_\F(S)$  & Example            \\ \hline
$ {1}$    & $-$               & $\SL_2(5)$         & $4$                & $-$         \\
$ {2}$    & $\frac{1}{2}(\Sym(5) \times 4)$               & $\SL_2(5).2$         & $4 \times 2$                & $-$         \\
$ {3}$    & $-$               & $\SL_2(5).2$         & $4 \times 2$                & $-$         \\
$ {4}$    & $\Sym(5) \times 4$               & $\GL_2(5)$         & $4 \times 4$                & $-$         \\
$ {5}$    & $-$               & $\GL_2(5)$         & $4 \times 4$                & $-$         \\
\hline
\end{tabular}
\end{table}

\subsubsection{$\sgb(5^6,631)$}\label{5^6.631}
This group is isomorphic to a Sylow $5$-subgroup of $\Alt(25)$ and of $\PSU(6,4)$.

\begin{table}[H]
\scriptsize
\caption{Saturated fusion systems on a Sylow $5$-subgroup of $\Alt(25)$}
\label{5^6.631t}
\renewcommand{\arraystretch}{1.4}
\begin{tabular}{|c|c|c|c|c|c|}
\hline
          & $\Out_\F(A)$         & $\Out_\F(V_0)$ & $\Out_\F(E_0)$ & $\Out_\F(S)$ & Example \\ \hline
$ 1$    & $\frac{1}{4}(4 \wr \Frob(20))$      & $\SL_2(5)$   & $-$          & $4$         &   $-$   \\
$ 2$    & $\frac{1}{2}(2 \wr \Frob(20))$      & $\SL_2(5)$   & $-$          & $4$         &  $-$    \\
$ 3$    & $-$      & $\SL_2(5)$   & $-$          & $4$         &  $-$    \\
$ 4$    & $\Sym(5)$      & $\SL_2(5)$   & $-$          & $4$         &   $-$   \\
$ 5$    &  $\frac{1}{4}(4 \wr \Frob(20))$      & $-$   & $\SL_2(5)$          & $4$         &   $-$   \\
$ 6$    & $\frac{1}{4}(4 \wr \Sym(5))$      & $-$   & $\SL_2(5)$          & $4$     & $-$\\
$ 7$    &$\frac{1}{2}(2 \wr \Sym(5))$      & $-$   & $\SL_2(5)$          & $4$         &  $\mathrm{P\Omega}^-_{10}(4) $   \\
$ 8$    & $\Sym(6)$      & $-$   & $\SL_2(5)$          & $4$         &     $\PSU_6(4)$ \\
$ 9$    & $\frac{1}{2}(2 \wr \Frob(20))$     & $-$   & $\SL_2(5)$         & $4$         & $-$     \\
$ {10}$    & $\frac{1}{2}(4 \times \Sym(5)), (1/3/ 1)$      & $\SL_2(5).2$   & $-$          & $4 \times 2$         & $-$     \\
$ {11}$    & $\frac{1}{2}(4 \times \Sym(5)), (5)$      & $\SL_2(5).2$   & $-$          & $4 \times 2$         & $-$     \\
$ {12}$    & $\frac{1}{2}(4 \wr \Frob(20))$       & $\SL_2(5).2$   & $-$          & $4 \times 2$         &   $-$   \\
$ {13}$    & $\frac{1}{2}(4 \wr \Sym(5))$       & $\SL_2(5).2$   & $-$          & $4 \times 2$         & $\Alt(25)$     \\
$ {14}$    & $2\wr    \Sym(5)$      & $\SL_2(5).2$   & $-$          & $4 \times 2$         &     $-$ \\
$ {15}$    & $\frac{1}{2}(4 \times \Sym(6))$      & $\SL_2(5).2$   & $-$          & $4 \times 2$       & $\PSU_6(4).2$     \\
$ {16}$    &$2 \wr \Frob(20)$       & $\SL_2(5).2$   & $-$          & $4 \times 2$         &    $-$  \\
$ {17}$    & $-$      & $\SL_2(5).2$   & $-$          & $4 \times 2$         &   $-$   \\
$ {18}$    & $2\times \Sym(5)$      & $-$   & $\SL_2(5).2$          & $4 \times 2$         &   $-$   \\
$ {19}$    & $\frac{1}{2}(4 \wr \Frob(20))$     & $-$   & $\SL_2(5).2$         & $4 \times 2$         &   $-$   \\
$ {20}$    &  $\frac{1}{2}(4 \wr \Sym(5))$     & $-$   & $\SL_2(5).2$          & $4 \times 2$         &  $-$    \\
$ {21}$    & $2 \wr \Sym(5)$     & $-$   & $\SL_2(5).2$         & $4 \times 2$         &$\PSp_{10}(4)$      \\
$ {22}$    & $2 \times \Sym(6)$      & $-$   & $\SL_2(5).2$          & $4 \times 2$         &   $-$   \\
$ {23}$    & $2 \wr \Frob(20)$      & $-$   & $\SL_2(5).2$        & $4 \times 2$         &  $-$    \\
$ {24}$    & $4 \times \Sym(5), (1/3/1)$    & $\GL_2(5)$   & $-$          & $4\times 4$         &   $-$   \\
$ {25}$    & $4 \times \Sym(5), (5)$      & $\GL_2(5)$   & $-$          & $4\times 4$         &  $-$    \\
$ {26}$    & $4 \wr \Frob(20)$       & $\GL_2(5)$   & $-$          & $4\times 4$         &  $-$    \\
$ {27}$    & $4 \wr \Sym(5)$      & $\GL_2(5)$   & $-$          & $4\times 4$         &   $\Sym(25)$   \\
$ {28}$    & $(2^4 \times 4):\Sym(5)$      & $\GL_2(5)$   & $-$          & $4\times 4$        &  $-$    \\
$ {29}$    & $4 \times \Sym(6)$     & $\GL_2(5)$   & $-$          & $4\times 4$         &  $-$    \\
$ {30}$    & $2^4:(5:4^2)$      & $\GL_2(5)$   & $-$          & $4\times 4$         &    $-$  \\
$ {31}$    & $4 \times \Sym(5)$       & $-$   & $\GL_2(5)$          & $4\times 4$         &  $-$    \\
$ {32}$    & $4 \wr \Frob(20)$      & $-$   & $\GL_2(5)$          & $4\times 4$         &  $-$    \\
$ {33}$    & $4 \wr \Sym(5)$       & $-$   & $\GL_2(5)$         & $4\times 4$         &  $-$    \\
$ {34}$    & $(2^4 \times 4):\Sym(5)$      & $-$   & $\GL_2(5)$          & $4\times 4$         & $\PSp_{10}(4).2$     \\
$ {35}$    & $4 \times \Sym(6)$      & $-$   & $\GL_2(5)$          & $4\times 4$         &   $\mathrm{P\Gamma U}_6(4)$   \\
$ {36}$    & $2^4:(5:4^2)$      & $-$   & $\GL_2(5)$          & $4\times 4$         &  $-$    \\
$ {37}$    & $-$      & $\GL_2(5)$   & $-$          & $4\times 4$         & $-$     \\

\hline
\end{tabular}
\end{table}
\subsubsection{$\sgb(5^6,632)$}\label{5^6.632} This is $B(5,6;0,0,2,0)$. It has a maximal abelian subgroup $A$ which is not elementary abelian.
The fusion systems are listed below:
\begin{table}[H]
\caption{Saturated fusion systems on the group $B(5,6;0,0,2,0)$}
\label{5^6.632t}
\renewcommand{\arraystretch}{1.4}
\begin{tabular}{|c|c|c|c|c|}
\hline
$\F$      & $\Out_\F(A)$ & $\Out_\F(V_0)$     &  $\Out_\F(S)$  & Example           \\ \hline
$ {1}$    & $-$               & $\SL_2(5)$         & $4$                & $-$         \\
$ {2}$    & $\SL_2(5).2$               & $\SL_2(5).2$         & $4 \times 2$                & $-$         \\
$ {3}$    & $-$               & $\SL_2(5).2$         & $4 \times 2$                & $-$         \\
$ {4}$    & $\GL_2(5)$               & $\GL_2(5)$         & $4 \times 4$                & $-$         \\
$ {5}$    & $-$               & $\GL_2(5)$         & $4 \times 4$                & $-$         \\
\hline
\end{tabular}
\end{table}

\subsubsection{$\sgb(5^6,633)$}\label{5^6.633}
This is $B(5,6;0,0,3,0)$. It has a maximal abelian subgroup which is not elementary abelian. The fusion systems are described as in Table \ref{5^6.630t}. They are all exotic.

\subsubsection{$\sgb(5^6,634)$}\label{5^6.634}
This is $B(5,6;0,0,4,0)$. It has a maximal abelian subgroup which is not elementary abelian. The fusion systems are described as in Table \ref{5^6.630t}. They are all exotic.

\subsubsection{$\sgb(5^6,636)$}\label{5^6.636}
 This is $B(5,6;0,1,0,0)$. In this case, $S$  is a maximal class group with non-abelian  maximal subgroups, $\gamma_1(S)=C_S(Z_2(S))$ has nilpotency class $2$ and $|Z(\gamma_1(S))|=5^3$.  The fusion system is described   in Table \ref{5^6.636t} and is exotic.

\begin{table}[H]
\caption{Saturated fusion systems on the groups $B(5,6;0,1,k,0)$, $0 \le k \le 4$.}
\label{5^6.636t}
\renewcommand{\arraystretch}{1.4}
\begin{tabular}{|c|c|c|c|}
\hline
$\F$      & $\Out_\F(V_0)$ & $\Out_\F(S)$ & Example       \\ \hline
$ {1}$ & $\SL_2(5)$     & $4$          & $-$      \\ \hline
\end{tabular}
\end{table}

\subsubsection{$\sgb(5^6,639)$}\label{5^6.639}
 This is $B(5,6;0,1,1,0)$. $S$  is a maximal class group with non-abelian  maximal subgroups, $\gamma_1(S)=C_S(Z_2(S))$ has nilpotency class $2$ and $|Z(\gamma_1(S))|=5^3$.  The fusion system is described   in Table \ref{5^6.636t} and is exotic.

\subsubsection{$\sgb(5^6,640)$}\label{5^6.640}
 This is $B(5,6;0,1,2,0)$. $S$  is a maximal class group with non-abelian  maximal subgroups, $\gamma_1(S)=C_S(Z_2(S))$ has nilpotency class $2$ and $|Z(\gamma_1(S))|=5^3$.  The fusion system is described   in Table \ref{5^6.636t} and is exotic.
\subsubsection{$\sgb(5^6,641)$}\label{5^6.641}
 This is $B(5,6;0,1,3,0)$.$S$  is a maximal class group with non-abelian  maximal subgroups, $\gamma_1(S)=C_S(Z_2(S))$ has nilpotency class $2$ and $|Z(\gamma_1(S))|=5^3$.  The fusion system is described   in Table \ref{5^6.636t} and is exotic.
\subsubsection{$\sgb(5^6,642)$}\label{5^6.642}
 This is $B(5,6;0,1,4,0)$. $S$  is a maximal class group with non-abelian  maximal subgroups, $\gamma_1(S)=C_S(Z_2(S))$ has nilpotency class $2$ and $|Z(\gamma_1(S))|=5^3$.  The fusion system is described   in Table \ref{5^6.636t} and is exotic.
\subsubsection{$\sgb(5^6,643)$}\label{5^6.643} This group is isomorphic to a Sylow $5$-subgroup of $\Gee_2(5)$. It has maximal nilpotency class. The fusion systems were determined  in \cite[Theorem 7.2]{parker2018fusion} and we take our notation from  there. Hence $Q$ and $R$ are the unipotent radicals of the maximal parabolic subgroups containing $S$ with $Q$ extraspecial.

\begin{table}[H]
\caption{Saturated fusion systems on a Sylow $5$-subgroup of $\Gee_2(5)$ }
\label{3^7.2007t}
\renewcommand{\arraystretch}{1.4}
\begin{tabular}{|c|c|c|c|c|}
\hline
$\F$      & $\Out_\F(Q)$                 & $\Out_\F(R)$          & $\Out_\F(S)$            & Example \\ \hline

$ 1$ & $2^{1+4}_-.\Frob(20)$ &   $4\circ \SL_2(5)$& $2 \times 4$  & $\HN$       \\

$ 2$ & $\GL_2(5)$ & $\GL_2(5)$ &  $4\times 4$&$\Gee_2(5)$\\

       $ 3$ & $4\circ 2^{1+4}_-.\Alt(5)$ & $\GL_2(5)$&  $4\times 4$&$\mathrm B$\\
       $ 4$ & $4\circ 2^{1+4}_-.\Frob(20)$ & $\GL_2(5)$&  $4\times 4$&$\Aut(\HN)$\\
       $ 5$ & $4\circ 2^{\;.}\Sym(6)$ & $\GL_2(5)$&  $4\times 4$&$\Ly$\\
       \hline
\end{tabular}
\end{table}

\subsection{Groups of order $3^7$}
This subsection catalogues the fusion systems on groups of order $3^7$.  As always we use our standard notation for groups of maximal class.

\subsubsection{$\sgb(3^7,366)$}\label{3^7.366} This group is a Sylow $3$-subgroup of $\PSU_4(8)$ is isomorphic to the wreath product $9\wr 3$. It has an abelian subgroup $A$ of index 3 which is homocyclic. We have $E_0 \cong 9 \circ 3^{1+2}_+$.

\begin{table}[H]
\caption{Saturated fusion systems on a Sylow $3$-subgroup of $\PSU_4(8)$}
\label{3^7.366t}
\renewcommand{\arraystretch}{1.4}
\begin{tabular}{|c|c|c|c|c|}
\hline
$\F$      & $\Out_\F(A)$                 & $\Out_\F(E_0)$              & $\Out_\F(S)$            & Example \\ \hline
$ 1$ & $\Sym(4)$ & $\SL_2(3)$                         & $2$          & $\PSU_4(8)$  \\

$ 2$ & $2 \wr \Sym(3)$ & $\GL_2(3)$                         & $2 \times 2$          & $\PSU_4(8).2$ \\ \hline
\end{tabular}
\end{table}

\subsubsection{$\sgb(3^7,386)$}\label{3^7.386}
This is the maximal class group  $B(3,7;0,0,0,0)$. It is isomorphic to a Sylow $3$-subgroup of $\PSU_3(53)$ and has an abelian subgroup $A$ of index $3$. Again,  we indicate in the Example column how these fusion systems correspond to the exotic examples  discovered in \cite[Table 6]{DiazRuizViruel2007}. As in Section \ref{3^5.26}, the fusion system $\F(3^7,386,4 )$ is exotic and so incorrectly labelled in \cite[Table 6]{DiazRuizViruel2007}.

\begin{table}[H]
 \caption{Saturated fusion systems on a Sylow $3$-subgroup of $\PSU_3(53)$ }
\label{3^7.386t}
\renewcommand{\arraystretch}{1.4}
\begin{tabular}{|c|c|c|c|c|c|c|c|}
\hline
$\F$      & $\Out_\F(A)$                 & $\Out_\F(V_0)$ & $\Out_\F(E_0)$ & $\Out_\F(E_1)$  & $\Out_\F(E_2)$              & $\Out_\F(S)$            & Example \\ \hline

$ 1$ & $\GL_2(3)$ & $-$ &  $\GL_2(3)$                         & $\SL_2(3)$ & $E_1\sim_\F E_2$& $2 \times 2$  & $^2\operatorname{F}_4(512)$       \\

$ 2$ & $\GL_2(3)$ & $\GL_2(3)$ &  $-$                         & $\SL_2(3)$ & $E_1\sim_\F E_2$ & $2 \times 2$    &$\F_{\mathrm{DRV}}(3^7,4)$      \\

$ 3$ & $\GL_2(3)$ & $-$ &  $\GL_2(3)$                         & $-$ & $-$ &$2 \times 2$ & $\F_{\mathrm{DRV}}(3^7,2)$          \\

$ 4$ & $-$ & $\GL_2(3)$ &  $-$                         & $\SL_2(3)$ & $E_1\sim_\F E_2$ & $2 \times 2$          & $-$   \\

$ 5$ & $\GL_2(3)$ & $-$ &  $-$                         & $\SL_2(3)$ & $E_1\sim_\F E_2$ & $2 \times 2$          & $ \F_{\mathrm{DRV}}(3^7,1) $\\

$ 6$ & $\GL_2(3)$ & $\GL_2(3)$ &  $-$                         & $-$ & $-$ & $2 \times 2$          &  $\F_{\mathrm{DRV}}(3^7,3)$ \\

$ 7$ & $-$ & $\GL_2(3)$ &  $-$                         & $-$ & $-$ & $2 \times 2$          &  $\mathrm{PGU}_3(53).2 $\\

\hline
\end{tabular}
\end{table}
\subsubsection{$\sgb(3^7,2007)$}\label{3^7.2007}
This group is isomorphic to a Sylow $3$-subgroup $S$ of both the Suzuki and Lyons sporadic simple groups.  This group has a unique elementary abelian subgroup of order $3^5$ and it is therefore characteristic in $S$. Taking $R =\langle x\mid x \in S \setminus A, x^3=1\rangle$, we find $R$ has index $3$ in $S$ and is also characteristic.  In the fusion systems, the essential subgroups are $A$ and $R$.

\begin{table}[H]
\caption{Saturated fusion systems on a Sylow $3$-subgroup of $\Suz$ }
\label{3^7.2007t}
\renewcommand{\arraystretch}{1.4}
\begin{tabular}{|c|c|c|c|c|}
\hline
$\F$      & $\Out_\F(A)$                 & $\Out_\F(R)$          & $\Out_\F(S)$            & Example \\ \hline

$ 1$ & $\M_{11}$ &   $2^{1+4}_-.\Sym(3)$& $\SDih(16)$  & $\Suz$       \\

$ 2$ & $2\times \M_{11}$ & $(\mathrm Q_8\circ 2^{\;.}\Alt(5)).2$ &  $2 \times \SDih(16)$&$\Ly$\\

       $ 3$ & $2\times \M_{11}$ & $2^{1+4}_-.\Dih(12)$&  $2\times \SDih(16)$&$\Aut(\Suz)$\\
       \hline
\end{tabular}
\end{table}

\subsubsection{$\sgb(3^7,8705)$}\label{3^7.8705}

This group is isomorphic to $S_1 \times S_2$ where $S_1 \in \Syl_3(\PSL_3(3))$ and $S_2 \in \Syl_3(\PSp_4(3))$. As $S_1$ supports $4$ fusion systems with $O_3(\F)=1$ and $O^3(\F)=\F$, and $S_2$ supports $6$ (see Table \ref{3^4.7t}) we obtain $24$ fusion systems which are direct products $\F_1 \times \F_2$ of fusion systems on $S_1$ and $S_2$. There are $2$ fusion systems $\F$ on $S_1$ for which $O^{3'}(\F)$ has index $2$ and $3$ such systems on $S_2$. Hence there are $6$ ``diagonal" systems of shape $(O^{3'}(\F_1) \times O^{3'}(\F_2)):2$. This accounts for all the systems and we do not list them individually.

\subsubsection{$\sgb(3^7,8707)$}\label{3^7.8707} This group is  $S= S_1 \times S_2$ with $S_1=3^{1+2}_+ $ a Sylow $3$-subgroup of $\PSL_3(3)$ and $S_2= B(3,4;0,1,0,0) =B(3,4;0,0,2,0)$.  As in \ref{3^7.8705}, $S_1$ supports $4$ fusion systems $\F$ of which $2$ contain $O^{3'}(\F)$ with index $2$ and $S_2$ supports $2$ fusion systems, only one of which has this property. From this we construct $10$ fusion systems of the form $\F_1 \times \F_2$ and $(O^{3'}(\F_1) \times O^{3'}(\F_2)):2$ in which $\F_i$ is a fusion system on $S_i$. Again they are not individually listed.

\subsubsection{$\sgb(3^7,8709)$}\label{3^7.8709}  This group is  $S= S_1 \times S_2$ with $S_1=3^{1+2}_+$ and $S_2=B(3,4;0,0,0,0) $. As above there are $4$ fusion systems $\F$ on $S_1$ of which $2$ contain $O^{3'}(\F)$ with index $2$ and $7$ fusion systems on $S_2$, $3$ of which have this property ($\F(3^4,9,4)$ is simple). This accounts for all $34 = 4\cdot 7 + 2\cdot 3$ fusion systems and we do not list them individually.

\subsubsection{$\sgb(3^7,8713)$}\label{3^7.8713}
This group is isomorphic to a Sylow $3$-subgroup $S$ of $\mathrm{P\Gamma L}_6(4)$. We let $A$ be the unique  elementary abelian subgroup of index $9$.  We have $S$ is the image in $\mathrm{P\Gamma L}_6(4)$ of a Sylow $3$-subgroup $T$ of $\GL_3(4) \times \GL_3(4)$.  Now $A$ is the image of the unique subgroup of $T$ which is elementary abelian of order $3^6$ and $R$ is the image of $3^{1+2}_+$ times the second factor. Notice that $R$ is not normal in the normalizer of $T$. There is just one fusion system.

\begin{table}[H]
\caption{Saturated fusion systems on a Sylow $3$-subgroup of $\PGL_6(4)$ }
\label{3^7.8713t}
\renewcommand{\arraystretch}{1.4}
\begin{tabular}{|c|c|c|c|c|}
\hline
$\F$      & $\Out_\F(A)$                 & $\Out_\F(R)$          & $\Out_\F(S)$            & Example \\ \hline

$ 1$ & $2\times \Sym(6)$& $2\times \GL_2(3)$&$2\times \SDih(16)$&$\mathrm{P\Gamma L}_6(4)$\\

       \hline
\end{tabular}
\end{table}

\subsubsection{$\sgb(3^7,9035)$}\label{3^7.9035}

This group is isomorphic to a Sylow $3$-subgroup of $\Co_3$. It has a unique subgroup $A$ which is elementary abelian of index $9$. We infer from the structure of $\Co_3$ that $S$ has a subgroup $R$ of index $3$ which is extraspecial. Notice that as $|Z_2(S)|=3^3$, $R$ is the unique extraspecial subgroup of index $3$.
\begin{table}[H]
\caption{Saturated fusion systems on a Sylow $3$-subgroup of $\Co_3$ }
\label{3^7.2007t}
\renewcommand{\arraystretch}{1.4}
\begin{tabular}{|c|c|c|c|c|}
\hline
$\F$      & $\Out_\F(A)$                 & $\Out_\F(R)$          & $\Out_\F(S)$            & Example\\ \hline

$ 1$ & $2\times \M_{11} $& $4\circ 2^{\;.}\Sym(6)$&$2\times \Dih(8)$&$\Co_3$\\

       \hline
\end{tabular}
\end{table}

\newpage
\section{Finding overgroups with a given Sylow $p$-subgroup. }

A key part to our algorithm requires that for each proto-essential subgroup $E$ of $S$ we determine all the candidates for $\Aut_\F(E)$. To determine these possibilities we have as input $\Aut_S(E)$ as a subgroup of $\Aut(E)$  and we know  $\Aut_S(E)$  should be a Sylow $p$-subgroup of $\Aut_\F(E)$.  In this short section we present an algorithm which, given a prime $p$, group $G$ and  $p$-subgroup $T$ determines up to $G$-conjugacy all the subgroups of $G$ which contain $T$ as a Sylow $p$-subgroup.

Suppose that $H \le G$ and $T \in \Syl_p(H)$. Then by definition $O^{p'}(H) = \langle T^H\rangle$ and $H= N_H(T)O^{p'}(H)$ by the Frattini Argument. We first find the candidates for overgroups $H$ with $H= O^{p'}(H)$. Denote the smallest subnormal subgroup of $G$ which contains $T$ by $\mathrm{sn}(G,T)$. Thus  $\mathrm{sn}(G,T)$ is the intersection of all the subnormal subgroups of $G$ which contain $T$. It can be calculated by setting $G_0=G$ and, for $i \ge 1$, $G_i= \langle T^{G_{i-1}}\rangle$.  Then $G_\infty = \mathrm{sn}(G,T)$.

\begin{Lem}
Suppose that $H= O^{p'}(H)$ has $T \in \Syl_p(H)$. Then $H\le \mathrm{sn}(G,T)$.
\end{Lem}

\begin{proof} Suppose that $K$ is a subnormal subgroup of $G$ which contains $T$. Then $H \cap K$ is subnormal in $H$ and contains $T$.  Hence $K \ge O^{p'}(H)=H$.  It follows that $$H \le  \bigcap_{{T \le K}\atop{K \text{ subnormal in } G} }\hspace{-3mm}K \hspace{3mm}=\mathrm{sn}(G,T) .$$
\end{proof}

Our algorithm proceeds as follows. Suppose that $T \le X$.

    \begin{itemize}
    \item[(1)] Replace $X$ by $X_\infty=\mathrm{sn}(X,T)$ and set $\mathcal X= \emptyset$.
    \item[(2)] Determine the maximal subgroups $M_1, \dots, M_k$ of $X_\infty$ up to $X_\infty$-conjugacy.
    \item[(3)] For each maximal subgroup $M_i$, determine up to $M_i$-conjugacy the subgroups of $M_i$ which are $X_\infty$-conjugate to $T$. Create conjugates $M_{i_j}$ of $M_i$ which contain $T$ from these conjugacy classes.  For each $M_{i_j}$-conjugacy class add the subnormal closure of $T$ in $M_{i_j}$ to  $\mathcal X$.
    \item[(4)] For each new  member $X \in \mathcal X$ go to step (2).  Continue this until the only new member is $T$.
    \item[(5)] For each element of $\mathcal X$ determine whether or not $T$ is a Sylow $p$-subgroup; if it is not, remove it from $\mathcal X$.
    \item[(6)] For each $H_0 \in \mathcal X$, set $N=N_{N_G(T)}(H_0)$ and determine the subgroups $M$ of  $N$ which contain $T $ as a Sylow $p$-subgroup up to $N$-conjugacy.  Make the subgroup $H_0M$ and add it to the list of subgroups with $T$ as a Sylow $p$-subgroup.
    \end{itemize}

We make some remarks concerning the implementation of this algorithm in our programs. By Lemma \ref{Strongly p-embedded Sylows}, if $\Out_S(E)$ is not cyclic or quaternion then $\Out_\F(E)$ is not soluble and the command \texttt{Subgroups(G:NonSolvable:=true)} can be implemented as part of a timely alternative to the algorithm provided above. When $\Out_S(E)$ is cyclic our algorithm is preferable.

\newpage
\section{An example}\label{s:appendixc}
 Take $G= \SL_4(2)$ and $S\in \Syl_2(G)$. The initial calculation before the calculation of potential Borel groups provides $3$ $\Aut(S)$-conjugacy classes of proto-essential subgroups. They have orders $32$, $32$ and $16$.
 The total number of proto-essential subgroups is $5$. The program calculates that there are two potential Borel groups, $B=S$ and $B= S:C$ where $C$ is cyclic of order $3$.  There is an elementary abelian subgroup $E$ of order $16$ which is normal of index $4$ in $S$. It follows that $\Out_S(E) \cong C_2 \times C_2$ and that a candidate for $\Aut_\F(E)$ contains a subgroup $\SL_2(4)$.  The program recognizes this cannot be an essential subgroup if $B= S$, rather in this case we find a unique saturated fusion system with $3$ essential subgroups isomorphic to $\F_S(\Alt(8))$ . When $B= S:C$, the output is $\F_S(\PSU_4(2))$.

\begin{verbatim}
time A:=AllFusionSystems(Sylow(Alt(8),2):Printing:=true);
The group has 40 centric subgroups
The set ProtoEssentialAutClasses has 3 elements
This group has  2  Borel groups
**********************************************
 Borel 1 of 2 [ <2, 6> ]
**********************************************
There are 5 proto-essential subgroups before the extension test
They have orders 32 32 32 32 16
4 proto-essentials which pass both the  strongly p-embedded
and extension tests
The number of forbidden pairs of essential subgroups is  0
Checking 8 automizer sequences with 4 essentials of orders: 32 32 32 32
Checking 4 automizer sequences with 3 essentials of orders: 32 32 32
Executed saturation test: result is true
Checking 8 automizer sequences with 3 essentials of orders: 32 32 32
Checking 4 automizer sequences with 2 essentials of orders: 32 32
Checking 2 automizer sequences with 2 essentials of orders: 32 32
**********************************************
 Borel 2 of 2 [ <2, 6>, <3, 1> ]
**********************************************
There are 3 proto-essential subgroups before the extension test
They have orders 32 32 16
3 proto-essentials which pass both the  strongly p-embedded
and extension tests
The number of forbidden pairs of essential subgroups is  0
Checking 4 automizer sequences with 3 essentials of orders: 32 32 16
Checking 2 automizer sequences with 2 essentials of orders: 32 16
Executed saturation test: result is true
Checking 2 automizer sequences with 2 essentials of orders: 32 32
Checking 4 automizer sequences with 2 essentials of orders: 32 16
Checking 2 automizer sequences with 1 essentials of orders: 32
>A;
[
    Fusion System with 3 F-classes of essential subgroups
    They have orders: [ 32, 32, 32 ] Out_F(E)  have orders: [ 6, 6, 6 ]
    Out_F(S) has order  1,
    Fusion System with 2 F-classes of essential subgroups
    They have orders: [ 32, 16 ] Out_F(E)  have orders: [ 18, 60 ]
    Out_F(S) has order  3
]
Time: 7.130
 > G:= GroupFusionSystem(Alt(8),2); IsIsomorphic(G,A[1]);
true
 > G:= GroupFusionSystem(PSU(4,2),2); IsIsomorphic(G,A[2]);
true
 > IsIsomorphic(A[1],A[2]);
false
 \end{verbatim}

\clearpage

\bibliographystyle{amsalpha}
\bibliography{mybib}

\end{document}